\documentclass{amsart}
\usepackage{amsfonts}%
\usepackage{amssymb}
\usepackage{mathrsfs}
\allowdisplaybreaks
  \usepackage{graphics}
  \usepackage{epsfig}
\usepackage{graphicx}  \usepackage{epstopdf}
 \usepackage[colorlinks=true]{hyperref}
\hypersetup{urlcolor=blue, citecolor=red}

\newcommand{\dv}{\mathrm{div}\,}
\newcommand{\cl}{\mathrm{curl}\,}

\newtheorem{Theorem}{Theorem}[section]
\newtheorem{Lemma}{Lemma}[section]
\newtheorem{Proposition}{Proposition}[section]

\theoremstyle{definition}
\newtheorem{Def}{Definition}[section]

\theoremstyle{remark}
\newtheorem{Remark}{Remark}[section]

\title[Well-posedness and low Mach number limit of the  MHD equations]
 {Local well-posedness and low Mach number limit of
 the  compressible  magnetohydrodynamic equations in critical spaces}

\author[ Fucai Li, Yanmin Mu and Dehua Wang]{}



 \email{fli@nju.edu.cn}
  \email{yminmu@126.com}
 \email{dwang@math.pitt.edu}

\thanks{Y. Mu is the corresponding author}

\begin{document}
\maketitle

\centerline{\scshape Fucai Li}
\medskip
{\footnotesize
 \centerline{Department of Mathematics, Nanjing University}
  \centerline{Nanjing
 210093,  China}

\medskip

\centerline{\scshape Yanmin Mu}
\medskip
{\footnotesize
 \centerline{School  of Applied Mathematics, Nanjing University of Finance \& Economics}
  \centerline{Nanjing
 210046,  China}

\medskip

\centerline{\scshape Dehua Wang}
\medskip
{\footnotesize
 \centerline{Department of Mathematics,   University of Pittsburgh}
  \centerline{ Pittsburgh, PA 15260}

\bigskip


\begin{abstract}
The local well-posedness and low Mach number limit are considered for
 the  multi-dimensional isentropic  compressible  viscous magnetohydrodynamic equations in critical spaces.
First  the local well-posedness of solution  to the
viscous magnetohydrodynamic equations with large initial data is established.
 Then the low Mach number limit is studied for general large data and it is proved
 that  the solution of  the compressible  magnetohydrodynamic equations
   converges to
 that of the incompressible   magnetohydrodynamic equations as the Mach number
 tends to zero. Moreover,  the   convergence rates are obtained.
  \end{abstract}


\maketitle



\section{Introduction}
In this paper we consider  the local well-posedness and low Mach number limit to
 the following isentropic  compressible
magnetohydrodynamic (MHD) equations  in critical spaces (see \cite{LL,KL,PD}):
\begin{align}
&\partial_t\rho+\dv(\rho u)=0,\label{Aa1}\\
&\partial_t(\rho u)+\dv(\rho u\otimes u)+\nabla P(\rho)\nonumber\\
& \qquad\quad=
H\cdot\nabla H-\frac{1}{2}\nabla(|H|^2)+\mu\Delta u+(\mu+\lambda)\nabla\dv u,\label{Aa2}\\
&\partial_tH-\cl ( u \times H)=-\cl (\nu\,\cl H),\quad\dv H=0,\label{Aa3}\\
&(\rho,u,H)|_{t=0}=(\rho_{0},u_{0},H_{0})(x), \quad x\in \mathbb{R}^d.\label{Aa4}
\end{align}
Here $\rho$ denotes the density of the fluid, $u=(u^{(1)},\dots,u^{(d)})\in\mathbb{R}^d\, (d=2,3)$   is the fluid velocity field,  $H=(H^{(1)}, \dots,H^{(d)})\in\mathbb{R}^d$
is the magnetic field,  and $P$  is the pressure function satisfying $P^{'}(\rho)>0$.
The constants $\mu>0$ and $\lambda$ denotes the shear and bulk viscosity coefficients of the flow, respectively, satisfying  $ 2\mu+\lambda>0$.
The constant $\nu>0 $ is the magnetic diffusivity acting as a magnetic diffusion coefficient of the magnetic field.

The system \eqref{Aa1}-\eqref{Aa3} can be derived from the isentropic Navier-Stokes-Maxwell system by taking the zero dielectric constant limit \cite{JLi}.
 Recently, many results on the system \eqref{Aa1}-\eqref{Aa3} were obtained.
 Li and Yu \cite{LY} obtained the optimal decay rate of smooth solution when the initial data is a small perturbation of some give constant
 state. Suen and Hoff \cite{SH} established the global weak solutions when the initial energy is small. Later, this result was extended
 to the case when the initial data may contain  large oscillations or vacuum \cite{LXZ,LYZ}. Hu and Wang \cite{HW1} obtained the global existence
 and large-time behavior of general weak solution with finite energy in the sense of \cite{f2,f1,Lio}.
 Li, Su and Wang \cite{LSW} obtained the local strong solution to \eqref{Aa1}-\eqref{Aa3} with large initial data.
 Suen \cite{Su} as well as Xu and Zhang \cite{XZ} established some blow-up criteria for \eqref{Aa1}-\eqref{Aa3}.
The low Mach number limit of the system    \eqref{Aa1}-\eqref{Aa3} has also been studied recently.
  Hu and Wang \cite{HW} proved the convergence of the
weak solutions of the compressible MHD equations   to a weak solution of the viscous incompressible  MHD equations.
Jiang, Ju and Li obtained the convergence of the weak  solutions of the compressible MHD equations
 to the strong solution of the ideal incompressible MHD equations   in
  the whole space \cite{JJL1} or to the viscous incompressible  MHD equations  in torus \cite{JJL2} for general
  initial data. Feireisl, Novotny,  and Sun \cite{FNS} extended and improved the results in \cite{JJL1}
  to the unbounded domain case.  Li \cite{Li} studied the invisid, incompressible limit of the viscous isentropic
compressible MHD  equations for local smooth solutions  with well-prepared initial data.  Dou,  Jiang, and Ju \cite{DJJ}
studied the low Mach number limit for the compressible magnetohydrodynamic equations in
a bounded domain with perfectly conducting boundary.
See the  recent papers \cite{DJJ,FNS,JJL1,JJL2,Li}  and the references therein for more discussions of   other related results.

We point out that all of the above results were carried out in the framework of Sobolev spaces.
Obviously, up to a change of the pressure function $P$ into $l^2P$ in the system \eqref{Aa1}-\eqref{Aa3}, it
  is invariant under the scaling:  
\begin{align}\label{scaling}
 \rho^\epsilon(t,x), u^\epsilon(t,x), H^\epsilon(t,x))\rightarrow  \rho^\epsilon(l^2t,lx),  l u^\epsilon(l^2t,lx), l^2H^\epsilon(l^2t,lx).
\end{align}
Thus it is natural to study the system \eqref{Aa1}-\eqref{Aa3} in critical spaces. \big(A function space $\mathcal{E}\in
\mathcal{S}'(\mathbb{R}^+\times \mathbb{R}^d; \mathbb{R}\times \mathbb{R}^d\times \mathbb{R}^d)$ is called
a critical space for the system \eqref{Aa1}-\eqref{Aa3} if the associated norm is invariant under the
transformation of \eqref{scaling} (up to a constant independent of $l$)\big).
In \cite{Ha}, Hao obtained the global existence of solution to the  system \eqref{Aa1}-\eqref{Aa3} in
the critical space when the initial data is a small  perturbation of some given constant
state. In \cite{Mu}, the second author low Mach number limit of
 the  system \eqref{Aa1}-\eqref{Aa3} for small initial data in   Besov space.
 In \cite{BY}, Bian and Yuan studied the inviscid version of \eqref{Aa1}-\eqref{Aa3} in the super critical
Besov spaces.

The purpose of this  paper is to study the local well-posedness and low Mach number limit of
 the  system \eqref{Aa1}-\eqref{Aa3} with large initial data  in critical Besov spaces in the whole space $\mathbb{R}^d$.
We add the the following condition to  the system  \eqref{Aa1}-\eqref{Aa3} in the far field:
\begin{equation}\label{far}
\rho \rightarrow 1,  \   u \rightarrow 0, \  H \rightarrow 0 \ \ \text{as}  \ \  |x|\rightarrow \infty.
\end{equation}
Denoting $$a:=\rho-1,$$
 introducing the viscosity operator $$\mathcal{A}:=\mu\Delta+(\lambda+\mu)\nabla\dv\!,$$
and using the identities:
\begin{gather*}
\cl\cl H= \nabla \dv H-\Delta H,\\
\cl( u \times H) = u(\dv H)- H(\dv u) +H\cdot \nabla u-u\cdot \nabla H,
\end{gather*}
 we can
 rewrite the Cauchy problem  \eqref{Aa1}-\eqref{Aa4} as the following:
\begin{align}
&\partial_t a+ u\cdot\nabla a=-(1+a)\dv u,\label{Aa5}\\
&\partial_t u+ u\cdot\nabla u-\frac{1}{1+a}\mathcal{A}u+\nabla G(a)=
\frac{1}{1+a}\Big(H\cdot\nabla H-\frac{1}{2}\nabla(|H|^2)\Big),\label{Aa6}\\
&\partial_tH+u\cdot\nabla H-H\cdot\nabla u-\nu\Delta H=-(\dv u)H,\quad\dv H=0,\label{Aa7}\\
&(a,u,H)|_{t=0}=(a_{0},u_{0},H_{0})(x), \quad x\in \mathbb{R}^d, \label{Aa8}
\end{align}
where $$\nabla G(a):=\frac{1}{1+a}\nabla P(1+a).$$

To state our results, we introduce the following function spaces:
\begin{align*}
  E^{\alpha}_{T}:=\, & \widetilde{C}_{T}( \dot{B}^{\frac{d}{2}}_{2,1}\cap \dot{B}^{\frac{d}{2}+\alpha}_{2,1})
\times\big(\widetilde{C}_{T}( \dot{B}^{\frac{d}{2}-1}_{2,1}\cap \dot{B}^{\frac{d}{2}-1+\alpha}_{2,1})
\cap L^{1}_T(\dot{B}^{\frac{d}{2}+1}_{2,1}\cap \dot{B}^{\frac{d}{2}+1+\alpha}_{2,1})\big)^{d}\nonumber\\
 &\times\big(\widetilde{C}_{T}( \dot{B}^{\frac{d}{2}-1}_{2,1}\cap \dot{B}^{\frac{d}{2}-1+\alpha}_{2,1})
 \cap L^{1}_T(\dot{B}^{\frac{d}{2}+1}_{2,1}\cap \dot{B}^{\frac{d}{2}+1+\alpha}_{2,1})\big)^{d},\nonumber
 \\
F^{\frac{d}{2}+\beta}_{T}:=\,& \mathcal{C}_{b}([0,T];\dot{B}^{\frac{d}{2}+\beta-1}_{2,1})^{2d}\cap L^{1}_T(\dot{B}^{\frac{d}{2}+\beta+1}_{2,1})^{2d}
 \end{align*}
 with
 \begin{align*}
& \widetilde{\mathcal{C}}_{T}(\dot{B}^{s}_{p,1}):=\mathcal{C}([0,T];\dot{B}^{s}_{p,1})\cap
\widetilde{L}^{\infty}_T(\dot{B}^{s}_{p,1}),
 \end{align*}
where $\dot{B}^{s}_{p,1}$ denotes the homogeneous Besov space. We shall explain these notations in detail in  Appendix \ref{AppA}.

Our first result of this paper reads as follows.

\begin{Theorem}\label{ThA}
Assume that the initial data $(a_{0},u_{0},H_{0})$ satisfy $ \dv H_0=0$ and
$$
a_{0}\in \dot{B}^{\frac{d}{2}}_{2,1}\cap\dot{B}^{\frac{d}{2}+\alpha}_{2,1},\quad
u_{0}\in \dot{B}^{\frac{d}{2}-1}_{2,1}\cap\dot{B}^{\frac{d}{2}-1+\alpha}_{2,1},\quad
H_{0}\in \dot{B}^{\frac{d}{2}}_{2,1}\cap\dot{B}^{\frac{d}{2}+\alpha}_{2,1},\quad
$$
 for some $\alpha\in (0,1]$.  If, in addition, $\displaystyle \inf_{x\in \mathbb{R}^d}a_{0}(x)>-1,$  then there exists a   $T>0$ such that
 the problem \eqref{Aa5}-\eqref{Aa8} has a unique solution $(a,u,H)$ on $[0,T]\times \mathbb{R}^{d}$ which belongs to
 $ E^{\alpha}_{T}$ and satisfies $\displaystyle{\inf_{(t,x)\in [0,T]\times\mathbb{R}^{d}}}a(t,x)>-1$.
\end{Theorem}

\begin{Remark}\label{RemA}
   Theorem \ref{ThA} still holds for $\alpha=0$. Here we assume additional  regularity  on the initial
   data to obtain more regular solution, which is needed in the study of the low Mach number limit to the system \eqref{Aa1}-\eqref{Aa3} below.
For the case $\alpha=0$, the proof of the uniqueness of   solution in dimension two  needs additional arguments, and we refer the readers to \cite{Db,De}
for  the corresponding discussions on the isentropic Navier-Stokes equations.

\end{Remark}

Denote by $\epsilon$ the (scaled) Mach number.  Introducing the scaling:
\begin{equation*}
 \rho(x,t)=\rho^\epsilon (x,\epsilon t),\quad u(x,t)=\epsilon u^\epsilon(x,\epsilon t), \quad  H(x,t)=\epsilon H^\epsilon(x,\epsilon t),
 \end{equation*}
 and assuming that the viscosity coefficients ${\mu}$,
${\xi}$, and ${\nu}$
  are   small constants and scaled as:
\begin{align}\nonumber
  {\mu}=\epsilon \mu^\epsilon, \quad {\lambda}=\epsilon \lambda^\epsilon,
  \quad {\nu}=\epsilon \nu^\epsilon,
\end{align}
then   we can rewrite the problem \eqref{Aa1}-\eqref{Aa4} as the following:
\begin{align}
&\partial_t\rho^{\epsilon}+\dv(\rho^{\epsilon} u^{\epsilon})=0,\label{Aa9}\\
&\partial_t(\rho^{\epsilon} u^{\epsilon})+\dv(\rho^{\epsilon} u^{\epsilon}\otimes u^{\epsilon})+\frac{\nabla P^{\epsilon}}{\epsilon^2}\nonumber\\
& \qquad  \quad =
H^{\epsilon}\cdot\nabla H^{\epsilon}-\frac{1}{2}\nabla(|H^{\epsilon}|^2)
+\mu^{\epsilon}\Delta u^{\epsilon}+(\mu^{\epsilon}+\lambda^{\epsilon})\nabla\dv u^{\epsilon},\label{Aa10}\\
&\partial_tH^{\epsilon}+(\dv u^{\epsilon})H^{\epsilon}+u^{\epsilon}\cdot\nabla H^{\epsilon}-H^{\epsilon}\cdot\nabla u^{\epsilon}
  =\nu^{\epsilon}\Delta H^{\epsilon},\quad\dv H^{\epsilon}=0,\label{Aa11}\\
&(\rho^{\epsilon},u^{\epsilon},H^{\epsilon})|_{t=0}=(\rho^{\epsilon}_{0},u^{\epsilon}_{0},H^{\epsilon}_{0})(x),  \quad x\in \mathbb{R}^d,\label{Aa12}
\end{align}
where $P^{\epsilon}:=P(\rho^\epsilon)$ stands for the pressure.
 For the simplicity of notations and presentation,  we shall assume  that
$\mu^\epsilon $, $\lambda^\epsilon $, and
$\nu^\epsilon $ are constants, independent of $\epsilon$, and still denote them as  ${\mu}$,
${\lambda}$, and ${\nu}$ with an abuse of notations.

  Formally, if we let $\epsilon$ go to zero, then we  have $\nabla P^{\epsilon}\rightarrow 0$.
  Thus, if $P'(\cdot)$ does not vanish, the limit density has to be a constant.
Denote by  $(v,B)$   the limit of $(u^{\epsilon},H^{\epsilon})$. Taking  the limit in the mass equation \eqref{Aa9}
implies that the limit $v$ is divergence-free. Passing to the limit in the equations  \eqref{Aa10} and \eqref{Aa11},
we conclude that  $(v,B)$ must satisfy the following incompressible MHD equations:
\begin{align}
&\partial_t v+v\cdot\nabla v-\mu\Delta v+\nabla\pi=
B\cdot\nabla B-\frac{1}{2}\nabla(|B|^2),\label{Aa14}\\
&\partial_tB+v\cdot\nabla B-B\cdot\nabla v=\nu\Delta B,\label{Aa15}\\
&\dv v=0, \quad \dv B=0,\label{Aa13}\\
&(v,B)|_{t=0}=(v_{0},B_{0})(x),  \quad x\in \mathbb{R}^d. \label{Aa16}
\end{align}

As mentioned before, the rigorous derivation of the above heuristic process was  proved recently in \cite{Li,HW,JJL1,JJL2}
 in the framework of Sobolev spaces.
Here we  want to   justify the above formal process in critical Besov spaces.
More precisely, we shall establish the convergence of the system \eqref{Aa9}-\eqref{Aa11} to the system \eqref{Aa14}-\eqref{Aa13}
based on the results  obtained in Theorem \ref{ThA}.  We shall focus on the case of ill-prepared data
where the  acoustic waves caused by the oscillations must  be considered.

Writing $\rho^{\epsilon}=1+\epsilon b^{\epsilon}$, it is easy to check  that   $(b^{\epsilon},u^{\epsilon},H^{\epsilon})$
satisfies
\begin{align}
&\partial_t b^{\epsilon}+\frac{\dv u^{\epsilon}}{\epsilon}=-\dv(b^{\epsilon}u^{\epsilon}),\label{Aa17}\\
&\partial_t u^{\epsilon}+u^{\epsilon}\cdot\nabla u^{\epsilon}-\frac{\mathcal{A}u^{\epsilon}}{1+\epsilon b^{\epsilon}}
+\frac{ P^{'}(1+\epsilon b^{\epsilon})\nabla b^{\epsilon}}{(1+\epsilon b^{\epsilon})\epsilon}\nonumber\\
& \qquad \qquad \qquad \qquad \qquad= \frac{1}{1+\epsilon b^{\epsilon}}\big(H^{\epsilon}\cdot\nabla H^{\epsilon}-\frac{1}{2}\nabla(|H^{\epsilon}|^2)\big) ,\label{Aa18}\\
&\partial_tH^{\epsilon}+(\dv u^{\epsilon})H^{\epsilon}+u^{\epsilon}\cdot\nabla H^{\epsilon}-H^{\epsilon}\cdot\nabla u^{\epsilon}
=\nu\Delta H^{\epsilon},\quad\dv H^{\epsilon}=0,\label{Aa19}\\
&(b^{\epsilon},u^{\epsilon},H^{\epsilon})|_{t=0}=(b^{\epsilon}_{0},u^{\epsilon}_{0},H^{\epsilon}_{0})(x),  \quad x\in \mathbb{R}^d. \label{Aa20}
\end{align}
  For the sake of simplicity, we shall also assume that the initial data $(b_{0}^{\epsilon},u^{\epsilon}_{0},H^{\epsilon}_{0})$
does not depend on $\epsilon$ and will be   denoted by $(b_{0},u_{0},H_{0})$. The general
case of  $ (b_{0}^{\epsilon},u^{\epsilon}_{0},H^{\epsilon}_{0}) \rightarrow (b_{0},u_{0},H_{0})$  as $\epsilon \rightarrow 0$
in some Besov spaces can be treated similarly by a slight modification of the arguments presented here.

Denoting  $\mathcal{P}$ the Leray projector on solenoidal vector fields defined by $\mathcal{P}:=I -\mathcal{P}^{\bot}$ with
$\mathcal{P}^{\bot}:=\Delta^{-1}\nabla \dv$, and introducing the following functional space:
\begin{align}
E^{\frac{d}{2}+\beta}_{\epsilon,T}:=\,& \mathcal{C}_{b}([0,T];\tilde{B}^{\frac{d}{2}+\beta,\infty}_{\epsilon})
\cap L^{1}_T(\tilde{B}^{\frac{d}{2}+\beta,1}_{\epsilon})\nonumber\\
            &\times\big(\mathcal{C}_{b}([0,T];\dot{B}^{\frac{d}{2}-1+\beta}_{2,1})\cap L^{1}_T(\dot{B}^{\frac{d}{2}+1+\beta}_{2,1})\big)^{d}\nonumber\\
            &\times\big(\mathcal{C}_{b}([0,T];\dot{B}^{\frac{d}{2}-1+\beta}_{2,1})\cap L^{1}_T(\dot{B}^{\frac{d}{2}+1+\beta}_{2,1})\big)^{d},\nonumber
            \end{align}
 with the norm
 \begin{align}
 \|(\rho,u,H)\|_{E^{\frac{d}{2}+\beta}_{\epsilon,T}}:=\,&\|\rho\|_{L^{\infty}_T(\tilde{B}^{\frac{d}{2}+\beta,\infty}_{\epsilon})
 \cap L^{1}_T(\tilde{B}^{\frac{d}{2}+\beta,1}_{\epsilon})} +\|u\|_{L^{\infty}_T(\dot{B}^{\frac{d}{2}-1+\beta}_{2,1})
\cap L^{1}_T(\dot{B}^{\frac{d}{2}+1+\beta}_{2,1}))}\nonumber\\
&+\|H\|_{L^{\infty}_T(\dot{B}^{\frac{d}{2}-1+\beta}_{2,1})
\cap L^{1}_T(\dot{B}^{\frac{d}{2}+1+\beta}_{2,1})},\nonumber
\end{align}
our second result of the paper can be stated as follows.

\begin{Theorem}\label{ThB} Let $T_{0}\in (0,\infty].$
Assume that the initial data $(b_{0},u_{0},H_{0})$ satisfy $\dv H_0=0$ and
$$
b_{0}\in \dot{B}^{\frac{d}{2}-1}_{2,1}\cap\dot{B}^{\frac{d}{2}+\alpha}_{2,1} ,
\quad (u_{0},H_{0})\in (\dot{B}^{\frac{d}{2}-1}_{2,1}\cap\dot{B}^{\frac{d}{2}-1+\alpha}_{2,1})^{d+d}
$$
with $\alpha\in(0,1/2)$ if $d=3$ or $\alpha \in (0,1/6]$ if $d= 2$.
 Suppose that the incompressible system \eqref{Aa13}-\eqref{Aa16} with initial
data  $(\mathcal{P}u_{0},H_{0})$ has a solution $(v,B)\in F^{\frac{d}{2}}_{T_{0}}\cap F^{\frac{d}{2}+\alpha}_{T_{0}}$.
Let $V:=\|(v,B)\|_{F^{\frac{d}{2}}_{T_{0}}\cap F^{\frac{d}{2}}_{T_{0}}}$ and
$$
X_{0}:=\|b_{0}\|_{\dot{B}^{\frac{d}{2}-1}_{2,1}\cap\dot{B}^{\frac{d}{2}+\alpha}_{2,1}}+
\|\mathcal{P}^{\bot}u_{0}\|_{\dot{B}^{\frac{d}{2}-1}_{2,1}\cap\dot{B}^{\frac{d}{2}-1+\alpha}_{2,1}}
+\|H_{0}\|_{\dot{B}^{\frac{d}{2}-1}_{2,1}\cap\dot{B}^{\frac{d}{2}-1+\alpha}_{2,1}},
$$
then there exist two positive constants $\epsilon_{0}$ and $C$, depending only on $d,\alpha,\lambda,\mu,\nu, P, V,$ and $X_{0}$,
 such that the following results hold true:
\begin{itemize}

 \item[(i)] For all $\epsilon \in (0,\epsilon_0] $, the problem \eqref{Aa17}-\eqref{Aa20} has a unique global solution $(b^{\epsilon},u^{\epsilon},H^{\epsilon})$
in $E^{\frac{d}{2}}_{\epsilon,T_{0}}\cap E^{\frac{d}{2}+\alpha}_{\epsilon,T_{0}}$ such that
$$\|(b^{\epsilon},u^{\epsilon},H^{\epsilon})\|_{E^{\frac{d}{2}}_{\epsilon,T_{0}}\cap E^{\frac{d}{2}+\alpha}_{\epsilon,T_{0}}}\leq C;$$

\item[(ii)]   $(\mathcal{P}u^{\epsilon},H^{\epsilon})\rightarrow (v, B)$ in $ F^{\frac{d}{2}}_{T_{0}}\cap F^{\frac{d}{2}+\alpha}_{T_{0}}$
 as $\epsilon \rightarrow 0$, and
\begin{align*}
 & \|\mathcal{P}u^{\epsilon}-v\|_{F^{\frac{d}{2}}_{T_{0}}\cap F^{\frac{d}{2}+\alpha}_{T_{0}}}\leq C\epsilon^{\frac{2\alpha}{2+d+2\alpha}},\\
& \|H^{\epsilon}-B\|_{F^{\frac{d}{2}}_{T_{0}}\cap F^{\frac{d}{2}+\alpha}_{T_{0}}}\leq C\epsilon^{\frac{2\alpha}{2+d+2\alpha}};
\end{align*}

\item[(iii)]  $(b^{\epsilon},\mathcal{P}^{\bot}u^{\epsilon})$ tends to $(0,0)$ as  $\epsilon \rightarrow 0$ in the following sense:
\begin{align*}
& \|(b^{\epsilon},\mathcal{P}^{\bot}u^{\epsilon})\|_{L^{p}_{T}(\dot{B}^{\alpha-1+\frac{1}{p}}_{\infty,1})}\leq C\epsilon^{\frac{1}{p}} \ \ \text{if}\ \ d=3
\ \text{and} \ \ 2< p<\infty,\\
   &\|(b^{\epsilon},\mathcal{P}^{\bot}u^{\epsilon})\|_{L^{4}_{T}(\dot{B}^{\alpha-\frac{3}{4}}_{\infty,1})}\leq C\epsilon^{\frac{1}{4}}\ \  \ \, \ \text{if}\ \ d= 2.
\end{align*}
\end{itemize}
\end{Theorem}

\begin{Remark}
The regularity assumption on  the solution of the incompressible system \eqref{Aa13}-\eqref{Aa16} is reasonable. Since  we can not find it in the literature, we
shall present a brief proof  in Proposition \ref{imhd} of the Appendix \ref{lo-imhd}.
\end{Remark}

\begin{Remark}
 In Theorem \ref{ThB}, we need the additional constraint  on the index $\alpha$ since it  provides some decay in $\epsilon$,
 which is used in many places of  the proof.
\end{Remark}

\begin{Remark}
For the case $d = 2$, one can choose $T_0 =+\infty$ in Theorem \ref{ThB}.
\end{Remark}

\begin{Remark}
Due to the absence  of disperse effects on the oscillation equations,  it is more complicated and difficult to study
the low Mach number limit of the system \eqref{Aa9}-\eqref{Aa11} for the period case in  Besov spaces. We shall report this result in a forthcoming paper.
\end{Remark}

We now recall a few closely related results on the isentropic Navier-Stokes equations
(i.e., $H=0$ in the system \eqref{Aa1}-\eqref{Aa3}).
In  Danchin \cite{Daa}    the global well-posedness of
 isentropic Navier-Stokes equations in the critical Besov space was first obtained  when the initial data is a small perturbation around some given constant state,
 and recently,  the results of \cite{Daa} were extended to more general Besov spaces
 in \cite{CD,CMZ,Has}.
In a series of papers by Danchin \cite{DA,Db,De},   the local well-posedness  of solutions to the  isentropic Navier-Stokes equations
with large initial data was proved. In \cite{DC,D02},  the zero Mach number limit of the isentropic Navier-Stokes equations in the whole space or torus with ill-prepared initial data was studied.

Next we give some comments on the proofs of Theorems \ref{ThA} and \ref{ThB}.  We remark that  when $H=0$  our results   coincide with the results
obtained by Danchin \cite{DC,Db,De} on the isentropic Navier-Stokes equations, hence extend some results in \cite{DC,Db,De}
to the isentropic compressible MHD equations. In our proofs of  Theorems \ref{ThA} and \ref{ThB}, we use some ideas developed in \cite{DC,Db,De}.
Besides the difficulties mentioned in \cite{DC,Db,De}, here the main difficulty is the strong coupling of the velocity and the magnetic field.
We shall deal with them in detail in the proofs of  Theorems \ref{ThA} and \ref{ThB}. More precisely,  in  the proof of Theorem \ref{ThA}, we introduce
a linearized  version of the equation for the magnetic field \eqref{bb2} below and obtain a tame estimate of the solution,   and  the coupling terms of the velocity and the magnetic field
are analyzed in detail
in each step of the proof of Theorem \ref{ThA};  see especially   the proof of Proposition  \ref{PrA} in Section \ref{local} below.
In the proof of Theorem \ref{ThB},  several new systems analogous to the incompressible MHD
equations (see the systems \eqref{eqb3} and \eqref{eqb6} and Propositions \ref{AProp2} and \ref{AProp3} below)
are introduced and studied to establish the estimates on the incompressible part of the original compressible MHD equations,  and also
the coupling terms of the velocity and the magnetic field are analyzed in detail
in each step of the proof of Theorem \ref{ThB}; see Section \ref{lowlimit} below. In particular, the special structure of the isentropic compressible
MHD  system
are fully utilized in our analysis.


Our paper is organized as follows.
In Section \ref{local},  we establish the local existence and uniqueness of solution to the
problem \eqref{Aa5}-\eqref{Aa8}.
In Section \ref{lowlimit}, we discuss the low Mach limit of the problem \eqref{Aa17}-\eqref{Aa20}.
We close our paper with two appendices. In  Appendix \ref{AppA}, we define some functional spaces (homogeneous and hybrid Besov spaces), recall some basic
tools on paradifferential calculus and state some tame estimates for composition or product.
Finally, in Appendix  \ref{lo-imhd}, we present the  regularity results  on incompressible MHD equations  \eqref{Aa14}-\eqref{Aa13}
and the analogies which are needed in the proof of Theorem \ref{ThB}.

\bigskip

\section{Local Well-posedness of the Compressible MHD Equations}\label{local}

In this section we shall establish the local well-posedness of the compressible MHD equations \eqref{Aa5}-\eqref{Aa7}.
We shall follow and adapt the methods developed by Danchin in \cite{DA,Db,De} (see also \cite{BCD}). We shall focus on the analysis  of  the
coupling terms of the velocity field $u$ and the magnetic field $H$.  We divide this section into four parts.
First, we recall some basic results on the linear transport equation  and prove a result on the linearized magnetic field equation
and a result on the smooth solution of the system \eqref{Aa5}-\eqref{Aa7} which is new in some sense and play an essential role in our proof.
Next, we establish the local existence of the solution. Third, we discuss the uniqueness of the solution.
Finally, we state  a continuation criterion of the solution.


Letting  $I$ be an interval of $\mathbb{R}$ and $X$ be a Banach space, we use the notation
$\mathcal{C}_{b}(I,X)$  to denote the set of bounded and continuous function on $I$ with values in  $X$.
Similarly, $L^{r}(I,X)$ is used to  denote  the set of measurable functions on $I$ valued in $X$
such that the map $t\rightarrow\|u(t)\|_{X}$  belongs to the Lebesgue space $L^{r}(I)$.
If $I=[0,T]$, we shall abbreviate $L^{r}(I,X)$ as $L^{r}_T(X)$.
For any $p\geq 1$ we use $p'$ to denote the conjugate exponent of $p$, defined by
$\frac{1}{p}+\frac{1}{p'}=1$. We use the letter $C$ to denote the positive constant which may change from line to line. We also omit
the spatial domain $\mathbb{R}^d\,(d=2, 3)$  in the integrals and the norms of function spaces for simplicity of presentation.


\subsection{A priori estimates for the linearized equations}

Let us first recall the standard estimates in the Besov spaces for the following linear transport equation:
\begin{equation}\label{bb1}
\partial_{t}a+ v\cdot \nabla a=f, \quad a|_{t=0}=a_{0}.
\end{equation}

\begin{Proposition}[\cite{BCD}]\label{prop4.1}
 Let $\sigma\in (-\frac{d}{2},\frac{d}{2}]$.
There exists a constant C, depending only on $d$ and $\sigma$, such that for all solution
$a\in L^{\infty}_T(\dot{B}^{\sigma}_{2,1})$ of \eqref{bb1}, initial data $a_{0}$ in $\dot{B}^{\sigma}_{2,1}$,
and $f$ in $ L^{1}_T(\dot{B}^{\sigma}_{2,1})$, we have, for a.e. $t\in[0,T]$,
\begin{align}
\|a\|_{\widetilde{L}^{\infty}_{t}(\dot{B}^{\sigma}_{2,1})}\leq \bigg\{\|a_{0}\|_{\dot{B}^{\sigma}_{2,1}}
+\int_{0}^{t}e^{-C V(\tau)}\|f(\tau)\|_{\dot{B}^{\sigma}_{2,1}} {\rm d}\tau \bigg\}e^{C V(t)}\nonumber
\end{align}
with $$V:= \int_{0}^{t}V^{'}(\tau)\mathrm{d}\tau$$ and $$V^{'}(t) :=\|\nabla v(t)\|_{\dot{B}^{\frac{d}{2}}_{2,1}}.$$

\end{Proposition}

\begin{Proposition}[\cite{BCD}]\label{le4.1}
 Let $1\leq p\leq p_{1}\leq\infty.$  Assume that
$$
\sigma >-d \min\Big\{\frac{1}{p_{1}},\frac{1}{p^{'}}\Big\} \ \ \  \text{or} \ \ \
  \sigma>-1-d \min\Big\{\frac{1}{p_{1}},\frac{1}{p^{'}}\Big\}\ \ \text{if} \ \ \dv v=0
$$
with the additional condition
$
  \sigma \leq1+\frac{d}{p_{1}}.
$
Let $a_{0}\in \dot{B}^{\sigma}_{p,r},f\in L^{1}_T(\dot{B}^{\sigma}_{p,1})$ and $v$ be a time-dependent vector field
such that $v\in L^{\rho}_T(\dot{B}^{-M}_{\infty,\infty})$ for some $\rho >1 $ and $M>0$ and
$$
\nabla v\in L^{1}_T(\dot{B}^{p_{1}}_{p_{1},\infty}\cap L^{\infty})\ \ \text{if}\ \ \sigma<1+\frac{d}{p_{1}}.
$$
Then the problem \eqref{bb1} has a unique solution $a$ in the space $\mathcal{C}([0,T];\dot{B}^{\sigma}_{p,1})$.
\end{Proposition}

For the momentum equation \eqref{Aa6}, we have to consider a linearization which allows for non-constant coefficients, namely,
\begin{align}
 \partial_{t}u+v\cdot\nabla u+u\cdot\nabla w-b\mathcal{A}u=g, \quad u|_{t=0}=u_{0},\label{ba2}
\end{align}
where $b$ is a given positive function depending on $(t,x)$ and tending to some constant (say 1) when $x$ goes to infinity.

\begin{Proposition}[\cite{BCD}]\label{prop4.2}
Let $\alpha\in(0,1]$  and $s\in (-\frac{d}{2},\frac{d}{2}]$. Assume that b=1+c with $c\in L^{\infty}_{T}(\dot{B}^{\frac{d}{2}+\alpha}_{2,1})$ and that
$$
b_{*}:=\inf_{(t,x)\in [0,T]\times \mathbb{R}^{d}} b(t,x)>0.
$$
Assume that $u_0\in \dot{B}^{s}_{2,1}$, $g\in L^1_T(\dot{B}^{s}_{2,1})$, and
$v,w\in L^1_T(\dot{B}^{\frac{d}{2}+1}_{2,1})$ are time-dependent vector fields.
Then there exists a universal constant $\kappa$, and a constant C depending only on $d,\alpha,\,and\, s$,
such that, for all $t\in [0,T],$ the solution of the problem \eqref{ba2} satisfies
\begin{align}
&\|u\|_{\widetilde{L}^{\infty}_{t}(\dot{B}^{s}_{2,1})}+\kappa b_{*}\mu\|u\|_{L^{1}_{t}(\dot{B}^{s+2}_{2,1})}
\leq\big(\|u_{0}\|_{\dot{B}^{s}_{2,1}}+\|g\|_{L^{1}_{t}(\dot{B}^{s}_{2,1})}\big)\nonumber \\
&\qquad \times\exp \bigg\{C\int^{t}_{0}\Big(\|v\|_{\dot{B}^{\frac{d}{2}+1}_{2,1}}+\|w\|_{\dot{B}^{\frac{d}{2}+1}_{2,1}}+
b_{*}\mu\Big(\frac{2\mu+\lambda}{b_{*}\mu}\Big)^{\frac{2}{\alpha}}\|c\|_{\dot{B}^{\frac{d}{2}+\alpha}_{2,1}}^{\frac{2}{\alpha}}\Big){\rm d}\tau
\bigg\}.\nonumber
\end{align}
\end{Proposition}

For the linearized magnetic equation associated with the system \eqref{Aa5}-\eqref{Aa8}:
 \begin{equation}\label{bb2}
\left\{\begin{array}{l}
\partial_{t}H+ u\cdot \nabla H+H\cdot\nabla u-\nu\Delta H= -(\dv u)H+g,\\
\dv H=0,\\
H|_{t=0}=H_{0},
\end{array}
\right.
\end{equation}
we have  the following proposition.
\begin{Proposition}\label{prop4.3}
Let $s\in (-\frac{d}{2},\frac{d}{2}].$
Assume that $H_0\in \dot{B}^{s}_{2,1}$ with $\dv H_0=0$, $g\in L^1_T(\dot{B}^{s}_{2,1})$, and
$u\in L^1_T( \dot{B}^{\frac{d}{2}+1}_{2,1})$ is time-dependent vector field.
Then there exist a universal constant $\kappa$, and a constant C depending only
on $d$ and $s$, such that
$$
\|H\|_{\widetilde{L}^{\infty}_{t}(\dot{B}^{s}_{2,1})}+\kappa \nu \|H\|_{L^{1}_{t}(\dot{B}^{s+2}_{2,1})}
\leq\big( \|H_{0}\|_{\dot{B}^{s}_{2,1}}+\|g\|_{L^{1}(\dot{B}^{s}_{2,1})}\big) \exp \bigg\{C\int_{0}^{t}\|\nabla u\|_{\dot{B}^{\frac{d}{2}}_{2,1}}{\rm d}\tau\bigg\}.
$$
\end{Proposition}

\begin{proof}
We shall adopt the homogeneous Littlewood-Paley decomposition technique to obtain the desired estimate.
 More precisely, by applying $\dot{\Delta}_{j}$ to \eqref{bb2}, we obtain that
$$
\partial_{t}H_{j}+u\cdot\nabla H_{j}-\nu\Delta H_{j}=-\dot{\Delta}_{j}((\dv u)H)-\dot{\Delta}_{j}(H\cdot\nabla u)+R_{j}+g_{j},
\ \ H_{j}|_{t=0}=H_{{0}j},
$$
where
$$
H_{j}:=\dot{\Delta}_{j}H, \quad  R_{j}:=\sum_{k}[u^{k},\dot{\Delta}_{j}]\partial_{k}H, \quad  g_{j}:=\dot{\Delta}_{j}g,
\quad H_{{0}j}:=\dot{\Delta}_{j}H_{{0}}.
$$

Taking the $L^{2}$ inner product of the above equation with $H_{j}$, we easily get
\begin{align}
&\frac{1}{2}\frac{{\rm d}}{{\rm d}t}\|H_{j}\|_{L^{2}}^{2}-\frac{1}{2}\int |H_{j}|^{2}\dv u {\rm d}x +\nu\int |\nabla H_{j}|^{2}{\rm d}x
\nonumber\\
&\qquad \quad\leq\|H_{j}\|_{L^{2}}\big(\|\dot{\Delta}_{j}(\dv u H)\|_{L^{2}}+\|\dot{\Delta}_{j}(H\cdot\nabla u)\|_{L^{2}}+\|R_{j}\|_{L^{2}}+\|g_{j}\|_{L^{2}}\big).\nonumber
\end{align}
Hence, by the Bernstein's inequality, we get, for some universal constant $\kappa$,
\begin{align}\label{bAA}
\frac{1}{2}&\frac{{\rm d}}{{\rm d}t}\|H_{j}\|_{L^{2}}^{2} +2\kappa\nu 2^{2j} \|\nabla H_{j}\|_{L^{2}}
\nonumber\\
&\leq\|H_{j}\|_{L^{2}}\big(\|\dot{\Delta}_{j}(\dv u H)\|_{L^{2}}+\|\dot{\Delta}_{j}(H\cdot\nabla u)\|_{L^{2}}\nonumber\\
& \quad+\|R_{j}\|_{L^{2}}+\frac{1}{2}\|\dv u\|_{L^{\infty}}\|H_{j}\|_{L^{2}}+\|g_{j}\|_{L^{2}}\big).
\end{align}
Thanks to   Proposition \ref{prop2.2} and  the commutator estimates (see Lemma 2.100 in \cite{BCD}),
 we have the following estimates for $\dot{\Delta}_{j}(\dv u H)$,
$\dot{\Delta}_{j}(H\cdot\nabla u)$ and $R_{j}$:
\begin{align*}
&\|\dot{\Delta}_{j}(\dv u H)\|_{L^{2}}\leq C c_{j}2^{-js}\|\dv u\|_{\dot{B}^{\frac{d}{2}}_{2,1}}
\|H\|_{\dot{B}^{s}_{2,1}},\ \ \text{if} \ \  -\frac{d}{2}<s\leq\frac{d}{2},\\
&\|\dot{\Delta}_{j}(H\cdot\nabla u)\|_{L^{2}}\leq C c_{j}2^{-js}\|\nabla u\|_{\dot{B}^{\frac{d}{2}}_{2,1}}
\|H\|_{\dot{B}^{s}_{2,1}},\ \ \text{if}\ \  -\frac{d}{2}<s\leq\frac{d}{2},\\
&\|R_{j}\|_{L^{2}}\leq C c_{j}2^{-js}\|\nabla u\|_{\dot{B}^{\frac{d}{2}}_{2,1}}
\|H\|_{\dot{B}^{s}_{2,1}},\ \ \text{if} \ \ -\frac{d}{2}<s\leq\frac{d}{2}+1,
\end{align*}
where $(c_{j})_{j\in \mathbb{Z}}$ denotes a positive sequence such that $\sum_{j\in \mathbb{Z}}c_{j}=1$.

Formally dividing both sides of the inequality \eqref{bAA} by $\|H_{j}\|_{L^{2}}$ and integrating over $[0,t]$,
one has
\begin{align}
&\|H_{j}(t)\|_{L^{2}}+2\kappa\nu2^{2j}\int^{t}_{0}\|H_{j}\|_{L^{2}}{\rm d}\tau\nonumber\\
 & \qquad \qquad \leq \|H_{j}(0)\|_{L^{2}}+\|g_{j}\|_{L^{2}}
+C2^{-js}\int^{t}_{0}c_{j}\|\nabla u\|_{\dot{B}^{\frac{d}{2}}_{2,1}}\|H\|_{\dot{B}^{s}_{2,1}}{\rm d}\tau. \nonumber
\end{align}
Now, multiplying both sides by $2^{js}$ and summing over $j$, we end up with
\begin{align}
\|H\|_{\widetilde{L}^{\infty}_{t}(\dot{B}^{s}_{2,1})}+\kappa\nu \|H\|_{L^{1}_{t}(\dot{B}^{s+2}_{2,1})}
&\leq \|H(0)\|_{\dot{B}^{s}_{2,1}}+\|g_{j}\|_{L^{2}}+C\int^{t}_{0}\|\nabla u\|_{\dot{B}^{\frac{d}{2}}_{2,1}}\|H\|_{\dot{B}^{s}_{2,1}}{\rm d}\tau\nonumber
\end{align}
for some constant $C$ depending only on $d$ and $s$. Applying Gronwall's lemma then completes  the proof.
\end{proof}

With these estimates in hand,  we can prove the following result for smooth solutions to the problem \eqref{Aa5}-\eqref{Aa8}.

\begin{Proposition}\label{PrA}
Let $(a,u,H)$ satisfy \eqref{Aa5}-\eqref{Aa8} on $[0,T]\times\mathbb{R}^{d}$. Suppose that  there exist two positive constants $b_{*}$ and $b^{*}$, such that
$$b_{*}\leq 1+a_{0}\leq b^{*}$$
and that $a\in \mathcal{C}^{1}([0,T];\dot{B}^{\frac{d}{2}}_{2,1}\cap\dot{B}^{\frac{d}{2}+\alpha}_{2,1})$ and
$(u,H) \in \mathcal{C}^{1}([0,T];\dot{B}^{\frac{d}{2}-1}_{2,1}\cap\dot{B}^{\frac{d}{2}+1+\alpha}_{2,1})^{d+d}.$
Assume, in addition, that there exists a function $u_{L}\in \mathcal{C}^{1}([0,T];\dot{B}^{\frac{d}{2}-1}_{2,1}\cap\dot{B}^{\frac{d}{2}+1+\alpha}_{2,1})^{d}$
satisfies
$$
\partial_{t}u_{L}-\mathcal{A}u_{L}=0,\qquad u_{L}|_{t=0}=u_{0},
$$
and that there exists a function $H_{L}\in \mathcal{C}^{1}([0,T];\dot{B}^{\frac{d}{2}-1}_{2,1}\cap\dot{B}^{\frac{d}{2}+1+\alpha}_{2,1})^{d}$
satisfies
$$
\partial_{t}H_{L}-\nu\Delta H_{L}=0,\qquad {H_{L}}|_{t=0}=H_{0}.
$$
Denote  $ \bar{u} := u-u_{L}$ and
\begin{align*}
& A_{0}^{\alpha} := \|a_{0}\|_{\dot{B}^{\frac{d}{2}}_{2,1}\cap\dot{B}^{\frac{d}{2}+\alpha}_{2,1}},\quad \ \
A^{\alpha}(t) := \|a\|_{L^{\infty}_{t}(\dot{B}^{\frac{d}{2}}_{2,1}\cap\dot{B}^{\frac{d}{2}+\alpha}_{2,1})},\\
& U_{0}^{\alpha}(t) := \|u_{0}\|_{\dot{B}^{\frac{d}{2}-1}_{2,1}\cap\dot{B}^{\frac{d}{2}-1+\alpha}_{2,1}},\quad \ \
U^{\alpha}_{L}(t) := \|u_{L}\|_{L^{1}_{t}(\dot{B}^{\frac{d}{2}+1}_{2,1}\cap\dot{B}^{\frac{d}{2}+1+\alpha}_{2,1})},\\
& \bar{U}^{\alpha}(t) := \|\bar{u}\|_{L^{\infty}_{t}(\dot{B}^{\frac{d}{2}-1}_{2,1}
\cap\dot{B}^{\frac{d}{2}-1+\alpha}_{2,1})}+b_{*}\mu \|\bar{u}\|_{L^{1}_{t}(\dot{B}^{\frac{d}{2}+1}_{2,1}
\cap\dot{B}^{\frac{d}{2}+1+\alpha}_{2,1})},\\
& H_{0}^{\alpha} := \|H_{0}\|_{\dot{B}^{\frac{d}{2}-1}_{2,1}\cap\dot{B}^{\frac{d}{2}-1+\alpha}_{2,1}},\quad
H_{L}(t) := \|H_{L}\|_{L^{1}_{t}(\dot{B}^{\frac{d}{2}+1}_{2,1})},\\
& H^{\alpha}(t) := \|H\|_{L^{\infty}_{t}(\dot{B}^{\frac{d}{2}-1}_{2,1}
\cap\dot{B}^{\frac{d}{2}-1+\alpha}_{2,1})}+\nu \|H\|_{L^{1}_{t}(\dot{B}^{\frac{d}{2}+1}_{2,1}
\cap\dot{B}^{\frac{d}{2}+1+\alpha}_{2,1})}.
\end{align*}
Assume further that there exist two constant $\eta$ and C, depending only on $d,\alpha$, and G, such that
\begin{align}
 & b_{*}\mu(\frac{\bar{\nu}}{b_{*}\mu})^{\frac{2}{\alpha}}T(A_{0}^{\alpha}+1)^{\frac{2}{\alpha}}\leq\eta,\label{bb3}\\
& \big(1+A_{0}^{\alpha}\big)^{2}
\Big(T(1+(H_{0}^{\alpha})^{2})+\bar{\nu}U_{L}^{\alpha}(T)+(U_{0}^{\alpha}+\bar{\nu})(H_{0}^{\alpha})^{2}U_{L}^{\alpha}(T)\nonumber\\
& \qquad \qquad\qquad\qquad   +(U_{0}^{\alpha}+(H_{0}^{\alpha})^{2}+1)(H_{0}^{\alpha})^{\frac{3}{2}}H_{L}^{\frac{1}{2}}(T)\Big)\leq\eta b_{*}\mu,\label{bb4}
\end{align}
then we have
\begin{align}
&\frac{b_{*}}{2} \leq1+a(t,x)\leq2b^{*} \ \  \text{for  all}\ \   (t,x)\in[0,T]\times\mathbb{R}^{d},\label{bb5}\\
&A^{\alpha}(T) \leq2A_{0}^{\alpha}+1,\quad H^{\alpha}(T)\leq 2H_{0}^{\alpha},\label{bb6}\\
&\bar{U}^{\alpha}(T) \leq C\big(1+A_{0}^{\alpha}\big)^{2}
\Big(T(1+(H_{0}^{\alpha})^{2})+\bar{\nu}U_{L}^{\alpha}(T)+(U_{0}^{\alpha}+\bar{\nu})(H_{0}^{\alpha})^{2}U_{L}^{\alpha}(T)\nonumber\\
&\qquad\qquad  +(U_{0}^{\alpha}+(H_{0}^{\alpha})^{2}+1)(H_{0}^{\alpha})^{\frac{3}{2}}H_{L}^{\frac{1}{2}}(T)\Big),\label{bb7}
\end{align}
with $\bar{\nu}:=\lambda+2\mu$.
\end{Proposition}

\begin{proof} Setting   $\bar{H}:= H-H_{L}$ and $I(a):= \frac{a}{1+a}$,
we may write the system satisfied by $(a,\bar{u},H, \bar H)$ as:
\begin{equation}\label{bb8}
\left\{\begin{array}{l}
\partial_{t}a+u\cdot\nabla a+(1+a)\dv u=0,\\
\partial_{t}u+u\cdot\nabla u+\bar{u}\cdot \nabla u_{L}-\frac{1}{1+a}\mathcal{A}\bar{u}\\
\qquad \qquad =
-u_{L}\cdot\nabla u_{L}-I(a)\mathcal{A}u_{L}-\nabla G(a)+\frac{1}{1+a}\big(H\cdot\nabla H-\frac{1}{2}\nabla|H|^{2}\big), \\
\partial_{t}H+u\cdot\nabla H-H\cdot\nabla u-\nu\Delta H=-\dv u H  , \quad \dv H=0,\\
\partial_{t}\bar{H}+u\cdot\nabla \bar{H}-\bar{H}\cdot\nabla u-\nu\Delta \bar{H}  \\
\qquad \qquad=-\dv u \bar{H}  -\dv u H_{L}-u\cdot\nabla H_{L}-H_{L}\cdot\nabla u,   \\
(a,\bar{u},H,\bar{H})_{|t=0}=(a_{0},0,H_{0},0),
\end{array}\right.
\end{equation}

We first estimate the bound of $a$. Applying the product law in Besov spaces, we get
$$
\|(1+a)\dv u\|_{\dot{B}^{\frac{d}{2}}_{2,1}\cap\dot{B}^{\frac{d}{2}+\alpha}_{2,1}}
\leq C\Big(1+\|a\|_{\dot{B}^{\frac{d}{2}}_{2,1}\cap\dot{B}^{\frac{d}{2}+\alpha}_{2,1}}\Big)
\|\dv u\|_{\dot{B}^{\frac{d}{2}}_{2,1}\cap\dot{B}^{\frac{d}{2}+\alpha}_{2,1}}.
$$
Hence, combining Proposition \ref{prop4.1} with Gronwall's lemma yields, for all $t\in[0,T],$
\begin{align}\label{bb9}
A^{\alpha}(t)\leq\, & A_{0}^{\alpha}\exp\bigg\{C\int_{0}^{t}\|\nabla u\|_{\dot{B}^{\frac{d}{2}}_{2,1}\cap\dot{B}^{\frac{d}{2}+\alpha}_{2,1}}{\rm d}\tau\bigg\}\nonumber\\
& +\exp\bigg\{C\int_{0}^{t}\|\nabla u\|_{\dot{B}^{\frac{d}{2}}_{2,1}\cap\dot{B}^{\frac{d}{2}+\alpha}_{2,1}}{\rm d}\tau\bigg\}-1.
\end{align}
In order to ensure that the condition \eqref{bb5} is satisfied, we   use the fact:
$$
(\partial_{t}+u\cdot\nabla)(1+a)^{\pm1}\pm(1+a)^{\pm1}\dv u=0.
$$
Hence, taking advantage of Gronwall's lemma, we obtain that
$$\|(1+a)^{\pm1}(t)\|_{L^{\infty}}\leq \|(1+a_{0})^{\pm1}\|_{L^{\infty}}\exp\bigg\{\int_{0}^{t}\|\dv u\|_{L^{\infty}}{\rm d}\tau\bigg\}.$$
Therefore, the condition \eqref{bb5} is satisfied on $[0,t]$ if
\begin{align}\label{bb10}
\int_{0}^{t}\|\dv u\|_{L^{\infty}} {\rm d}\tau \leq \log2.
\end{align}

Now, we estimate $H$ and $\bar{H}$. By Proposition \ref{prop4.3} and Remark \ref{rem2.1}, we have
\begin{align*}
& H^{\alpha}(t)\leq C H_{0}^{\alpha}(t)\exp\bigg\{C\bigg(\int_{0}^{t}\|u_{L}\|_{\dot{B}^{\frac{d}{2}+1}_{2,1}}{\rm d}\tau+
\int_{0}^{t}\|\bar{u}\|_{\dot{B}^{\frac{d}{2}+1}_{2,1}}{\rm d}\tau\bigg)\bigg\}, \\
&\|H_{L}\|_{L^{\infty}_{t}(\dot{B}^{\frac{d}{2}-1}_{2,1})}+
\|H_{L}\|_{L^{1}_{t}(\dot{B}^{\frac{d}{2}+1}_{2,1})}\leq C H_{0}^{\alpha},\nonumber\\
&\|u_{L}\|_{L^{\infty}_{t}(\dot{B}^{\frac{d}{2}-1}_{2,1})}+
\|u_{L}\|_{L^{1}_{t}(\dot{B}^{\frac{d}{2}+1}_{2,1})}\leq C U_{0}^{\alpha}.
\end{align*}
Thanks to Proposition \ref{prop2.2}, we obtain that
\begin{align*}
\| u\cdot\nabla H_{L}\|_{L^{1}_{t}(\dot{B}^{\frac{d}{2}-1}_{2,1})}
\leq \,& C
\|u\|_{L_{T}^{\infty}(\dot{B}^{\frac{d}{2}-1}_{2,1})}\|H_{L}\|_{L_{T}^{1}(\dot{B}^{\frac{d}{2}+1}_{2,1})}\nonumber\\
 \leq\, & C(U_{0}^{\alpha}+\bar{U}^{\alpha}(T))(H_{0}^{\alpha})^{\frac{1}{2}}
\|H_{L}\|^{\frac{1}{2}}_{L_{T}^{1}(\dot{B}^{\frac{d}{2}+1}_{2,1})},\nonumber\\
\| H_{L} \cdot\nabla u\|_{L^{1}_{t}(\dot{B}^{\frac{d}{2}-1}_{2,1})}\leq \,& C
\|u\|_{L_{T}^{2}(\dot{B}^{\frac{d}{2}}_{2,1})}\|H_{L}\|_{L_{T}^{2}(\dot{B}^{\frac{d}{2}}_{2,1})}\nonumber\\
  \leq\, & C \|u\|^{\frac{1}{2}}_{L_{T}^{\infty}(\dot{B}^{\frac{d}{2}-1}_{2,1})}
\|u\|^{\frac{1}{2}}_{L_{T}^{1}(\dot{B}^{\frac{d}{2}+1}_{2,1})}
\|H_{L}\|^{\frac{1}{2}}_{L_{T}^{\infty}(\dot{B}^{\frac{d}{2}-1}_{2,1})}
\|H_{L}\|^{\frac{1}{2}}_{L_{T}^{1}(\dot{B}^{\frac{d}{2}+1}_{2,1})}\nonumber\\
\leq\, & C(U_{0}^{\alpha}+\bar{U}^{\alpha}(T))(H_{0}^{\alpha})^{\frac{1}{2}}
\|H_{L}\|^{\frac{1}{2}}_{L_{T}^{1}(\dot{B}^{\frac{d}{2}+1}_{2,1})}.
\end{align*}
Therefore,
\begin{align*}
&\|\bar{H}\|_{L^{1}_{t}(\dot{B}^{\frac{d}{2}+1}_{2,1})}\leq C\Big(\| H_{L} \cdot\nabla u\|_{L^{1}_{t}(\dot{B}^{\frac{d}{2}-1}_{2,1})}
+\| u\cdot\nabla H_{L}\|_{L^{1}_{t}(\dot{B}^{\frac{d}{2}-1}_{2,1})}\Big)  \exp \bigg\{C\int_{0}^{t}\|u\|_{\dot{B}^{\frac{d}{2}+1}_{2,1}}{\rm d}\tau\bigg\}\nonumber\\
&\qquad \qquad\qquad\leq C(U_{0}^{\alpha}+\bar{U}^{\alpha}(T))(H_{0}^{\alpha})^{\frac{1}{2}}
\|H_{L}\|^{\frac{1}{2}}_{L_{T}^{1}(\dot{B}^{\frac{d}{2}+1}_{2,1})}\exp\bigg\{ C\int_{0}^{t}\|u\|_{\dot{B}^{\frac{d}{2}+1}_{2,1}}{\rm d}\tau\bigg\}.\nonumber
\end{align*}
In order to bound $\bar{u}$, we use Proposition \ref{prop4.2} with $c=-I(a)$. Thanks  to Proposition \ref{prop2.2}, we have,
for all $\beta\in\{0,\alpha\}$,
\begin{align}
&\|u_{L}\cdot\nabla u_{L}\|_{\dot{B}^{\frac{d}{2}-1+\beta}_{2,1}}\leq C\|\nabla u_{L}\|_{\dot{B}^{\frac{d}{2}}_{2,1}}
\| u_{L}\|_{\dot{B}^{\frac{d}{2}-1+\beta}_{2,1}},\nonumber\\
&\|\nabla G(a)\|_{\dot{B}^{\frac{d}{2}-1+\beta}_{2,1}}\leq C\|a\|_{\dot{B}^{\frac{d}{2}+\beta}_{2,1}},\nonumber\\
&\|I(a)\cdot\mathcal{A} u_{L}\|_{\dot{B}^{\frac{d}{2}-1+\beta}_{2,1}}\leq C \bar{\nu}\|I(a) \|_{\dot{B}^{\frac{d}{2}}_{2,1}
\cap\dot{B}^{\frac{d}{2}+\beta}_{2,1}}\|\nabla^{2} u_{L}\|_{\dot{B}^{\frac{d}{2}-1+\beta}_{2,1}},\nonumber\\
& \|I(a) \|_{\dot{B}^{\frac{d}{2}}_{2,1}\cap\dot{B}^{\frac{d}{2}+\beta}_{2,1}}\leq
C \|a \|_{\dot{B}^{\frac{d}{2}}_{2,1}\cap\dot{B}^{\frac{d}{2}+\beta}_{2,1}},\nonumber\\
& \Big\|\frac{1}{1+a}\big(H\cdot\nabla H-\frac{1}{2}\nabla |H|^{2}\big)\Big\|_{\dot{B}^{\frac{d}{2}-1+\beta}_{2,1}}
\leq C(1+\|a\|_{\dot{B}^{\frac{d}{2}}_{2,1}})\|\nabla H\|_{\dot{B}^{\frac{d}{2}}_{2,1}}\| H\|_{\dot{B}^{\frac{d}{2}-1+\beta}_{2,1}}.\nonumber
\end{align}
It is easily prove that
$$
\|u_{L}\|_{L^{\infty}_{T}(\dot{B}^{\frac{d}{2}-1}_{2,1}\cap\dot{B}^{\frac{d}{2}-1+\alpha}_{2,1})}
\leq C U_{0}^{\alpha}.
$$
Thus, we have
\begin{align}
\bar{U}^{\alpha}(T)\leq &C \exp\bigg\{C\int _{0}^{T}\Big(\|\bar{u}\|_{\dot{B}^{\frac{d}{2}+1}_{2,1}}+
\|u_{L}\|_{\dot{B}^{\frac{d}{2}+1}_{2,1}}+b_{*}\mu\big(\frac{\bar{\nu}}{b_{*}\mu}\big)^{\frac{2}{\alpha}}
\|a\|^{\frac{2}{\alpha}}_{\dot{B}^{\frac{d}{2}+\alpha}_{2,1}}\Big){\rm d}t\bigg\}\nonumber\\
&\times \left(\|u_{L}\|_{L^{1}(\dot{B}^{\frac{d}{2}+1}_{2,1})}
\|u_{L}\|_{L^{\infty}_{T}(\dot{B}^{\frac{d}{2}-1}_{2,1}\cap\dot{B}^{\frac{d}{2}-1+\alpha}_{2,1})}\right.\nonumber\\
&\qquad +\|a\|_{L^{\infty}_{T}(\dot{B}^{\frac{d}{2}}_{2,1}\cap\dot{B}^{\frac{d}{2}+\alpha}_{2,1})}
\Big(T+\bar{\nu}\|u_{L}\|_{L^{1}_{T}(\dot{B}^{\frac{d}{2}+1}_{2,1}\cap\dot{B}^{\frac{d}{2}+1+\alpha}_{2,1})}\Big)\nonumber\\
&\qquad \left. +\Big(1+\|a\|_{L^{\infty}_{T}(\dot{B}^{\frac{d}{2}}_{2,1}\cap\dot{B}^{\frac{d}{2}+\alpha}_{2,1})}\Big)
\|H\|_{L^{1}_{T}(\dot{B}^{\frac{d}{2}+1}_{2,1})}
\|H\|_{L^{\infty}_{T}(\dot{B}^{\frac{d}{2}-1}_{2,1}\cap\dot{B}^{\frac{d}{2}-1+\alpha}_{2,1})}\right). \label{bb11}
\end{align}
Now, if $T$ is sufficiently small so that
\begin{align}
& \exp\{CU_{L}^{\alpha}(T)\}\leq\sqrt{2},\quad \exp\Big\{C\frac{\bar{U}^{\alpha}(T)}{b_{*}\mu}\Big\}\leq\sqrt{2},\label{bb12}\\
& \exp{\Big\{Cb_{*}\mu\big(\frac{\bar{\nu}}{b_{*}\mu}\big)^{\frac{2}{\alpha}}T(A^{\alpha}(T))^{\frac{2}{\alpha}}\Big\}}\leq 2,\label{bb13}
\end{align}
we have:
\begin{align}
&A^{\alpha}(T)\leq 2A_{0}^{\alpha}+1,\label{bb14}\\
&\bar{U}^{\alpha}(T)\leq C (U_{0}^{\alpha}U_{L}^{\alpha}(T)+(1+A_{0}^{\alpha})(T+\bar{\nu}U_{L}^{\alpha}(T)+
(H_{0}^{\alpha})^{2}))  :=C X(T),\label{bb15}\\
&\|\bar{H}\|_{L^{1}_{T}(\dot{B}^{\frac{d}{2}+1}_{2,1})}\leq C (U_{0}^{\alpha}+X(T))(H_{0}^{\alpha})^{\frac{1}{2}}
\|H_{L}\|^{\frac{1}{2}}_{L^{1}_{T}(\dot{B}^{\frac{d}{2}+1}_{2,1})},\label{bb16}\\
&\|H\|_{L^{1}_{T}(\dot{B}^{\frac{d}{2}+1}_{2,1})}\leq C(U_{0}^{\alpha}+X(T)+1)(H_{0}^{\alpha})^{\frac{1}{2}}
\|H_{L}\|^{\frac{1}{2}}_{L^{1}_{T}(\dot{B}^{\frac{d}{2}+1}_{2,1})}. \label{bb17}
\end{align}
By \eqref{bb11}, \eqref{bb12}, \eqref{bb13}, and \eqref{bb17},    we obtain that
\begin{align}
\bar{U}^{\alpha}(T)&\leq C \big(U_{0}^{\alpha}U^{\alpha}_{L}(T)+(1+A_{0}^{\alpha})(T+\bar{\nu}U_{L}^{\alpha}(T))\nonumber\\
&\quad +(1+A_{0}^{\alpha})(U_{0}^{\alpha}+X(T)+1)(H_{0}^{\alpha})^{\frac{3}{2}}
H_{L}(T)^{\frac{1}{2}}\big)\nonumber\\
&\leq C\big(1+A_{0}^{\alpha}\big)^{2}
\big(T(1+(H_{0}^{\alpha})^{2})+\bar{\nu}U_{L}^{\alpha}(T)+(U_{0}^{\alpha}
+\bar{\nu})(H_{0}^{\alpha})^{2}U_{L}^{\alpha}(T)\nonumber\\
&\quad+(U_{0}^{\alpha}+(H_{0}^{\alpha})^{2}+1)(H_{0}^{\alpha})^{\frac{3}{2}}H_{L}^{\frac{1}{2}}(T)\big).\nonumber
\end{align}
Thus, if we choose $T>0$ such that \eqref{bb3}-\eqref{bb4} is satisfied for some sufficiently small constant $\eta$, then both \eqref{bb7} and
the above conditions \eqref{bb12}-\eqref{bb13}  are satisfied with a strict inequality.  It is now easy to complete the proof
by means of a bootstrap argument.
\end{proof}

\subsection{Existence of the local solution}

In this subsection, we shall prove the existence part of Theorem \ref{ThA}. We adopt the simlar process developed by Danchin
\cite{Db,De} for the compressible Navier-Stokes equations (see also \cite{BCD}). Briefly, this process can be described as follows.
First, we approximate the system \eqref{Aa5}-\eqref{Aa7} by a sequence of ordinary differential equations by applying the Friedrichs regularity method.
Then,  we prove uniform a priori estimates in $E_{T}^{\alpha}$ for these solutions. Next,
 we establish further boundedness properties involving the H\"{o}lder regularity with respect to time for these
approximate solutions. Finally, we use the previous steps  to show compactness and convergence of the  approximate solutions (up to an extraction).
We shall focus on the analysis on the
coupling term of the velocity field  and the magnetic field.

%

\subsubsection{Friedrichs approximation of the system.}
Let $\dot{L}_{n}^{2}$ be the set of $L^{2}$ functions spectrally supported in the annulus
$\mathcal{C}_{n} :=\{\xi\in\mathbb{R}^{d}|\,  n^{-1}\leq\xi\leq n\}$ and let $\Omega_{n}$
be the set of functions $(a,u,H)$ of $(\dot{L}^{2}_{n})^{2d+1}$ such that $\inf_{x\in \mathbb{R}^{d}}a>-1$.
The linear space $\dot{L}_{n}^{2}$ is endowed with the standard $L^{2}$ topology. Due to
the Bernstein's inequality, the $L^{\infty}$ topology on $\dot{L}_{n}^{2}$ is weaker than the usual $L^{2}$
topology, thus $\Omega_{n}$ is an open set of $(\dot{L}^{2}_{n})^{2d+1}$.
Let $$\dot{\mathbb{E}}_{n}:L^{2}\rightarrow \dot{L}_{n}^{2}$$ be the Friedrichs projector, defined by
$$
\mathcal{F}\mathbb{\dot{E}}_{n}U(\xi) := \textbf{1}_{\mathcal{C}_{n}}(\xi)\mathcal{F}U(\xi)\ \  \text{for all}\ \ \xi\in \mathbb{R}^{d}.
$$
We aim to solve the system of ordinary differential equations:
\begin{align}\label{bb18}
\frac{{\rm d}}{{\rm d}t}\begin{pmatrix}a\\\bar{u}\\H\end{pmatrix}=\begin{pmatrix}F_{n}(a,\bar{u},H)\\
G_{n}(a,\bar{u},H)\\Q_{n}(a,\bar{u},H)\end{pmatrix},\qquad
\begin{pmatrix}a\\ \bar{u}\\H\end{pmatrix}\Bigg|_{t=0}=\begin{pmatrix}\mathbb{\dot{E}}_{n}a_{0}\\
0\\ \mathbb{\dot{E}}_{n}H_{0}\end{pmatrix}
\end{align}
with $u  := \bar{u}+u_{L}$ and
\begin{align}
F_{n}(a,\bar{u},H) & := -\mathbb{\dot{E}}_{n}\dv((1+a)u), \nonumber\\
G_{n}(a,\bar{u},H)& := \mathbb{\dot{E}}_{n}\Big(\frac{1}{1+a}\mathcal{A}\bar{u}\Big)-\mathbb{\dot{E}}_{n}
(u\cdot\nabla u)-\mathbb{\dot{E}}_{n}(I(a)\mathcal{A}u_{L})\nonumber\\
&\quad \  -\mathbb{\dot{E}}_{n}\nabla\big(G(a)\Big)
+\mathbb{\dot{E}}_{n}\Big(\frac{1}{1+a}\big(H\cdot\nabla H-\frac{1}{2}\nabla |H|^{2}\big)\Big),\nonumber\\
Q_{n}(a,\bar{u},H)& := \mathbb{\dot{E}}_{n}(H\cdot\nabla u)
-\mathbb{\dot{E}}_{n}(u\cdot\nabla H)+\nu\mathbb{\dot{E}}_{n}(\Delta H)-\mathbb{\dot{E}}_{n}(H\dv u).\nonumber
\end{align}
Notice that if $1+a_{0}$ is positive and bounded away from zero, then so is $1+\mathbb{\dot{E}}_{n}a_{0}$ for sufficiently large $n$,
 and hence the initial data belongs to $\Omega_{n}$. It is easy to check that the map
$$(a,\bar{u},H)\rightarrow(F_{n}(a,\bar{u},H),G_{n}(a,\bar{u},H),Q_{n}(a,\bar{u},H))$$
belongs to  $\mathcal{C}(\mathbb{R}^{+}\times\Omega_{n};(\dot{L}^{2}_{n})^{2d+1})$ and is locally Lipschitz with respect to
the variable $(a,\bar{u},H)$. Therefore, the system \eqref{bb18} has a unique maximal solution $(a^{n},\bar{u}^{n},H^{n})$ in the
space $\mathcal{C}^{1}([0,T^{*}_{n});\Omega_{n})$.

%
\subsubsection{Uniform estimates of $(a^{n},{u}^{n},H^{n})$.}
First, we note that $(a^{n},\bar{u}^{n},H^{n})$ satisfies the system:
\begin{align}
\left\{\begin{array}{l}
\!\partial_{t}a^{n}+\mathbb{\dot{E}}_{n}(u^{n}\cdot\nabla a^{n})+\mathbb{\dot{E}}_{n}((1+a^{n})\dv u^{n} )=0,\\
\!\partial_{t}\bar{u}^{n}-\mathbb{\dot{E}}_{n}\Big(\frac{1}{1+a^{n}}\mathcal{A}\bar{u}^{n}\Big)+\mathbb{\dot{E}}_{n}
\big(u^{n}\cdot\nabla u^{n}\big)-\mathbb{\dot{E}}_{n}\big(I(a^{n})\mathcal{A}u^{n}_{L}\big)\\
\quad \qquad   +\nabla\mathbb{\dot{E}}_{n}\big(G(a^{n})\big)
-\mathbb{\dot{E}}_{n}\Big(\frac{1}{1+a^{n}}\big(H^{n}\cdot\nabla H^{n}-\frac{1}{2}\nabla |H^{n}|^{2}\big)\Big)=0,\\
\!\partial_{t}H^{n}-\mathbb{\dot{E}}_{n}(H^{n}\cdot\nabla u^{n})+\mathbb{\dot{E}}_{n}\big(u^{n}\cdot\nabla H^{n})-\nu\mathbb{\dot{E}}_{n}(\Delta H^{n})
+\mathbb{\dot{E}}_{n}(H^{n}\dv u^{n})=0 \nonumber
\end{array}\right.
\end{align}
with the initial data
$$( a^{n},\bar{u}^{n},H^{n})|_{t=0}=(\mathbb{\dot{E}}_{n}a_{0},0,\mathbb{\dot{E}}_{n}H_{0})(x), \quad x\in \mathbb{R}^d,$$
where $u^{n}:=u^{n}_{L}+\bar{u}^{n}$.
We claim that $T^{*}_{n}$ may be bounded from below by the supremum $T$ of all the time satisfying both \eqref{bb3} and \eqref{bb4},
and that $(a^{n},u^{n},H^{n})_{n\geq1}$ is bounded in $E_{T}^{\alpha}$. In fact,
since $\mathbb{\dot{E}}_{n}$ is an $L^{2}$ orthogonal projector, it has no effect on the energy estimates which are used in the proof of Proposition \ref{PrA}.
Hence, the  Proposition
\ref{PrA} applies to our approximate solution $(a^{n},u^{n},H^{n})$.
We remark  that the dependence on $n$ in the conditions \eqref{bb3} and \eqref{bb4} and in the inequalities \eqref{bb5}-\eqref{bb7} may be omitted.
Now, as $(a^{n},\bar{u}^{n},H^{n})$ is spectrally supported in $\mathcal{C}_{n}$,
the inequalities \eqref{bb5}-\eqref{bb7} ensure that it is bounded in $L^{\infty}_T(\dot{L}^{2}_{n})$.
Thus,  the standard continuation criterion for ordinary differential equations implies that
$T^{*}_{n}$ is greater than any time $T$ satisfying \eqref{bb3}-\eqref{bb4} and  that, for all $n\geq 1$,
\begin{align}
&\|a^{n}\|_{L^{\infty}_{T}(\dot{B}^{\frac{d}{2}}_{2,1}\cap\dot{B}^{\frac{d}{2}+\alpha}_{2,1})}\leq2A_{0}^{\alpha}+1,\nonumber\\
 &\|H^{n}\|_{L^{\infty}_{T}(\dot{B}^{\frac{d}{2}}_{2,1}\cap\dot{B}^{\frac{d}{2}+\alpha}_{2,1})}+
\nu\|H^{n}\|_{L^{1}_{T}(\dot{B}^{\frac{d}{2}+1}_{2,1}\cap\dot{B}^{\frac{d}{2}+1+\alpha}_{2,1})}\leq 2H_{0}^{\alpha},\nonumber\\
&\|\bar{u}^{n}\|_{L^{\infty}_{T}(\dot{B}^{\frac{d}{2}}_{2,1}\cap\dot{B}^{\frac{d}{2}+\alpha}_{2,1})}+
b_{*}\mu\|\bar{u}^{n}\|_{L^{1}_{T}(\dot{B}^{\frac{d}{2}+1}_{2,1}\cap\dot{B}^{\frac{d}{2}+1+\alpha}_{2,1})}\nonumber\\
& \qquad \leq C\big(1+A_{0}^{\alpha}\big)^{2}
\Big(T(1+(H_{0}^{\alpha})^{2})+\bar{\nu}U_{L}^{\alpha}(T) +(U_{0}^{\alpha}+\bar{\nu})(H_{0}^{\alpha})^{2}U_{L}^{\alpha}(T)\nonumber\\
& \qquad \qquad\qquad\qquad\quad
+(U_{0}^{\alpha}+(H_{0}^{\alpha})^{2}+1)(H_{0}^{\alpha})^{\frac{3}{2}}H_{L}^{\frac{1}{2}}(T)\Big).\nonumber
\end{align}
In particular, $(a^{n},u^{n},H^{n})_{n\geq 1}$ is bounded in $E_{T}^{\alpha}$.

%
\subsubsection{Time derivatives of $(a^{n},\bar {u}^{n},\bar H^{n})$.}
In order to  pass the limit in $(a^{n},{u}^{n},H^{n})$, we need the compactness in time of $(\bar a^{n},\bar {u}^{n},\bar H^{n})$ which
can be stated as the following lemma.

\begin{Lemma}\label{lemma2.2}
Let $\bar{a}^{n}:=a^{n}-\mathbb{\dot{E}}_{n}a_{0}$ and $\bar{H}^{n}:=H^{n}-\mathbb{\dot{E}}_{n}H_{0}$.
Then the sequence $(\bar{a}^{n})_{n\geq1}$ is   bounded in
$$\mathcal{C}([0,T];\dot{B}^{\frac{d}{2}}_{2,1}\cap\dot{B}^{\frac{d}{2}+\alpha}_{2,1})\cap
\mathcal{C}^{\frac{1}{2}}([0,T];\dot{B}^{\frac{d}{2}-1}_{2,1}\cap\dot{B}^{\frac{d}{2}-1+\alpha}_{2,1}), $$
  the sequence $(\bar{u}^{n})_{n\geq1}$ is bounded in
$$\mathcal{C}([0,T];\dot{B}^{\frac{d}{2}-1}_{2,1}\cap\dot{B}^{\frac{d}{2}-1+\alpha}_{2,1})\cap
\mathcal{C}^{\frac{1}{2}}([0,T];\dot{B}^{\frac{d}{2}-1}_{2,1}+\dot{B}^{\frac{d}{2}-2+\alpha}_{2,1}),$$
and the sequence $(\bar{H}^{n})_{n\geq1}$ is bounded in
$$
\mathcal{C}([0,T];\dot{B}^{\frac{d}{2}-1}_{2,1}\cap\dot{B}^{\frac{d}{2}-1+\alpha}_{2,1})\cap
\mathcal{C}^{\frac{1}{2}}([0,T];\dot{B}^{\frac{d}{2}-2}_{2,1}\cap\dot{B}^{\frac{d}{2}-2+\alpha}_{2,1}).
$$
\end{Lemma}
\begin{proof}
The result for $(\bar{a}^{n})_{n\geq1}$ follows from the facts that $\bar{a}^{n}|_{t=0}=0$ and that
\begin{align}
 \partial_{t}\bar{a}^{n}=-\mathbb{\dot{E}}_{n}(\dv(u^{n}(1+a^{n}))). \label{bAB}
\end{align}
Indeed, as $(a^{n},u^{n})_{n\geq1}$ is bounded in $E_{T}^{\alpha}$, by the product law in Besov spaces,
the right-hand side of \eqref{bAB} is bounded in $L^{2}_T(\dot{B}^{\frac{d}{2}-1}_{2,1}\cap\dot{B}^{\frac{d}{2}-1+\alpha}_{2,1})$.

For $(\bar{u}^{n})_{n\geq1}$, it suffices to prove that $(\partial_{t}\bar{u}^{n})_{n\geq1}$ is bounded in
$L^{2}_T(\dot{B}^{\frac{d}{2}-1+\beta}_{2,1}+\dot{B}^{\frac{d}{2}-2+\beta}_{2,1})$ for $\beta\in\{0,\alpha\}$.
We rewrite the equation for $\bar{u}^n$ as
\begin{align}
\partial_{t}\bar{u}^{n}= \,& -\mathbb{\dot{E}}_{n}(u^{n}\cdot\nabla u^{n})+\mathbb{\dot{E}}_{n}\Big(\frac{1}{1+a^{n}}\cdot\mathcal{A} \bar{u}^{n}\Big)
-\nabla\mathbb{\dot{E}}_{n}(G(a))\nonumber\\
&+\mathbb{\dot{E}}_{n}\Big(\frac{1}{1+a^{n}}\big(H^{n}\cdot\nabla H^{n}
-\frac{1}{2}|H^{n}|^{2}\big)\Big)-\mathbb{\dot{E}}_{n}(I(a^{n}\mathcal{A}u_{L}^{n}))\label{bAC}
\end{align}
with $\bar{u}^n|_{t=0}=0$. Because  $(u^{n})_{n\geq1}$, $(\bar{u}^{n})_{n\geq1}$, and $(H^{n})_{n\geq1}$
are bounded in $L^{2}_{T}(\dot{B}^{\frac{d}{2}}_{2,1}\cap\dot{B}^{\frac{d}{2}+\alpha}_{2,1})
\cap L^{\infty}_{T}(\dot{B}^{\frac{d}{2}-1}_{2,1}\cap\dot{B}^{\frac{d}{2}-1+\alpha}_{2,1})$,
and $\bar{a}^{n}$ is bounded in $L^{\infty}_T(\dot{B}^{\frac{d}{2}}_{2,1}\cap\dot{B}^{\frac{d}{2}+\alpha}_{2,1})$,
we easily deduce that the first four terms on the right-hand side of \eqref{bAC} are in $L^{2}_{T}(\dot{B}^{\frac{d}{2}-2}_{2,1}\cap\dot{B}^{\frac{d}{2}-2+\alpha}_{2,1})$
and that the last one is in $L^{\infty}_T(\dot{B}^{\frac{d}{2}-1}_{2,1}\cap\dot{B}^{\frac{d}{2}-1+\alpha}_{2,1})$ uniformly.

Similarly, the estimate for $(\bar{H}^{n})_{n\geq1}$ follows from the facts that $\bar{H}^{n}|_{t=0}=0$ and that
\begin{align}
 \partial_{t}\bar{H}^{n}=-\mathbb{\dot{E}}_{n}(u^{n}\cdot\nabla H^{n})+\mathbb{\dot{E}}_{n}(H^{n}\cdot\nabla u^{n})
+\nu\mathbb{\dot{E}}_{n}(\Delta H^{n})-\nu\mathbb{\dot{E}}_{n}(H^{n}\dv u^{n}). \label{bAD}
\end{align}
Indeed, as   $u^{n} $  and $H^{n}$ are bounded in $L^{2}_{T}(\dot{B}^{\frac{d}{2}}_{2,1}\cap\dot{B}^{\frac{d}{2}+\alpha}_{2,1})
\cap L^{\infty}_{T}(\dot{B}^{\frac{d}{2}-1}_{2,1}\cap\dot{B}^{\frac{d}{2}-1+\alpha}_{2,1})$,
we reduce that the right-hand side of \eqref{bAD} is bounded in    $L^{2}_T(\dot{B}^{\frac{d}{2}-2}_{2,1}\cap\dot{B}^{\frac{d}{2}-2+\alpha}_{2,1})$.
This is a simple consequence of the product and composition laws
for the homogeneous Besov spaces, as stated in Appendix \ref{AppA}.
\end{proof}

%
\subsubsection{Compactness and convergence of $(a^{n},\bar {u}^{n},\bar H^{n})$.}
By the results obtained in the above three steps, we begin to discuss the compactness and convergence of $(a^{n},\bar {u}^{n},\bar H^{n})$.
The arguments are very similar to that of \cite{BCD} on the isentropic Navier-Stokes equations. Here
we present them for the sake of completeness.
As in \cite{BCD}, we introduce a sequence $(\varphi_{p})_{p\geq 1}$ of smooth functions with values in $[0,1]$, supported in the ball $B(0,p+1)$ and equal to $1$ on $B(0,p)$.
According to the previous lemma, the sequence $(\bar{a}^{n})_{n\geq1}$ is bounded
in the space $\mathcal{C}^{\frac{1}{2}}([0,T];\dot{B}^{\frac{d}{2}-1}_{2,1}\cap\dot{B}^{\frac{d}{2}-1+\alpha}_{2,1})$.
Moreover, we have:

(i) By virtue of Proposition \ref{prop2.3},
$(\varphi_{p}\bar{a}^{n})_{n\geq1}$ is bounded in $$\mathcal{C}([0,T];B^{\frac{d}{2}}_{2,1}\cap B^{\frac{d}{2}+\alpha}_{2,1})
\cap\mathcal{C}^{\frac{1}{2}}([0,T];B^{\frac{d}{2}-1}_{2,1}\cap B^{\frac{d}{2}-1+\alpha}_{2,1});$$

(ii) According to Proposition \ref{prop2.4}, \\
\centerline{the map $ z\rightarrow \varphi_{p}z $ is compact from $B^{\frac{d}{2}+\alpha}_{2,1}$  to $ B^{\frac{d}{2}-1+\alpha}_{2,1}$; }

(iii) Since $\varphi_{p}\bar{a}^{n}$ is uniformly bounded in $\mathcal{C}^{\frac{1}{2}}([0,T];B^{\frac{d}{2}-1+\alpha}_{2,1})$, we have
$\varphi_{p}\bar{a}^{n}$ is uniformly equicontinuous with values in $B^{\frac{d}{2}-1+\alpha}_{2,1}$.

Therefore,   the Ascoli's theorem ensues that there exists some function $\bar{a}_{p}$ such that, up to a subsequence, \\
\centerline{$(\varphi_{p}\bar{a}^{n})_{n\geq1}$  converges to $\bar{a}_{p}$ in $\mathcal{C}^{\frac{1}{2}}([0,T];B^{\frac{d}{2}-1+\alpha}_{2,1})$.}\\
Using Cantor's diagonal process, we can then find a subsequence of $(\bar{a}^{n})_{n\geq1}$  (still denoted by $(\bar{a}^{n})_{n\geq 1}$)
such that for,  all $p\geq1$,\\
\centerline{$\varphi_{p}\bar{a}^{n}$ converges to $\bar{a}_{p}$ in $\mathcal{C}^{\frac{1}{2}}([0,T];B^{\frac{d}{2}-1+\alpha}_{2,1})$.}\\
As $\varphi_{p}\varphi_{p+1}=\varphi_{p}$, we have  $\bar{a}_{p}=\varphi_{p}\bar{a}_{p+1}$. Thus, we can easily deduce that
there exists some function $\bar{a}$ such that, for all $\varphi\in\mathcal{C}_{c}^{\infty}(\mathbb{R}^{d})$, \\
\centerline{$\varphi \bar{a}^{n}$ tends to $\varphi \bar{a}$ in  $\mathcal{C}^{\frac{1}{2}}([0,T];B^{\frac{d}{2}-1+\alpha}_{2,1})$.}

A similar argument gives us that
there exists a vector field $\bar{u}$ such that (up to extraction), for all $\varphi\in\mathcal{C}_{c}^{\infty}(\mathbb{R}^{d})$,\\
\centerline{$(\varphi\bar{u}^{n})_{n\geq1}$ converges to $\varphi \bar{u}$ in $\mathcal{C}^{\frac{1}{2}}([0,T];B^{\frac{d}{2}-2+\alpha}_{2,1})$,}\\
and there exists a vector field $\bar{H}$ such that (up to extraction), for all $\varphi\in\mathcal{C}_{c}^{\infty}(\mathbb{R}^{d})$,\\
\centerline{the sequence $(\varphi\bar{H}^{n})_{n\geq1} $ converges to $\varphi \bar{H}$ in $\mathcal{C}^{\frac{1}{2}}([0,T];B^{\frac{d}{2}-2+\alpha}_{2,1})$.}

Next, the uniform bounds supplied by the second step and the Fatou property together ensure that  $1+a$ is positive and
\begin{align*}
 (\bar{a},\bar{u},\bar{H})\in \tilde{L}^{\infty}_{T}(\dot{B}^{\frac{d}{2}}_{2,1}\cap\dot{B}^{\frac{d}{2}+\alpha}_{2,1})
\times \tilde{L}^{\infty}_{T}(\dot{B}^{\frac{d}{2}-1}_{2,1}\cap\dot{B}^{\frac{d}{2}-1+\alpha}_{2,1})
\times\tilde{L}^{\infty}_{T}(\dot{B}^{\frac{d}{2}-1}_{2,1}\cap\dot{B}^{\frac{d}{2}-1+\alpha}_{2,1}).
\end{align*}
We claim that $(\bar{u},\bar{H})$ also belongs to $\big(L^{1}_T(\dot{B}^{\frac{d}{2}+1}_{2,1}\cap{\dot{B}^{\frac{d}{2}+1+\alpha}_{2,1}})\big)^{d+d}$.
Indeed,  since  $(\bar{u}^{n})_{n\geq 1}$ is bounded in $L^{1}_T(\dot{B}^{\frac{d}{2}+1}_{2,1})\cap L^{1}_T(\dot{B}^{\frac{d}{2}+1+\alpha}_{2,1})$,
we know
that $\bar{u}$ belongs to the set $\mathcal{M}_{T}(\dot{B}^{\frac{d}{2}+1}_{2,1}\cap \dot{B}^{\frac{d}{2}+1+\alpha}_{2,1})$ of bounded measures on $[0,T]$ with
 values in the space $\dot{B}^{\frac{d}{2}+1}_{2,1}\cap \dot{B}^{\frac{d}{2}+1+\alpha}_{2,1}$, and that
 $$
 \int_{0}^{T}{\rm d}\|u\|_{\dot{B}^{\frac{d}{2}+1}_{2,1}\cap \dot{B}^{\frac{d}{2}+1+\alpha}_{2,1}}\leq C_{T},
 $$
 where $C_{T}$  stands for the right-hand side of \eqref{bb6}.

 It is clear that the same inequality holds for $\dot{\mathbb{E}}_{n}\bar{u}$, for all $n\geq1$. As
 $\bar{u}\in L^{\infty}_{T}(\dot{B}^{\frac{d}{2}-1}_{2,1}\linebreak \cap\dot{B}^{\frac{d}{2}-1+\alpha}_{2,1})$,
 we  have $\dot{\mathbb{E}}_{n}\bar{u}\in L^{1}_{T}(\dot{B}^{\frac{d}{2}+1}_{2,1}\cap\dot{B}^{\frac{d}{2}+1+\alpha}_{2,1})$.
 Thus,  we may write
 \begin{align*}
 \int_{0}^{T}\|\dot{\mathbb{E}}_{n}\bar{u}\|_{\dot{B}^{\frac{d}{2}+1}_{2,1}
 \cap\dot{B}^{\frac{d}{2}+1+\alpha}_{2,1}}{\rm d}t\leq C_{T}\,\,\,\text{for all}\,\,\,n\geq1.
 \end{align*}
Using the definition of the norm in $\dot{B}^{\frac{d}{2}+1}_{2,1}\cap\dot{B}^{\frac{d}{2}+1+\alpha}_{2,1}$, the above inequality implies that
\begin{align*}
\underset{N\rightarrow\infty}{\lim}\underset{|j|\leq N}{\sum}2^{j(\frac{d}{2}+1+\beta)}\int_{0}^{T}\|\dot{\Delta}_{j}\bar{u}\|_{L^{2}}{\rm d}t\leq C_{T},
\quad \beta\in\{0,\alpha\}.
\end{align*}
 Therefore, $\bar{u} \in L^{1}_{T}(\dot{B}^{\frac{d}{2}+1}_{2,1}\cap\dot{B}^{\frac{d}{2}+1+\alpha}_{2,1})$.
A similar argument implies that $\bar{H}\in L^{1}_{T}(\dot{B}^{\frac{d}{2}+1}_{2,1}\cap\dot{B}^{\frac{d}{2}+1+\alpha}_{2,1})$.

Interpolating   the above convergence results, we may get better convergence results for $(\bar{a}^{n},\bar{u}^{n},\linebreak\bar{H}^{n})$
and pass to the limit in \eqref{bb18}. Defining
\begin{align*}
 (a,u,H):=(\bar{a}+a_{0},u_{L}+\bar{u},\bar{H}+H_{0}),
\end{align*}
we thus get a solution
$(a,u,H)$ of \eqref{Aa5}-\eqref{Aa8} with the initial data $(a_{0},u_{0},H_{0})$. Using the equations of $(a,u,H)$  and the product laws, we also have
$(\partial_{t}+u\cdot\nabla)a\in L^{1}_{T}(\dot{B}^{\frac{d}{2}}_{2,1}\cap\dot{B}^{\frac{d}{2}+\alpha}_{2,1})$,
 $\partial_{t}u\in L^{1}_{T}(\dot{B}^{\frac{d}{2}-1}_{2,1}\cap\dot{B}^{\frac{d}{2}-1+\alpha}_{2,1})$,
and $\partial_{t}H\in L^{1}_{T}(\dot{B}^{\frac{d}{2}-1}_{2,1}\cap\dot{B}^{\frac{d}{2}-1+\alpha}_{2,1})$.
Proposition \ref{le4.1} therefore guarantees that $a\in \mathcal{\tilde{C}}_{T}(\dot{B}^{\frac{d}{2}}_{2,1}\cap\dot{B}^{\frac{d}{2}+\alpha}_{2,1})$. Obviously, we have
$u\in \mathcal{\tilde{C}}_{T}(\dot{B}^{\frac{d}{2}-1}_{2,1}\cap\dot{B}^{\frac{d}{2}-1+\alpha}_{2,1})$ and
$H\in \mathcal{\tilde{C}}_{T}(\dot{B}^{\frac{d}{2}-1}_{2,1}\cap\dot{B}^{\frac{d}{2}-1+\alpha}_{2,1})$.

\begin{Remark} \label{Rem21}
According to the properties of the semigroup for the heat kernel, we have the following estimates:
\begin{align*}
 &\|u_{L}\|_{L^{1}_{T}(\dot{B}^{s+2}_{2,1})}\leq C \sum_{j\in \mathbb{Z}}2^{js}(1-e^{-\kappa\bar{\nu}2^{2j}T})\|\dot{\Delta}_{j} u_{0}\|_{L^{2}},\\
&\|H_{L}\|_{L^{1}_{T}(\dot{B}^{s+2}_{2,1})}\leq C \sum_{j\in \mathbb{Z}}2^{js}(1-e^{-\kappa\nu2^{2j}T})\|\dot{\Delta}_{j} H_{0}\|_{L^{2}},
\end{align*}
where $\kappa$ is constant.
\end{Remark}

\begin{Remark}\label{Rem22}
Combining \eqref{bb3} and \eqref{bb4} with Remark \ref{Rem21} yields a rather explicit lower bound on the lifespan $T^{*}$ of the solution.
Indeed, using the fact that
$$
1-e^{-\beta T 2^{2j}}\leq (\beta T)^{\frac{\alpha}{2}}2^{j\alpha},
$$
we may find some constant $c$, depending only on $d, b_{*}, b^{*}, \alpha, \lambda, \mu$, and $\nu$, such that
\begin{align}
T^{*}\geq &\sup\bigg\{T\Big|\, T\leq\frac{c}{(1+A_{0}^{\alpha})^{2}\big((1+(H_{0}^{\alpha})^{2}+(U_{0}^{\alpha})^{2})(H_{0}^{\alpha})^{2}+U_{0}^{\alpha}\big)},
\ \  T\leq\frac{c}{(1+A_{0}^{\alpha})^{\frac{2}{\alpha}}},\nonumber\\
&\ \ \  \sum_{j\in \mathbb{Z}}2^{j(\frac{d}{2}-1+\alpha)}(1-e^{-\kappa\bar{\nu}2^{2j}T})\|\dot{\Delta}_{j}u_{0}\|_{L^{2}}\leq \frac{c}{(1+A_{0}^{\alpha})^{2}
(1+(1+U_{0}^{\alpha})(H_{0}^{\alpha})^{2})}
\bigg\}.\nonumber
\end{align}
\end{Remark}

\smallskip
\subsection{Uniqueness of the local solution}
In this subsection, we discuss the uniqueness of the local solution obtained in the previous subsection.
Let $(a^{1},u^{1},H^{1})$ and  $(a^{2},u^{2},H^{2})$ be two solutions in $E^{\alpha}_{T}$ of the system \eqref{Aa5}-\eqref{Aa8} with the same
initial data.
We assume, without loss of generality, that $(a^{2},u^{2},H^{2})$ is the solution constructed in the previous subsection satisfying
$$1+\underset{(t,x)\in[0,T]\times\mathbb{R}^{d}}{\inf}a^{2}(t,x)>0.$$
We need to prove that $(a^{1},u^{1},H^{1})=(a^{2},u^{2},H^{2})$ on $[0,T]\times \mathbb{R}^{d}$. To this end,  we shall estimate
$$(\hat a,\hat u ,\hat H ):=(a^{2}-a^{1},u^{2}-u^{1},H^{2}-H^{1})$$ with respect to a suitable norm.
A direction computation $(\hat a,\hat u ,\hat H )$ satisfies
\begin{equation}\label{bb19}
\left\{\begin{array}{l}
\partial_{t}\hat a +u^{2}\cdot\nabla \hat a +\hat G_{0}=0,\\
\partial_{t}\hat u +u^{2}\cdot\nabla \hat u + \hat u \cdot\nabla u^{1}-\frac{1}{1+a^{2}}\mathcal{A}\hat u = \hat G_{1}+\hat G_{2}+\hat G_{3},\\
\partial_{t}\hat H +u^{2}\cdot\nabla \hat H +\hat H \cdot\nabla u^{1}-\nu\Delta \hat H =\hat G_{4},
\end{array}\right.
\end{equation}
with
\begin{align}
\hat G_{0}&:=\hat u \cdot\nabla a^{1}+\hat a \dv u^{2}+(1+a^{1})\dv \hat u, \nonumber\\
\hat G_{1}&:=\big(I(a^{1})-I(a^{2})\mathcal{A}u^{1}\big),\qquad\hat G_{2}:=\nabla\big(G(a^{1})-G(a^{2})\big),\nonumber\\
\hat G_{3}&:=\frac{1}{1+a^{2}}H^{2}\nabla H^{2}-\frac{1}{1+a^{1}}H^{1}\nabla H^{1}\nonumber\\
&\quad \ \ -\frac{1}{1+a^{2}}\nabla |H^{2}|^{2}+\frac{1}{1+a^{1}}\nabla |H^{1}|^{2},\nonumber\\
\hat G_{4}&:=H^{2}\cdot\nabla\hat u -\hat u \cdot\nabla H^{1}-\dv u^{1}\cdot \hat H -\dv \hat u  \cdot H^{2}.\nonumber
\end{align}
Due to the hyperbolic structure of the mass equation, we could not avoid a loss of one derivative in the stability
estimates (the term $\hat u \cdot \nabla a^1$ in the first equation of \eqref{bb19} can not be better than
$L^{\infty}_T(\dot{B}^{\frac{d}{2}-1+\alpha}_{2,1})$ since we only know that $\nabla a^1\in L^{\infty}_T(\dot{B}^{\frac{d}{2}-1+\alpha}_{2,1}\cap\dot{B}^{\frac{d}{2}-1}_{2,1})$).
In addition,  the strong coupling in the equations for $(\hat a ,\hat u ,\hat H)$ implies that this loss of one derivative also results in
a loss of one derivative when bounding $\hat u $ and $\hat H $. Hence, we expect to prove uniqueness in the following function space,
\begin{align}
F^{\alpha}_{T} :=\, &\mathcal{C}([0,T];\dot{B}^{\frac{d}{2}-1+\alpha}_{2,1})
\times \big(\mathcal{C}([0,T];\dot{B}^{\frac{d}{2}-2+\alpha}_{2,1})\cap L^{1}_{T}(\dot{B}^{\frac{d}{2}+\alpha}_{2,1})\big)^{d}\nonumber\\
&\times\big(\mathcal{C}([0,T];\dot{B}^{\frac{d}{2}-2+\alpha}_{2,1})\cap L^{1}_{T}(\dot{B}^{\frac{d}{2}+\alpha}_{2,1})\big)^{d}.\nonumber
\end{align}
First, we show that $(\hat a ,\hat u ,\hat H) $ belongs to $F^{\alpha}_{T}$.
For $\hat a $, a similar argument to that in the proof of  Lemma \ref{lemma2.2} implies that $\partial_{t}\bar{a}^{i}\in L^{2}_T(\dot{B}^{\frac{d}{2}-1+\alpha}_{2,1}) (i=1,2)$.
Hence, we get $\bar{a}^{i}\in\mathcal{C}^{\frac{1}{2}}([0,T];\dot{B}^{\frac{d}{2}-1+\alpha}_{2,1})$.
To deal with $\hat u $, we introduce $\bar{u}^{i} := u^{i}-\bar{u}_{L}$, where $\bar{u}_{L}$ is the solution of
$$\partial_{t}\bar{u}_{L}-\mathcal{A}\bar{u}_{L}=-\nabla G(a_{0}),\qquad\bar{u}_L|_{t=0}=u_{0}.$$
Obviously, we have $\bar{u}^{i}|_{t=0}=0$ and
\begin{align}
\partial_{t}\bar{u}^{i}=\, &\mathcal{A}\bar{u}^{i}-I(a^{i})\mathcal{A}u^{i}-u^{i}\nabla u^{i}
+\frac{1}{a^{i}}\Big(H^{i}\cdot\nabla H^{i}-\frac{1}{2}\nabla|H^{i}|^{2}\Big)-\nabla\big( G(a^{i})-G(a_{0})\big).\nonumber
\end{align}
Since $\bar{a}^{i}\in L^{\infty}_{T}(\dot{B}^{\frac{d}{2}-1+\alpha}_{2,1})$ and $(a^{i},u^{i},H^{i})\in E^{\alpha}_{T}$,
the right-hand side of the above equation belongs to $ L^{2}_T(\dot{B}^{\frac{d}{2}-2+\alpha}_{2,1})$. Hence, $\bar{u}^{i}$ belongs to
$\mathcal{C}([0,T];\dot{B}^{\frac{d}{2}-2+\alpha}_{2,1})$. Similarly, since $H^{i}$ satisfies
$$\partial_{t}H^{i}=H^{i}\cdot\nabla u^{i}-u^{i}\cdot\nabla H^{i}-\nu\Delta H^{i}-(\dv u^{i})H^{i},$$
we easily deduce that $\partial_{t}H^{i}\in  L^{2}_T(\dot{B}^{\frac{d}{2}-2+\alpha}_{2,1})$. Thus,
$H^{i}\in\mathcal{C}([0,T];\dot{B}^{\frac{d}{2}-2+\alpha}_{2,1})$  and hence we   conclude that
 $(\hat a ,\hat u ,\hat H )\in F^{\alpha}_{T}$.

Next, applying Proposition \ref{prop4.1} to the first equation of \eqref{bb19}, we get, for all $\bar{T}\in [0,T]$, that
$$\|\hat a \|_{L^{\infty}_{\bar{T}}(\dot{B}^{\frac{d}{2}-1+\alpha}_{2,1})}
\leq\exp\Big\{C\|u^{2}\|_{L^{1}_{\bar{T}}(\dot{B}^{\frac{d}{2}+1}_{2,1})}\Big\}\|\hat G_{0}\|_{L^{1}_{\bar{T}}(\dot{B}^{\frac{d}{2}-1+\alpha}_{2,1})}
.$$
By  Proposition \ref{prop2.2}, an easy computation gives
\begin{align}
\|\hat G_{0}\|_{\dot{B}^{\frac{d}{2}-1+\alpha}_{2,1}}\leq &C\Big\{\|\hat u \|_{\dot{B}^{\frac{d}{2}}_{2,1}}\|\nabla a^{1}\|_{\dot{B}^{\frac{d}{2}-1+\alpha}_{2,1}}
+\|\dv u^{2}\|_{\dot{B}^{\frac{d}{2}}_{2,1}}\|\hat a \|_{\dot{B}^{\frac{d}{2}-1+\alpha}_{2,1}}\nonumber\\
&\qquad +\Big(1+\|a^{1}\|_{\dot{B}^{\frac{d}{2}}_{2,1}}\Big)\|\hat u \|_{\dot{B}^{\frac{d}{2}+\alpha}_{2,1}}\Big\}.\nonumber
\end{align}
Hence, using Gronwall's lemma and interpolation implies that there exists a constant $C_{T}$, independent of $\bar{T}$,
such that
\begin{align}\label{bb20}
\|\hat a \|_{L^{\infty}_{\bar{T}}(\dot{B}^{\frac{d}{2}-1+\alpha}_{2,1})}\leq C_{T}
\bigg(\|\hat u \|_{L^{1}_{\bar{T}}(\dot{B}^{\frac{d}{2}+\alpha}_{2,1})}+\|\hat u \|_{L^{\infty}_{\bar{T}}(\dot{B}^{\frac{d}{2}-2+\alpha}_{2,1})}\bigg).
\end{align}
Similarly, applying Proposition \ref{prop4.3} to the third equation of \eqref{bb19} gives,  for all $\bar{T}\in [0,T]$, that
\begin{align}
& \|\hat H \|_{L^{\infty}_{\bar{T}}(\dot{B}^{\frac{d}{2}-2+\alpha}_{2,1})}+\|\hat H \|_{L^{1}_{\bar{T}}(\dot{B}^{\frac{d}{2}+\alpha}_{2,1})}\nonumber\\
& \qquad \leq C\exp\bigg\{C(\|u^{1}\|_{L^{1}_{\bar{T}}(\dot{B}^{\frac{d}{2}+1}_{2,1})}+\|u^{2}\|_{L^{1}_{\bar{T}}(\dot{B}^{\frac{d}{2}+1}_{2,1})})\bigg\}
 \|\hat G_{4}\|_{L^{1}_{\bar{T}}(\dot{B}^{\frac{d}{2}-2+\alpha}_{2,1})}.\nonumber
\end{align}
By    Proposition \ref{prop2.2} and  Lemma \ref{le2.2}, we have
\begin{align}
\|\hat G_{4}\|_{L^{1}_{\bar{T}}(\dot{B}^{\frac{d}{2}-2+\alpha}_{2,1})}\leq\,& C
\bigg (\|H^{2}\|_{L^{\infty}_{\bar{T}}(\dot{B}^{\frac{d}{2}-1}_{2,1})}\|\hat u \|_{L^{1}_{\bar{T}}(\dot{B}^{\frac{d}{2}+\alpha}_{2,1})}\nonumber\\
 & \qquad+\|H^{1}\|_{L^{1}_{\bar{T}}(\dot{B}^{\frac{d}{2}+1}_{2,1})}\|\hat u \|_{L^{\infty}_{\bar{T}}(\dot{B}^{\frac{d}{2}-2+\alpha}_{2,1})}\nonumber\\
&\qquad+\int_{0}^{T}\|u^{1}\|_{\dot{B}^{\frac{d}{2}+1}_{2,1}}\|\hat H \|_{\dot{B}^{\frac{d}{2}-2+\alpha}_{2,1}}\mathrm{d}t\bigg).\nonumber
\end{align}
Once again,    using Gronwall's lemma and interpolation, we obtain that there exists some constant $C_{T}$, independent of $\bar{T}$,
such that
\begin{align}\label{bb21}
\|\hat H \|_{L^{\infty}_{\bar{T}}(\dot{B}^{\frac{d}{2}-2+\alpha}_{2,1})}+\|\hat H \|_{L^{1}_{\bar{T}}(\dot{B}^{\frac{d}{2}+\alpha}_{2,1})}
\leq C_{T}\bigg(\|\hat u \|_{L^{1}_{\bar{T}}(\dot{B}^{\frac{d}{2}+\alpha}_{2,1})}
+\|\hat u \|_{L^{\infty}_{\bar{T}}(\dot{B}^{\frac{d}{2}-2+\alpha}_{2,1})}\bigg).
\end{align}
Similarly, applying Proposition \ref{prop4.2} to the second equation of \eqref{bb19} gives
\begin{align}
& \|\hat u \|_{L^{1}_{\bar{T}}(\dot{B}^{\frac{d}{2}+\alpha}_{2,1})}+\|\hat u \|_{L^{\infty}_{\bar{T}}(\dot{B}^{\frac{d}{2}-2+\alpha}_{2,1})}\nonumber\\
&\quad \leq C \exp\bigg\{C\int_{0}^{\bar{T}}
(\|u^{1}\|_{\dot{B}^{\frac{d}{2}+1}_{2,1}}+\|u^{2}\|_{\dot{B}^{\frac{d}{2}+1}_{2,1}}+\|a^{2}\|^{\frac{2}{\alpha}}_{\dot{B}^{\frac{d}{2}+\alpha}_{2,1}})
\mathrm{d}t\bigg\}
\sum_{i=1}^{3}\|\hat G_{i}\|_{L^{1}_{\bar{T}}(\dot{B}^{\frac{d}{2}-2+\alpha}_{2,1})}.\nonumber
\end{align}
Because $\dot{B}^{\frac{d}{2}}_{2,1}(\mathbb{R}^{d})\hookrightarrow \mathcal{C}(\mathbb{R}^{d})$, we have $a^{1}\in\mathcal{C}([0,T]\times\mathbb{R}^{d})$.
Hence, for sufficiently small $\bar{T}$, $a^{1}$ also satisfies \eqref{bb5}. Therefore, applying Proposition \ref{prop2.2} and Lemmas \ref{le2.2} and \ref{le2.3} yields
\begin{align}
\|\hat G_{1}\|_{L^{1}_{\bar{T}}(\dot{B}^{\frac{d}{2}-2+\alpha}_{2,1})}
&\leq C(1+\|a^{1}\|_{L^{\infty}_{\bar{T}}(\dot{B}^{\frac{d}{2}}_{2,1})}+\|a^{2}\|_{L^{\infty}_{\bar{T}}(\dot{B}^{\frac{d}{2}}_{2,1})})\nonumber\\
&\quad \times\|\hat a \|_{L^{\infty}_{\bar{T}}(\dot{B}^{\frac{d}{2}-1+\alpha}_{2,1})}\|u^{1}\|_{L^{1}_{\bar{T}}(\dot{B}^{\frac{d}{2}+1+\alpha}_{2,1})},
\nonumber\\
\|\hat G_{2}\|_{L^{1}_{\bar{T}}(\dot{B}^{\frac{d}{2}-2+\alpha}_{2,1})}
&\leq C \bar{T}(1+\|a^{1}\|_{L^{\infty}_{\bar{T}}(\dot{B}^{\frac{d}{2}}_{2,1})}+\|a^{2}\|_{L^{\infty}_{\bar{T}}(\dot{B}^{\frac{d}{2}}_{2,1})})
\|\hat a \|_{L^{\infty}_{\bar{T}}(\dot{B}^{\frac{d}{2}-1+\alpha}_{2,1})},\nonumber\\
\|\hat G_{3}\|_{L^{1}_{\bar{T}}(\dot{B}^{\frac{d}{2}-2+\alpha}_{2,1})}
&\leq C(1+\|a^{1}\|_{L^{\infty}_{\bar{T}}(\dot{B}^{\frac{d}{2}}_{2,1})}+\|a^{2}\|_{L^{\infty}_{\bar{T}}(\dot{B}^{\frac{d}{2}}_{2,1})})\nonumber\\
&\quad\times\bigg\{\Big(\sum_{i=1}^{2}\|H^{i}\|_{L^{\infty}_{\bar{T}}(\dot{B}^{\frac{d}{2}-1}_{2,1})} \|H^{i}\|_{L^{1}_{\bar{T}}(\dot{B}^{\frac{d}{2}+1}_{2,1})}\Big)
\|\hat a \|_{L^{\infty}_{\bar{T}}(\dot{B}^{\frac{d}{2}-1+\alpha}_{2,1})}\nonumber\\
&\qquad \ \   +\Big(\|H^{1}\|^{\frac{1}{2}}_{L^{\infty}_{\bar{T}}(\dot{B}^{\frac{d}{2}-1}_{2,1})}
\|H^{1}\|^{\frac{1}{2}}_{L^{1}_{\bar{T}}(\dot{B}^{\frac{d}{2}+1}_{2,1})}+\|H^{2}\|_{L^{1}_{\bar{T}}(\dot{B}^{\frac{d}{2}+1}_{2,1})}\Big)\nonumber\\
&\qquad\qquad \times\Big(\|\hat H \|_{L^{\infty}_{\bar{T}}(\dot{B}^{\frac{d}{2}-2+\alpha}_{2,1})}
+\|\hat H \|_{L^{1}_{\bar{T}}(\dot{B}^{\frac{d}{2}+\alpha}_{2,1})}\Big)\bigg\}.\nonumber
\end{align}
Therefore, there exists a constant $C_{T}$, independent of $\bar{T}$,
such that
\begin{align}
& \!\!\!\!\!\!\!\!\!\!\!\!\!
\|\hat u \|_{L^{1}_{\bar{T}}(\dot{B}^{\frac{d}{2}+\alpha}_{2,1})}+\|\hat u \|_{L^{\infty}_{\bar{T}}(\dot{B}^{\frac{d}{2}-2+\alpha}_{2,1})}\nonumber\\
\leq \, & C_{T}\bigg\{\Big(\bar{T}+\|u^{1}\|_{L^{1}_{\bar{T}}(\dot{B}^{\frac{d}{2}+1}_{2,1})} +\sum_{i=1}^{2}\|H^{i}\|_{L^{1}_{\bar{T}}(\dot{B}^{\frac{d}{2}+\alpha}_{2,1})}\Big)
\|\hat a \|_{L^{\infty}_{\bar{T}}(\dot{B}^{\frac{d}{2}-1+\alpha}_{2,1})}\nonumber\\
&+(\|H^{2}\|_{L^{1}_{\bar{T}}(\dot{B}^{\frac{d}{2}+1}_{2,1})}
+\|H^{1}\|^{\frac{1}{2}}_{L^{1}_{\bar{T}}(\dot{B}^{\frac{d}{2}+1}_{2,1})})
(\|\hat H \|_{L^{\infty}_{\bar{T}}(\dot{B}^{\frac{d}{2}-2+\alpha}_{2,1})}+
\|\hat H \|_{L^{1}_{\bar{T}}(\dot{B}^{\frac{d}{2}+\alpha}_{2,1})})\bigg\}.\nonumber
\end{align}
Note that the factors
 $$
 \bar{T}+\|u^{1}\|_{L^{1}_{\bar{T}}(\dot{B}^{\frac{d}{2}+1}_{2,1})}
+\sum_{i=1}^{2}\|H^{i}\|_{L^{1}_{\bar{T}}(\dot{B}^{\frac{d}{2}+\alpha}_{2,1})}  \ \ \text{and} \ \
 \|H^{2}\|_{L^{1}_{\bar{T}}(\dot{B}^{\frac{d}{2}+1}_{2,1})}
+\|H^{1}\|^{\frac{1}{2}}_{L^{1}_{\bar{T}}(\dot{B}^{\frac{d}{2}+1}_{2,1})}
$$
decay to 0 when $\bar{T}$ goes to zero. Hence, plugging the inequalities \eqref{bb20} and \eqref{bb21} into the
above inequality, we conclude that $(\hat a ,\hat u ,\hat H )\equiv 0$ on a small time interval
$[0,\bar{T}]$. In order to show that  $\bar{T}=T$, we introduce the set
$$I:=\left\{t\in [0,T]\big| (a^{1},u^{1},H^{1})\equiv (a^{2},u^{2},H^{2})\,\, \text{on}\,\, [0,t]\right\}.$$
Obviously, $I$ is a nonempty closed subset of $[0,T]$. In addition, the above arguments may be carried over
 to any $t\in I\cap [0,T),$ which ensures that $I$ is also an open subset of $[0,T]$. Therefore, $I\equiv [0,T]$.
\medskip

The proof of Theorem \ref{ThA} is now completed.

\subsection{A continuation criterion}

\begin{Proposition}\label{PropB}
Under the hypotheses of Theorem \ref{ThA}, assume that the system\eqref{Aa5}-\eqref{Aa8} has a solution $(a,u,H)$
on $[0,T)\times \mathbb{R}^{d}$ which belongs to $E^{\alpha}_{T'}$ for all $T'<T$ and satisfies
\begin{align*}
 &  a\in L^{\infty}_T(\dot{B}^{\frac{d}{2}+\alpha}_{2,1}\cap \dot{B}^{\frac{d}{2}}_{2,1}),\quad 1+\underset{(t,x)\in[0,T)\times\mathbb{R}^{d}}{\inf}
a(t,x)>0;\\
&\int_{0}^{T}\|\nabla u\|_{\dot{B}^{\frac{d}{2}}_{2,1}}\mathrm{d}t<\infty,\quad \int_{0}^{T}\|\nabla H\|_{\dot{B}^{\frac{d}{2}}_{2,1}}\mathrm{d}t<\infty.
\end{align*}
Then  there exists a $T^{*}>T$ such that $(a,u,H)$ may be extended
to a solution of \eqref{Aa4}-\eqref{Aa8} on $[0,T^{*}]\times\mathbb{R}^{d}$ which belongs to $E^{\alpha}_{T^{*}}$.
\end{Proposition}

\begin{proof}
Note that $(u,H)$ satisfies
\[\left\{\begin{array}{ll}
\partial_{t} u+u\cdot\nabla u-\frac{1}{1+a}\mathcal{A}u+\nabla G(a)=\frac{1}{1+a}\big(H\cdot\nabla H-\frac{1}{2}\nabla |H|^{2}\big), & u|_{t=0}=u_{0},\\
\partial_{t} H +u\cdot\nabla H-H\cdot\nabla u-\nu\Delta H=-(\dv u)H,& H|_{t=0}=H_{0}.
\end{array}\right.\]
By taking the  same arguments as those used in the proof of Proposition \ref{prop4.3}, we easily see that
there exists  a universal constant $\kappa$ such that
\begin{align}
 \|u\|_{\tilde{L}^{\infty}_{t}(\dot{B}^{\frac{d}{2}-1+\beta}_{2,1})}&+\kappa b_{*}\mu \|u\|_{\tilde{L}^{1}_{t}(\dot{B}^{\frac{d}{2}+1+\beta}_{2,1})}\nonumber\\
\leq\,&  \|u_{0}\|_{\dot{B}^{\frac{d}{2}-1+\beta}_{2,1}}+\|a\|_{\tilde{L}^{1}_{t}(\dot{B}^{\frac{d}{2}+\beta}_{2,1})}\nonumber\\
&  +\int_{0}^{t}(1+\|a\|_{\dot{B}^{\frac{d}{2}}_{2,1}})\|H\|_{\dot{B}^{\frac{d}{2}+1}_{2,1}}\|H\|_{\dot{B}^{\frac{d}{2}-1+\beta}_{2,1}}\mathrm{d}\tau\nonumber\\
&+C\int_{0}^{t}\|u\|_{\dot{B}^{\frac{d}{2}+1}_{2,1}}\|u\|_{\dot{B}^{\frac{d}{2}-1+\beta}_{2,1}}\mathrm{d}\tau
+C\int_{0}^{t}\|c\|^{\frac{2}{\alpha}}_{\dot{B}^{\frac{d}{2}+\alpha}_{2,1}}\|u\|_{\dot{B}^{\frac{d}{2}-1+\beta}_{2,1}}\mathrm{d}\tau, \nonumber\\
\|H\|_{\tilde{L}^{\infty}_{t}(\dot{B}^{\frac{d}{2}-1+\beta}_{2,1})}&+\kappa \nu\|H\|_{\tilde{L}^{1}_{t}(\dot{B}^{\frac{d}{2}+1+\beta}_{2,1})}\nonumber\\
\leq\,&  \|H_{0}\|_{\dot{B}^{\frac{d}{2}-1+\beta}_{2,1}}+\|(\dv u) H\|_{\tilde{L}^{1}_{t}(\dot{B}^{\frac{d}{2}-1+\beta}_{2,1})}\nonumber\\
&+C\int_{0}^{t}\|\nabla u\|_{\dot{B}^{\frac{d}{2}}_{2,1}}\|H\|_{\dot{B}^{\frac{d}{2}-1+\beta}_{2,1}}\mathrm{d}\tau,\nonumber
\end{align}
with $c=-I(a),\beta\in\{0,\alpha\}$.
Adding the above two inequalities  and applying Gronwall's inequality, we then obtain,  for all $\beta\in\{0,\alpha\}$ and $T'<T$, that
\begin{align}
& \|u\|_{\tilde{L}^{\infty}_{T'}(\dot{B}^{\frac{d}{2}-1+\beta}_{2,1})}+\mu \|u\|_{\tilde{L}^{1}_{T'}(\dot{B}^{\frac{d}{2}+1+\beta}_{2,1})}
+\|H\|_{\tilde{L}^{\infty}_{T'}(\dot{B}^{\frac{d}{2}-1+\beta}_{2,1})}+ \nu\|H\|_{\tilde{L}^{1}_{T'}(\dot{B}^{\frac{d}{2}+1+\beta}_{2,1})}\nonumber\\
&\quad  \leq C\left( \|u_{0}\|_{\dot{B}^{\frac{d}{2}-1+\beta}_{2,1}}
+\|H_{0}\|_{\dot{B}^{\frac{d}{2}-1+\beta}_{2,1}}+\|a\|_{\tilde{L}^{1}_{T'}(\dot{B}^{\frac{d}{2}+\beta}_{2,1})}\right)\nonumber\\
& \qquad \times \exp\bigg\{ C\Big(\int_{0}^{T'}\|\nabla u\|_{\dot{B}^{\frac{d}{2}}_{2,1}}+(1+\|a\|_{\dot{B}^{\frac{d}{2}}_{2,1}})\|\nabla H\|_{\dot{B}^{\frac{d}{2}}_{2,1}}+
\|a\|^{\frac{2}{\alpha}}_{\dot{B}^{\frac{d}{2}+\alpha}_{2,1}}\Big)\mathrm{d}t\bigg\}\nonumber
\end{align}
for some constant $C$ depending only on $d,\alpha$, and the viscosity coefficients.
Hence, $(u,H)$ is bounded in $\tilde{L}^{\infty}_{T}(\dot{B}^{\frac{d}{2}-1}_{2,1}\cap\dot{B}^{\frac{d}{2}-1+\alpha}_{2,1}).$

Replacing $\|\dot{\Delta}_{j}u_{0}\|_{L^{2}}$ and $\|\dot{\Delta}_{j}H_{0}\|_{L^{2}}$ by $\|\dot{\Delta}_{j}u\|_{L^{\infty}_{T}(L^{2})} $ and
$\|\dot{\Delta}_{j}H\|_{L^{\infty}_{T}(L^{2})}$ in Remark \ref{Rem22}, respectively, we get an $\epsilon >0$ such that, for any $T'\in[0,T),$ the system
\eqref{Aa5}-\eqref{Aa8} with initial data $(a(T'),u(T'),H(T'))$ has a solution for $t\in [0,\epsilon].$ Taking $T'=T-\epsilon/2$ and using the fact that
the solution $(a,u,H)$ is unique on $[0,T)$, we thus get a continuation of $(a,u,H)$ beyond $T.$
\end{proof}

\bigskip

\section{Low Mach Number Limit for the Compressible MHD Equations}\label{lowlimit}

In this section we shall study the low Mach number limit of the compressible MHD equations  \eqref{Aa17}-\eqref{Aa19} for the
  local solution obtained in Theorem \ref{ThA}. The main strategy is to apply the Leray projector on the system to divide it into the
  incompressible part and  acoustic   part and then estimate the acoustic   part and the difference
  of the incompressible part with the incompressible MHD equations. We shall follow and adapt some ideas developed by
  Danchin \cite{DC} on the isentropic Navier-Stokes equations. Before we begin our proof, we briefly describe the   process as follows.
Firstly, we use the  dispersive inequalities of linear wave equations to bound a
suitable norm  of $(b^{\epsilon},\mathcal{P}^{\bot}u^{\epsilon})$,
and this bound will be controlled by the norm of $(b^{\epsilon},u^{\epsilon})$ in $E^{\frac{d}{2}+\beta}_{\epsilon,T}$ times some positive power of $\epsilon$.
Secondly, by means of estimates for the non-stationary incompressible MHD equations (see Proposition \ref{AProp3}) and paradifferential calculus,
 we get a priori bounds for $\epsilon^{-\beta}(\mathcal{P}u^{\epsilon}-v,H^{\epsilon}-B)$
in term of $\|(b^{\epsilon},u^{\epsilon},H^{\epsilon})\|_{E^{\frac{d}{2}+\beta}_{\epsilon,T}}$ and $\|(v,B)\|_{F^{\frac{d}{2}+\beta}_{T}}$.
Thirdly,  we show uniform bounds for $\|(b^{\epsilon},u^{\epsilon},H^{\epsilon})\|_{E^{\frac{d}{2}+\beta}_{\epsilon,T}}$
in term of $(v,B)$ and the initial data. We then use a bootstrap argument  to close the estimates on the first three steps. Finally, we use
 a continuity argument to complete our proof.

\begin{proof}[Proof of Theorem \ref{ThB}]

Throughout the proof we shall use the following notations:
\begin{align}
   w^{\epsilon}:=\,& \mathcal{P}u^{\epsilon}-v , \qquad  B^{\epsilon}:= H^{\epsilon}-B ,\nonumber\\
  X_{\beta}(T):=\,&\|b^{\epsilon}\|_{L_{T}^{1}(\tilde{B}^{\frac{d}{2}+\beta,1}_{\epsilon})}
+\|b^{\epsilon}\|_{L_{T}^{\infty}(\tilde{B}^{\frac{d}{2}+\beta,\infty}_{\epsilon})}\nonumber\\
&+\|\mathcal{P}^{\bot}u^{\epsilon}\|_{L_{T}^{1}(\dot{B}^{\frac{d}{2}+1+\beta}_{2,1})}
+\|\mathcal{P}^{\bot}u^{\epsilon}\|_{L_{T}^{\infty}(\dot{B}^{\frac{d}{2}-1+\beta}_{2,1})},\nonumber\\
V_{\beta}(T):=\,&\|v\|_{L_{T}^{1}(\dot{B}^{\frac{d}{2}+1+\beta}_{2,1})}
+\|v\|_{L_{T}^{\infty}(\dot{B}^{\frac{d}{2}-1+\beta}_{2,1})}\nonumber\\
&+\|B\|_{L_{T}^{1}(\dot{B}^{\frac{d}{2}+1+\beta}_{2,1})}
+\|B\|_{L_{T}^{\infty}(\dot{B}^{\frac{d}{2}-1+\beta}_{2,1})},\nonumber\\
W_{\beta}(T):=\,&\|w^{\epsilon}\|_{L_{T}^{1}(\dot{B}^{\frac{d}{2}+1+\beta}_{2,1})}
+\|w^{\epsilon}\|_{L_{T}^{\infty}(\dot{B}^{\frac{d}{2}-1+\beta}_{2,1})}\nonumber\\
&+\|B^{\epsilon}\|_{L_{T}^{1}(\dot{B}^{\frac{d}{2}+1+\beta}_{2,1})}
+\|B^{\epsilon}\|_{L_{T}^{\infty}(\dot{B}^{\frac{d}{2}-1+\beta}_{2,1})},\nonumber\\
Y^{p}_{\beta}(T):=\,& \|b^{\epsilon}\|_{L_{T}^{p}(\dot{B}^{\beta-1+\frac{1}{p}}_{\infty,1})}
+\|\mathcal{P}^{\bot}u^{\epsilon}\|_{L_{T}^{p}(\dot{B}^{\beta-1+\frac{1}{p}}_{\infty,1})} \qquad  \   \text{if}  \ \ d=3 \ \ \text{and}\ 2< p<\infty,\nonumber\\
Y_{\beta}(T):=\,& \|b^{\epsilon}\|_{L_{T}^{4}(\dot{B}^{\beta-\frac{3}{4}}_{\infty,1})}
+\|\mathcal{P}^{\bot}u^{\epsilon}\|_{L_{T}^{4}(\dot{B}^{\beta-\frac{3}{4}}_{\infty,1})}  \qquad  \   \text{if}  \ \  d=2. \nonumber
\end{align}
We shall also use the notations $P_{\beta}(T):=V_{\beta}(T)+W_{\beta}(T)$ and
$$
X_{0}^{\beta}:=\|b_{0}\|_{\tilde{B}^{\frac{d}{2}
+\beta,\infty}_{\epsilon}}+\|\mathcal{P}^{\bot}u_{0}\|_{\dot{B}^{\frac{d}{2}-1+\beta}_{2,1}}
+\|H_{0}\|_{\dot{B}^{\frac{d}{2}-1+\beta}_{2,1}}.
$$
In our arguments below the time $T$ will sometimes be omitted and $\beta$ will always stand for 0 or $\alpha$.


\subsection{Dispersive estimates for $(b^{\epsilon},\mathcal{P}^{\bot}u^{\epsilon})$.}

We first recall the dispersive inequalities for the following (reduced) system of acoustics:
\begin{equation}\label{Cc-1}
\left\{\begin{array}{l}
\partial_{t}b+\epsilon^{-1}\Lambda \Psi=F, \\
\partial_{t}\Psi-\epsilon^{-1}\Lambda b=G, \\
(b,\Psi)|_{t=0}=(b_{0},\Psi_{0})(x),\quad x\in \mathbb{R}^d.
\end{array}
\right.
\end{equation}
Recall that $\Lambda$ is defined as $\Lambda:=\sqrt{-\Delta}$ in Appendix \ref{AppA}.

\begin{Proposition}[\cite{DC}]\label{bprop3.1}
Let $(b,\Psi)$ be a solution of \eqref{Cc-1}. Then, for any $s\in \mathbb{R}$ and positive $T$ \emph{(}possibly infinite\emph{)}, the following
estimate holds:
$$\|(b,\Psi)\|_{\widetilde{L}^{r}_{T}(\dot{B}^{s+d(\frac{1}{p}-\frac{1}{2})\frac{1}{r}}_{p,1})}
\leq C\epsilon^{\frac{1}{r}}\|(b_{0},\Psi_{0})\|_{\dot{B}^{s}_{2,1}}+\epsilon^{1+\frac{1}{r}-\frac{1}{\overline{r}'}}
\|(F,G)\|_{\widetilde{L}^{\overline{r}'}_{T}(\dot{B}^{s+d(\frac{1}{\overline{p}'}-\frac{1}{2})+\frac{1}{\overline{r}}})}$$
with
\begin{align*}
&p\geq2,\quad \frac{2}{r}\leq \min\{1,\gamma(p)\}, \quad (r,p,d)\neq(2,\infty,3),\\
& \bar{p}\geq2, \quad \frac{2}{\bar{r}}\leq \min\{1,\gamma(\bar{p})\}, \quad (\bar{r},\bar{p},d)\neq(2,\infty,3),
\end{align*}
where
$$\gamma(q):=(d-1)\big(\frac{1}{2}-\frac{1}{q}\big),\quad
\frac{1}{\bar{p}}+\frac{1}{\bar{p}'}=1,  \quad \frac{1}{\bar{r}}+\frac{1}{\bar{r}'}=1.
$$
\end{Proposition}

It is easy to check that $(b^{\epsilon},\mathcal{P}^{\bot}u^{\epsilon})$ satisfy the system:
\begin{equation}\label{Cc-2}
\left\{\begin{array}{l}
\partial_{t}b^{\epsilon}+\epsilon^{-1}\dv \mathcal{P}^{\bot}u^{\epsilon}=F^{\epsilon},\\
\partial_{t}\mathcal{P}^{\bot}u^{\epsilon}+\epsilon^{-1}\nabla b^{\epsilon}=G^{\epsilon},
\end{array}
\right.
\end{equation}
with $F^{\epsilon}:=-\dv(b^{\epsilon}u^{\epsilon})$ and
$$
G^{\epsilon}:=-\mathcal{P}^{\bot}\Big(u^{\epsilon}\cdot\nabla u^{\epsilon}+
\frac{\mathcal{A}u^{\epsilon}}{1+\epsilon b^{\epsilon}}+\frac{K(\epsilon b^{\epsilon})}{\epsilon}\nabla b^{\epsilon}
+\frac{1}{1+\epsilon b^{\epsilon}}\big(H^{\epsilon}\cdot\nabla H^{\epsilon}-\frac{1}{2}\nabla(|H^{\epsilon}|^{2})\big)\Big),
$$
where $K(z):=\frac{P'(1+z)}{1+z}-1$\,\,(hence $K(0)=0$).
Obviously, the dispersive estimates stated in Proposition \ref{bprop3.1} are aslo true for the system \eqref{Cc-2} since
$b^{\epsilon}$ and $d^{\epsilon}:=\Lambda^{-1}\dv \mathcal{P}^{\bot}u^{\epsilon}$ satisfy \eqref{Cc-1}
with source terms $F^{\epsilon}$ and $\Lambda^{-1}\dv G^{\epsilon}$, and $\Lambda^{-1}\dv$ is a homogeneous
multiplier of degree 0.
Hence, we have,
for $d=3$ and $2<p<+\infty$,
\begin{align}\label{Cc-3}
Y_{\alpha}^{p}
\leq C\epsilon^{\frac{1}{p}}\left(\|(b_{0},\mathcal{P}u_{0})\|_{\dot{B}^{\frac{1}{2}+\alpha}_{2,1}}+
\|(F^{\epsilon},\Lambda^{-1}\dv G^{\epsilon})\|_{L^{1}_{T}(\dot{B}^{\frac{1}{2}+\alpha}_{2,1})}\right).
\end{align}
and for $d= 2$,
\begin{align}\label{Cc-4}
Y_{\alpha}
\leq C\epsilon^{\frac{1}{4}}\left(\|(b_{0},\mathcal{P}u_{0})\|_{\dot{B}^{\alpha}_{2,1}}+
\|(F^{\epsilon},\Lambda^{-1}\dv G^{\epsilon})\|_{L^{1}_{T}(\dot{B}^{\alpha}_{2,1})}\right).
\end{align}

From Propositions  \ref{prop2.1} and \ref{prop2.2}, we easily conclude that, for  $d= 3$ or $d=2$,
\begin{align}
\|\dv(b^{\epsilon}u^{\epsilon})\|_{L^{1}_{T}(\dot{B}^{\frac{d}{2}-1+\alpha}_{2,1})}&\leq C
\Big(\|b^{\epsilon}\|_{L^{2}_{T}(\dot{B}^{\frac{d}{2}}_{2,1})}\|u^{\epsilon}\|_{L^{2}_{T}(\dot{B}^{\frac{d}{2}+\alpha}_{2,1})}\nonumber\\
&\qquad\quad +\|b^{\epsilon}\|_{L^{2}_{T}(\dot{B}^{\frac{d}{2}+\alpha}_{2,1})}\|u^{\epsilon}\|_{L^{2}_{T}(\dot{B}^{\frac{d}{2}}_{2,1})}\Big)\nonumber\\
&\leq C\big(X_{0}(X_{\alpha}+P_{\alpha})+X_{\alpha}(X_{0}+P_{0})\big),\nonumber\\
\|\mathcal{P}^{\bot}(u^{\epsilon}\cdot\nabla u^{\epsilon})\|_{L^{1}_{T}(\dot{B}^{\frac{d}{2}-1+\alpha}_{2,1})}&\leq C
\|u^{\epsilon}\|_{L^{2}_{T}(\dot{B}^{\frac{d}{2}}_{2,1})}\|\nabla u^{\epsilon}\|_{L^{2}_{T}(\dot{B}^{\frac{d}{2}-1+\alpha}_{2,1})}\nonumber\\
&\leq C(X_{0}+P_{0})(X_{\alpha}+P_{\alpha}),\nonumber\\
\|\mathcal{P}^{\bot}(I(\epsilon b^{\epsilon})\cdot \mathcal{A}u^{\epsilon})\|_{L^{1}_{T}(\dot{B}^{\frac{d}{2}-1+\alpha}_{2,1})}&\leq C \epsilon
\|b^{\epsilon}\|_{L^{\infty}_{T}(\dot{B}^{\frac{d}{2}}_{2,1})}\|\mathcal{A}u^{\epsilon}\|_{L^{1}_{T}(\dot{B}^{\frac{d}{2}-1+\alpha}_{2,1})}\nonumber\\
&\leq CX_{0}(X_{\alpha}+P_{\alpha}),\nonumber\\
\|K(\epsilon b^{\epsilon})\nabla b^{\epsilon}\|_{L^{1}_{T}(\dot{B}^{\frac{d}{2}-1+\alpha}_{2,1})}
&\leq C \epsilon \|b^{\epsilon}\|_{L^{2}_{T}(\dot{B}^{\frac{d}{2}}_{2,1})}\|\nabla b^{\epsilon}\|_{L^{2}_{T}(\dot{B}^{\frac{d}{2}-1+\alpha}_{2,1})}\nonumber\\
&\leq C \epsilon X_{0}X_{\alpha}.\nonumber
\end{align}
Since
\begin{align}
& \!\!\!\!\!\!\!\!\!\!\!\!\!\!\!\!\!\!
\mathcal{P}^{\bot}\Big(\frac{1}{1+\epsilon b^{\epsilon}}\big(H^{\epsilon}\cdot\nabla H^{\epsilon}-\frac{1}{2}\nabla (|H^{\epsilon}|^{2})\big) \Big)\nonumber\\
=\,& \mathcal{P}^{\bot}\Big(H^{\epsilon}\cdot\nabla H^{\epsilon}-\frac{1}{2}\nabla ( |H^{\epsilon}|^{2})\Big)
+\mathcal{P}^{\bot}\Big(I(\epsilon b^{\epsilon})\big(H^{\epsilon}\cdot\nabla H^{\epsilon}-\frac{1}{2}\nabla( |H^{\epsilon}|^{2})\big)\Big),\nonumber
\end{align}
we have
\begin{align}
\|\mathcal{P}^{\bot}(H^{\epsilon}\cdot\nabla H^{\epsilon})\|_{L^{1}_{T}(\dot{B}^{\frac{d}{2}-1+\alpha}_{2,1})}&\leq C
\|H^{\epsilon}\|_{L^{2}_{T}(\dot{B}^{\frac{d}{2}}_{2,1})}\|\nabla H^{\epsilon}\|_{L^{2}_{T}(\dot{B}^{\frac{d}{2}-1+\alpha}_{2,1})}\nonumber\\
&\leq C P_{0}P_{\alpha},\nonumber\\
\|\mathcal{P}^{\bot}(\nabla (|H^{\epsilon}|^{2}))\|_{L^{1}_{T}(\dot{B}^{\frac{d}{2}-1+\alpha}_{2,1})}&\leq C  P_{0}P_{\alpha},\nonumber\\
\Big\|\mathcal{P}^{\bot}\Big(I(\epsilon b^{\epsilon}) H^{\epsilon}\cdot\nabla H^{\epsilon}\Big)\Big\|_{L^{1}_{T}(\dot{B}^{\frac{d}{2}-1+\alpha}_{2,1})}&\leq C
\epsilon\|b^{\epsilon}\|_{L^{\infty}_{T}(\dot{B}^{\frac{d}{2}}_{2,1})}\| H^{\epsilon}\cdot\nabla H^{\epsilon}\|_{L^{1}_{T}(\dot{B}^{\frac{d}{2}-1+\alpha}_{2,1})}\nonumber\\
&\leq CX_{0} P_{0}P_{\alpha},\nonumber\\
\Big\|\mathcal{P}^{\bot}\Big(I(\epsilon b^{\epsilon}) |\nabla H^{\epsilon}|^{2}\Big)\Big\|_{L^{1}_{T}(\dot{B}^{\frac{d}{2}-1+\alpha}_{2,1})}&\leq CX_{0} P_{0}P_{\alpha}.\nonumber
\end{align}
Plugging the above inequalities into \eqref{Cc-3} or \eqref{Cc-4}, we conclude that
for $d=3$ and $2<p< \infty$,
\begin{align}\label{Cc-5}
Y_{\alpha}^{p}\leq C \epsilon^{\frac{1}{p}}\Big(X_{\alpha}^{0}+X_{\alpha}+(X_{0}+P_{0})(X_{\alpha}+P_{\alpha})+X_{0}P_{0}P_{\alpha}\Big),
\end{align}
and for $d=2$,
\begin{align}\label{Cc-6}
Y_{\alpha}\leq C \epsilon^{\frac{1}{4}}\Big(X_{\alpha}^{0}+X_{\alpha}+(X_{0}+P_{0})(X_{\alpha}+P_{\alpha})+X_{0}P_{0}P_{\alpha}\Big).
\end{align}

\subsection{Estimates for $(w^{\epsilon},B^{\epsilon})$.}

From the system \eqref{Aa18}-\eqref{Aa19} and \eqref{Aa14}-\eqref{Aa15}, a direct computation gives
\begin{align}
\left\{
\begin{array}{l}
\partial_{t}w^{\epsilon}-\mu\Delta w^{\epsilon}+\mathcal{P}(A^{\epsilon}\cdot\nabla w^{\epsilon})+\mathcal{P}(w^{\epsilon}\cdot\nabla A^{\epsilon})
-\mathcal{P}(B^{\epsilon}\cdot\nabla B)\\
\qquad -\mathcal{P}(B\cdot\nabla B^{\epsilon})
=-\mathcal{P}M^{\epsilon}+\mathcal{P} L^{\epsilon},\\
\partial_{t}B^{\epsilon}-\nu\Delta B^{\epsilon}+\mathcal{P}(A^{\epsilon}\cdot\nabla B^{\epsilon})-\mathcal{P}(B^{\epsilon}\cdot\nabla A^{\epsilon})
+\mathcal{P}(w^{\epsilon}\cdot\nabla B)\\
\qquad -\mathcal{P}(B\cdot\nabla w^{\epsilon})
=-\mathcal{P}Q^{\epsilon},\\
(w^{\epsilon},B^{\epsilon})|_{t=0}=(0,0),
\end{array}
\right.\nonumber
\end{align}
where
\begin{align}
A^{\epsilon}&:=\mathcal{P}^{\bot}u^{\epsilon}+v,\nonumber\\
M^{\epsilon}&:=\mathcal{P}^{\bot}u^{\epsilon}\cdot\nabla v+v\cdot\nabla\mathcal{P}^{\bot}u^{\epsilon}
+w^{\epsilon}\cdot\nabla w^{\epsilon}+I(\epsilon b^{\epsilon})\mathcal{A} u^{\epsilon},\nonumber\\
L^{\epsilon}&:=B^{\epsilon}\cdot\nabla B^{\epsilon}+I(\epsilon b^{\epsilon})(H^{\epsilon}\cdot \nabla H^{\epsilon}-\frac{1}{2}|\nabla H^{\epsilon}|^{2}),\nonumber\\
Q^{\epsilon}&:=B^{\epsilon}\cdot\nabla w^{\epsilon}- w^{\epsilon}\cdot\nabla B^{\epsilon}
+B\cdot\nabla \mathcal{P}^{\bot}u^{\epsilon}- \mathcal{P}^{\bot}u^{\epsilon}\cdot\nabla B-(\dv \mathcal{P}^{\bot}u^{\epsilon})H^{\epsilon}.\nonumber
\end{align}
Applying Proposition \ref{AProp3} with $s=d/2-1+\beta$ yields
\begin{align}
W_{\beta}\leq\,& C\left(\|M^{\epsilon}\|_{L^{1}_{T}(\dot{B}^{\frac{d}{2}-1+\beta})}+\|L^{\epsilon}\|_{L^{1}_{T}(\dot{B}^{\frac{d}{2}-1+\beta})}
+\|Q^{\epsilon}\|_{L^{1}_{T}(\dot{B}^{\frac{d}{2}-1+\beta})}\right)\nonumber\\
&\times
\exp\bigg\{ C \int_{0}^{T}\big(\|\nabla A^{\epsilon}\|_{\dot{B}^{\frac{d}{2}}_{2,1}}+\|\nabla B\|_{\dot{B}^{\frac{d}{2}}_{2,1}}\big)\mathrm{d}t\bigg\}\nonumber\\
\leq\,& C  e^{C(V_{0}+X_{0})}\left(\|M^{\epsilon}\|_{L^{1}_{T}(\dot{B}^{\frac{d}{2}-1+\beta})}+\|L^{\epsilon}\|_{L^{1}_{T}(\dot{B}^{\frac{d}{2}-1+\beta})}
+\|Q^{\epsilon}\|_{L^{1}_{T}(\dot{B}^{\frac{d}{2}-1+\beta})}\right).\label{Cc-7}
\end{align}

We now bound $M^{\epsilon}, L^{\epsilon}$, and $Q^{\epsilon}$. First, we readily have
\begin{align}
&\|w^{\epsilon}\cdot\nabla w^{\epsilon}\|_{L^{1}_{T}(\dot{B}^{\frac{d}{2}-1+\beta}_{2,1})}
\leq C\|w^{\epsilon}\|_{L^{2}_{T}(\dot{B}^{\frac{d}{2}}_{2,1})}\|\nabla w^{\epsilon}\|_{L^{2}_{T}(\dot{B}^{\frac{d}{2}-1+\beta}_{2,1})}
\leq C W_{0}W_{\beta},\label{Cc-8}\\
&\|B^{\epsilon}\cdot\nabla B^{\epsilon}\|_{L^{1}_{T}(\dot{B}^{\frac{d}{2}-1+\beta}_{2,1})}
\leq C\|B^{\epsilon}\|_{L^{2}_{T}(\dot{B}^{\frac{d}{2}}_{2,1})}\|\nabla B^{\epsilon}\|_{L^{2}_{T}(\dot{B}^{\frac{d}{2}-1+\beta}_{2,1})}
\leq C W_{0}W_{\beta},\label{Cc-9}\\
&\|B^{\epsilon}\cdot\nabla w^{\epsilon}\|_{L^{1}_{T}(\dot{B}^{\frac{d}{2}-1+\beta}_{2,1})}
\leq C\|B^{\epsilon}\|_{L^{2}_{T}(\dot{B}^{\frac{d}{2}}_{2,1})}\|\nabla w^{\epsilon}\|_{L^{2}_{T}(\dot{B}^{\frac{d}{2}-1+\beta}_{2,1})}
\leq C W_{0}W_{\beta},\label{Cc-11}\\
&\|w^{\epsilon}\cdot\nabla B^{\epsilon}\|_{L^{1}_{T}(\dot{B}^{\frac{d}{2}-1+\beta}_{2,1})}
\leq C\|w^{\epsilon}\|_{L^{2}_{T}(\dot{B}^{\frac{d}{2}}_{2,1})}\|\nabla B^{\epsilon}\|_{L^{2}_{T}(\dot{B}^{\frac{d}{2}-1+\beta}_{2,1})}
\leq C W_{0}W_{\beta}.\label{Cc-10}
\end{align}
Next, by interpolation and  (ii) in Remark \ref{re2.2}, we have
\begin{align}
\|b^{\epsilon}\|_{\dot{B}^{\frac{d}{2}}_{2,1}}& \leq \|b^{\epsilon}\|^{\alpha}_{\dot{B}^{\frac{d}{2}+\alpha-1}_{2,1}}
\|b^{\epsilon}\|^{1-\alpha}_{\dot{B}^{\frac{d}{2}+\alpha}_{2,1}}
\leq\epsilon^{\alpha-1}\|b^{\epsilon}\|_{\tilde{B}^{\frac{d}{2}+\alpha,\infty}_{\epsilon}},\label{Cc-12}\\
& \!\! \!\! \!\! \!\!  \!\! \!\! \!\!\!\!\!\!\!\! \|I(\epsilon b^{\epsilon})\mathcal{A}u^{\epsilon}\|_{L^{1}_{T}(\dot{B}^{\frac{d}{2}-1+\beta}_{2,1})}\nonumber\\
&\leq C\epsilon \| b^{\epsilon}\|_{L^{\infty}_{T}(\dot{B}^{\frac{d}{2}}_{2,1})} \|\mathcal{A}u^{\epsilon}\|_{L^{1}_{T}(\dot{B}^{\frac{d}{2}-1+\beta}_{2,1})}\nonumber\\
&\leq C\epsilon^{\alpha}\|b^{\epsilon}\|_{L^{\infty}_{T}(\tilde{B}^{\frac{d}{2}+\alpha,\infty}_{\epsilon})}
\|u^{\epsilon}\|_{L^{1}_{T}(\dot{B}^{\frac{d}{2}+1+\beta}_{2,1})}\nonumber\\
&\leq C \epsilon^{\alpha}X_{\alpha}(V_{\beta}+W_{\beta}+X_{\beta}),\label{Cc-13}\\
& \!\! \!\! \!\! \!\!  \!\! \!\! \!\!\!\!\!\!\!\! \|I(\epsilon b^{\epsilon})H^{\epsilon}\cdot\nabla H^{\epsilon}\|_{L^{1}_{T}(\dot{B}^{\frac{d}{2}-1+\beta}_{2,1})}\nonumber\\
&\leq C\epsilon \| b^{\epsilon}\|_{L^{\infty}_{T}(\dot{B}^{\frac{d}{2}}_{2,1})}
\|H^{\epsilon}\cdot\nabla H^{\epsilon}\|_{L^{1}_{T}(\dot{B}^{\frac{d}{2}-1+\beta}_{2,1})}\nonumber\\
&\leq C\epsilon^{\alpha}\|b^{\epsilon}\|_{L^{\infty}_{T}(\tilde{B}^{\frac{d}{2}+\alpha,\infty}_{\epsilon})}
\|H^{\epsilon}\|_{L^{2}_{T}(\dot{B}^{\frac{d}{2}}_{2,1})}\|\nabla H^{\epsilon}\|_{L^{2}_{T}(\dot{B}^{\frac{d}{2}-1+\beta}_{2,1})}\nonumber\\
&\leq C\epsilon^{\alpha} X_{\alpha}(W_{0}+V_{0})(W_{\beta}+V_{\beta}),\label{Cc-14}\\
& \!\! \!\! \!\! \!\!  \!\! \!\! \!\!\!\!\!\!\!\! \|I(\epsilon b^{\epsilon})\cdot\nabla| H^{\epsilon}|^{2}\|_{L^{1}_{T}(\dot{B}^{\frac{d}{2}-1+\beta}_{2,1})}
 \leq C\epsilon^{\alpha} X_{\alpha}(W_{0}+V_{0})(W_{\beta}+V_{\beta}).\label{Cc-15}
\end{align}

We will deal with the other terms in $M^{\epsilon}, L^{\epsilon}$, and $Q^{\epsilon}$ according to $d=3$ or $d=2$.
We first consider the case:\emph{ $d=3$ and $p_{\alpha}=1+\frac{1}{2\alpha}$}.
Since $\mathcal{P}^{\bot}u^{\epsilon}$ is small in $L^{2}([0,T];\dot{B}^{0}_{\infty,1})$, using interpolation and embedding, we have
\begin{align}
\|\mathcal{P}^{\bot}u^{\epsilon}\|_{L^{2}_{T}(\dot{B}^{0}_{\infty,1})}
&\leq \|\mathcal{P}^{\bot}u^{\epsilon}\|^{\frac{1}{2}-\alpha}_{L^{1}_{T}(\dot{B}^{1+\alpha}_{\infty,1})}
\|\mathcal{P}^{\bot}u^{\epsilon}\|^{\frac{1}{2}+\alpha}_{L^{(1+2\alpha)/2\alpha}_{T}(\dot{B}^{\alpha-1/(1+2\alpha)}_{\infty,1})}\nonumber\\
&\leq C\|\mathcal{P}^{\bot}u^{\epsilon}\|^{\frac{1}{2}-\alpha}_{L^{1}_{T}(\dot{B}^{5/2+\alpha}_{2,1})}
\|\mathcal{P}^{\bot}u^{\epsilon}\|^{\frac{1}{2}+\alpha}_{L^{(1+2\alpha)/2\alpha}_{T}(\dot{B}^{\alpha-1/(1+2\alpha)}_{\infty,1})}\nonumber\\
&\leq C \epsilon^{\alpha}(X_{\alpha}+\epsilon^{-\frac{2\alpha}{1+2\alpha}}Y_{\alpha}^{p_{\alpha}}).\label{Cc-16}
\end{align}
From \eqref{Cc-16}  we expect to gain some smallness for $\mathcal{P}^{\bot}u^{\epsilon}\cdot\nabla v$, $v\cdot\nabla \mathcal{P}^{\bot}u^{\epsilon}$,
$\mathcal{P}^{\bot}u^{\epsilon}\cdot\nabla B$, $B\cdot\nabla \mathcal{P}^{\bot}u^{\epsilon},$ and $(\dv \mathcal{P}^{\bot}u^{\epsilon}) H^{\epsilon}$,
 by means of a judicious application of paradifferential calculus. For $\mathcal{P}^{\bot}u^{\epsilon}\cdot\nabla v$, we shall use the following decomposition (with $\eta <1$ to be fixed hereafter):
$$\mathcal{P}^{\bot}u^{\epsilon}\cdot\nabla v=\underbrace{\sum_{q\in \mathbb{Z}}\dot{\Delta}_{q}\mathcal{P}^{\bot}u^{\epsilon}\cdot S_{q-1+[\log_{2}\eta]}\nabla v}_{T_{1}}
+\underbrace{\sum_{q\in \mathbb{Z}}\dot{\Delta}_{q}\nabla v\cdot S_{q+2-[\log_{2}\eta]}\mathcal{P}^{\bot}u^{\epsilon}}_{T_{2}},
$$
which may be seen as a slight modification of Bony decomposition.

 Recall that, for any $k\in \mathbb{Z},$ we have
 $$\|\dot{S}_{j}\nabla v\|_{L^{\infty}}\leq C2^{2j}\|\nabla v\|_{\dot{B}^{-2}_{\infty,1}}.$$
 Therefore,
\begin{align}
&\!\!\!\!\!\!\!\!\!\!\!\!\!\!\!\!\!\!\!\ \!\!\!\!
\|\dot{\Delta}_{q}\mathcal{P}^{\bot}u^{\epsilon}\cdot S_{q-1+[\log_{2}\eta]}\nabla  v\|_{L^{2}}\nonumber\\
\leq  & C
\|S_{q-1+[\log_{2}\eta]}\nabla  v\|_{L^{\infty}}\|\dot{\Delta}_{q}\mathcal{P}^{\bot}u^{\epsilon}\|_{L^{2}}\nonumber\\
\leq &C \eta^{2}2^{-q(\frac{d}{2}+\beta-1)}
\|\nabla v\|_{\dot{B}^{-2}_{\infty,1}}\big(2^{q(\frac{d}{2}+\beta+1)}\|\dot{\Delta}_{q}\mathcal{P}^{\bot}u^{\epsilon}\|_{L^{2}}\big).\nonumber
\end{align}
As the function $\dot{\Delta}_{q}\mathcal{P}^{\bot}u^{\epsilon}\cdot S_{q-1+[\log_{2}\eta]}\nabla  v\|_{L^{2}}$
is spectrally supported in dyadic annuli  $2^{q}\mathcal{C}(0,R_{1}, \linebreak R_{2})$ with $R_{1}$ and $R_{2}$ independent of $\eta$,
Lemma \ref{le2.4} yields
\begin{align}\label{Cc-17}
\|T_{1}\|_{\dot{B}^{\frac{d}{2}-1+\beta}}\leq C \eta^{2}\|\nabla v\|_{\dot{B}^{\frac{d}{2}-2}_{2,1}}
\|\mathcal{P}^{\bot}u^{\epsilon}\|_{\dot{B}^{\frac{d}{2}+1+\beta}_{2,1}}.
\end{align}
Next, according to the properties of quasi-orthogonality of the dyadic decomposition, we have, for all $k\in \mathbb{Z}$,
$$\dot{\Delta}_{k}T_{2}=\sum_{q\geq k-2+[\log_{2}\eta]}\dot{\Delta}_{k}\big(\dot{S}_{q+2-[\log_{2}\eta]}\mathcal{P}^{\bot}u^{\epsilon}
\cdot\dot{\Delta}_{q}\nabla v\big).$$
Therefore,
\begin{align}
& \!\!\!\!\!\!\!\!\!\!\!\!\!\! \!\!\!\! 2^{k(\frac{d}{2}+\beta-1)}\|\dot{\Delta}_{k}T_{2}\|_{L^{2}}\nonumber\\
\leq\, & C\|\mathcal{P}^{\bot}u^{\epsilon}\|_{L^{\infty}}
\sum_{q\geq k-2+[\log_{2}\eta]}2^{(k-q)(\frac{d}{2}+\beta-1)}2^{q(\frac{d}{2}+\beta-1)}\|\dot{\Delta}_{q}\nabla v\|_{L^{2}}\nonumber\\
\leq\,& C \eta^{1-\beta-\frac{d}{2}}\|\mathcal{P}^{\bot}u^{\epsilon}\|_{L^{\infty}}\|\nabla v\|_{\dot{B}^{\frac{d}{2}+\beta-1}_{2,1}},\nonumber
\end{align}
from which it follows that
\begin{align}
\|T_{2}\|_{\dot{B}^{\frac{d}{2}+\beta-1}_{2,1}}\leq C \eta^{1-\beta-\frac{d}{2}}\| v\|_{\dot{B}^{\frac{d}{2}+\beta}_{2,1}}
\|\mathcal{P}^{\bot}u^{\epsilon}\|_{L^{\infty}}.\label{Cc-18}
\end{align}
By \eqref{Cc-17}, \eqref{Cc-18}, and  H\"{o}lder's inequality, we thus get
\begin{align}
\|\mathcal{P}^{\bot}u^{\epsilon}\cdot\nabla v\|_{L^{1}_{T}(\dot{B}^{\frac{d}{2}+\beta-1}_{2,1})}\leq \,& C\Big(
 \eta^{2}\|v\|_{L^{\infty}_{T}(\dot{B}^{\frac{d}{2}-1}_{2,1})}
\|\mathcal{P}^{\bot}u^{\epsilon}\|_{L^{1}_{T}(\dot{B}^{\frac{d}{2}+1+\beta}_{2,1})}\nonumber\\
&\quad +\eta^{1-\beta-\frac{d}{2}}\| v\|_{L^{2}_{T}(\dot{B}^{\frac{d}{2}+\beta}_{2,1})}
\|\mathcal{P}^{\bot}u^{\epsilon}\|_{L^{2}_{T}(L^{\infty})}\Big).\nonumber
\end{align}
Since $\dot{B}^{0}_{\infty,1}\hookrightarrow L^{\infty}$, by choosing $\eta=\epsilon^{\frac{2\alpha}{2+d+2\beta}}$ and using \eqref{Cc-16}, we can
now conclude that
\begin{align}\label{Cc-19}
\|\mathcal{P}^{\bot}u^{\epsilon}\cdot\nabla v\|_{L^{1}_{T}(\dot{B}^{\frac{d}{2}+\beta-1}_{2,1})}\leq C
\epsilon^{\frac{4\alpha}{2+d+2\beta}}\big(V_{0}X_{\beta}+V_{\beta}(X_{\alpha}+\epsilon^{-\frac{2\alpha}{1+2\alpha}}Y^{p_{\alpha}}_{\alpha})\big).
\end{align}
Similarly,
\begin{align}\label{Cc-20}
\|\mathcal{P}^{\bot}u^{\epsilon}\cdot\nabla B\|_{L^{1}_{T}(\dot{B}^{\frac{d}{2}+\beta-1}_{2,1})}\leq C
\epsilon^{\frac{4\alpha}{2+d+2\beta}}\big(V_{0}X_{\beta}+V_{\beta}(X_{\alpha}+\epsilon^{-\frac{2\alpha}{1+2\alpha}}Y^{p_{\alpha}}_{\alpha})\big).
\end{align}

The term $v\cdot \nabla \mathcal{P}^{\bot}u^{\epsilon}$ may be treated similarly. In fact, using the decomposition
$$v\cdot\nabla\mathcal{P}^{\bot}u^{\epsilon}
=\underbrace{\sum_{q\in \mathbb{Z}}\dot{\Delta}_{q}\nabla\mathcal{P}^{\bot}u^{\epsilon}\cdot S_{q-1+[\log_{2}\eta]} v}_{\tilde{T}_{1}}
+\underbrace{\sum_{q\in \mathbb{Z}}\dot{\Delta}_{q} v\cdot S_{q+2-[\log_{2}\eta]}\nabla\mathcal{P}^{\bot}u^{\epsilon}}_{\tilde{T}_{2}}
$$
and following the previous argument, we readily get
\begin{align}
&\|\tilde{T}_{1}\|_{L^{1}_{T}(\dot{B}^{\frac{d}{2}-1+\beta}_{2,1})}\leq C \eta \|v\|_{L^{\infty}_{T}(\dot{B}^{\frac{d}{2}-1}_{2,1})}
\|\nabla \mathcal{P}^{\bot}u^{\epsilon}\|_{L^{1}_{T}(\dot{B}^{\frac{d}{2}+\beta}_{2,1})},\nonumber\\
&\|\tilde{T}_{2}\|_{L^{1}_{T}(\dot{B}^{\frac{d}{2}-1+\beta}_{2,1})}\leq C \eta^{-\frac{d}{2}-\beta} \|v\|_{L^{2}_{T}(\dot{B}^{\frac{d}{2}+\beta}_{2,1})}
\|\nabla \mathcal{P}^{\bot}u^{\epsilon}\|_{L^{1}_{T}(\dot{B}^{-1}_{2,1})}.\nonumber
\end{align}
Choosing $\eta=\epsilon^{\frac{2\alpha}{2+d+2\beta}},$ we conclude that
\begin{align}\label{Cc-21}
\|v\cdot\nabla\mathcal{P}^{\bot}u^{\epsilon}\|_{L^{1}_{T}(\dot{B}^{\frac{d}{2}+\beta-1}_{2,1})}\leq C
\epsilon^{\frac{2\alpha}{2+d+2\beta}}\big(V_{0}X_{\beta}+V_{\beta}(X_{\alpha}+\epsilon^{-\frac{2\alpha}{1+2\alpha}}Y^{p_{\alpha}}_{\alpha})\big).
\end{align}
Similarly,
\begin{align}\label{Cc-22}
\|B\cdot\nabla\mathcal{P}^{\bot}u^{\epsilon}\|_{L^{1}_{T}(\dot{B}^{\frac{d}{2}+\beta-1}_{2,1})}\leq C
\epsilon^{\frac{4\alpha}{2+d+2\beta}}\big(V_{0}X_{\beta}+V_{\beta}(X_{\alpha}+\epsilon^{-\frac{2\alpha}{1+2\alpha}}Y^{p_{\alpha}}_{\alpha})\big).
\end{align}

To deal with the term $\dv(\mathcal{P}^{\bot}u^{\epsilon})H^{\epsilon}$, we introduce the decomposition
\begin{align*}
\dv(\mathcal{P}^{\bot}u^{\epsilon})H^{\epsilon}
=\,&\underbrace{\sum_{q\in \mathbb{Z}}\dot{\Delta}_{q}\dv\mathcal{P}^{\bot}u^{\epsilon}\cdot S_{q-1+[\log_{2}\eta]} H^{\epsilon}}_{\bar{T}_{1}}\\
& +\underbrace{\sum_{q\in \mathbb{Z}}\dot{\Delta}_{q} H^{\epsilon}\cdot S_{q+2-[\log_{2}\eta]}\dv\mathcal{P}^{\bot}u^{\epsilon}}_{\bar{T}_{2}}.
\end{align*}
Following the previous argument, we readily get
\begin{align}
&\|\bar{T}_{1}\|_{L^{1}_{T}(\dot{B}^{\frac{d}{2}-1+\beta}_{2,1})}\leq C \eta \|H^{\epsilon}\|_{L^{\infty}_{T}(\dot{B}^{\frac{d}{2}-1}_{2,1})}
\|\dv \mathcal{P}^{\bot}u^{\epsilon}\|_{L^{1}_{T}(\dot{B}^{\frac{d}{2}+\beta}_{2,1})},\nonumber\\
&\|\bar{T}_{2}\|_{L^{1}_{T}(\dot{B}^{\frac{d}{2}-1+\beta}_{2,1})}\leq C \eta^{-\frac{d}{2}-\beta} \|H^{\epsilon}\|_{L^{2}_{T}(\dot{B}^{\frac{d}{2}+\beta}_{2,1})}
\|\dv \mathcal{P}^{\bot}u^{\epsilon}\|_{L^{2}_{T}(\dot{B}^{-1}_{2,1})}.\nonumber
\end{align}
Choosing $\eta=\epsilon^{\frac{2\alpha}{2+d+2\beta}},$ we conclude that
\begin{align}\label{Cc-23}
\|\dv(\mathcal{P}^{\bot}u^{\epsilon})H^{\epsilon}\|_{L^{1}_{T}(\dot{B}^{\frac{d}{2}+\beta-1}_{2,1})}
\leq\,& C
\epsilon^{\frac{2\alpha}{2+d+2\beta}}\Big((W_{0}+V_{0})X_{\beta}\nonumber\\
&+(W_{\beta}+V_{\beta})(X_{\alpha}+\epsilon^{-\frac{2\alpha}{1+2\alpha}}Y^{p_{\alpha}}_{\alpha})\Big).
\end{align}

Now  we consider the case:  $d=2$.
For $\mathcal{P}^{\bot}u^{\epsilon}$, we have the following estimate:
\begin{align}
\|\mathcal{P}^{\bot}u^{\epsilon}\|_{L^{\frac{1}{1-3\alpha}}_{T}(\dot{B}^{2-10\alpha}_{\frac{2}{1-4\alpha},1})}
&\leq \|\mathcal{P}^{\bot}u^{\epsilon}\|^{1-4\alpha}_{L^{1}_{T}(\dot{B}^{2+\alpha}_{2,1})}
\|\mathcal{P}^{\bot}u^{\epsilon}\|^{4\alpha}_{L^{4}_{T}(\dot{B}^{\alpha-\frac{3}{4}}_{\infty,1})}\nonumber\\
&\leq C \epsilon^{\alpha}(X_{\alpha}+\epsilon^{-\frac{1}{4}}Y_{\alpha})\label{Cc-24}
\end{align}
with $\alpha\in(0,\frac{1}{6}]$.
In this part of proof, we need the following refined Bony decomposition:
\begin{align}\label{bonya}
\mathcal{P}^{\bot}u^{\epsilon}\cdot\nabla v= &\sum_{j\in \mathbb{Z}}\dot{S}_{j-1+[\log_{2}\eta]}\mathcal{P}^{\bot}u^{\epsilon}\nabla\Delta_{j}v\nonumber\\
&+ \sum_{j\in \mathbb{Z}}\big(\dot{S}_{j-1}-\dot{S}_{j-1+[\log_{2}\eta]}\big)\mathcal{P}^{\bot}u^{\epsilon}\nabla\Delta_{j}v\nonumber\\
&+\dot{R}(\mathcal{P}^{\bot}u^{\epsilon},\nabla v)+\dot{T}_{\nabla v}\mathcal{P}^{\bot}u^{\epsilon}.
\end{align}
As in the proof of case $d=3$ and \eqref{Cc-24}, we have
\begin{align}
&\!\!\!\!\!\!\!\!\!\!\!\!\!\! \!\!\!\! \bigg \|\sum_{j\in \mathbb{Z}}\dot{S}_{j-1+[\log_{2}\eta]}\mathcal{P}^{\bot}u^{\epsilon}\nabla\Delta_{j}v\bigg\|_{L^{1}_{T}(\dot{B}^{\beta}_{2,1})}\nonumber\\
\leq\,& C \eta \|\mathcal{P}^{\bot}u^{\epsilon}\|_{L_{T}^{\infty}(\dot{B}^{-1}_{\infty,1})}\|\nabla v\|_{L_{T}^{1}(\dot{B}^{1+\beta}_{2,1})}\nonumber\\
\leq\,& C \eta X_{0}V_{\beta},\nonumber\\
&\!\!\!\!\!\!\!\!\!\!\!\!\!\! \!\!\!\! \bigg\|\sum_{j\in \mathbb{Z}}\big(\dot{S}_{j-1}-\dot{S}_{j-1+[\log_{2}\eta]}\big)\mathcal{P}^{\bot}u^{\epsilon}\nabla\Delta_{j}\bigg\|_{L^{1}_{T}(\dot{B}^{\beta}_{2,1})}\nonumber\\
\leq\,& C \eta^{6\alpha-1}\|\mathcal{P}^{\bot}u^{\epsilon}\|_{L_{T}^{\frac{1}{1-3\alpha}}(\dot{B}^{1-6\alpha}_{\infty,1})}
\|\nabla v\|_{L_{T}^{\frac{1}{3\alpha}}(\dot{B}^{6\alpha-1+\beta}_{2,1})}\nonumber\\
\leq\,& C \eta^{6\alpha-1}\epsilon^{\alpha}V_{\beta}\big(X_{\alpha}+\epsilon^{-\frac{1}{4}}Y_{\alpha}\big).\nonumber
\end{align}
Choosing $\eta=\epsilon^{\alpha/(2-6\alpha)}$, we have,
$$\|\dot{T}_{\mathcal{P}^{\bot}u^{\epsilon}}\nabla v\|_{L^{1}_{T}(\dot{B}^{\beta}_{2,1})}\leq C \epsilon^{\frac{\alpha}{2-6\alpha}}
\big(X_{0}V_{\beta}+V_{\beta}(X_{\alpha}+\epsilon^{-\frac{1}{4}}Y_{\alpha})\big).$$
Using \eqref{Cc-24} and Remark \ref{re2.3}, we get
\begin{align}
\|\dot{R}(\mathcal{P}^{\bot}u^{\epsilon},\nabla v)\|_{L^{1}_{T}(\dot{B}^{\beta}_{2,1})}&\leq C
\|\nabla v\|_{L^{\frac{1}{3\alpha}}_{T}(\dot{B}^{6\alpha+\beta-1}_{2,1})}
\|\mathcal{P}^{\bot}u^{\epsilon}\|_{L^{\frac{1}{1-3\alpha}}_{T}(\dot{B}^{2-10\alpha}_{\frac{2}{1-4\alpha},1})} \nonumber\\
&\leq C \epsilon^{\alpha} V_{\beta}(X_{\alpha}+\epsilon^{-\frac{1}{4}}Y_{\alpha}),\nonumber\\
\|\dot{S}_{q-1}\nabla v\Delta_{q}\mathcal{P}^{\bot}u^{\epsilon}\|_{L^{2}}&\leq \|\Delta_{q}\mathcal{P}^{\bot}u^{\epsilon}\|_{L^{\frac{2}{1-4\alpha}}}
\sum_{j\leq q-2}\|\Delta_{j}\nabla v\|_{L^{\frac{1}{2\alpha}}} \nonumber\\
&\leq 2^{-q\beta}\big(2^{q(2-10\alpha)}\|\Delta_{q}\mathcal{P}^{\bot}u^{\epsilon}\|_{L^{\frac{2}{1-4\alpha}}}\big)
\|\nabla v\|_{\dot{B}^{10\alpha+\beta-2}_{\frac{1}{2\alpha},1}}.\nonumber
\end{align}
Here we have used the facts that $\alpha,\beta\leq\frac{1}{6}$ and $10\alpha+\beta-2<0$.
By the embedding $\dot{B}^{6\alpha+\beta-1}_{2,1}\hookrightarrow \dot{B}^{10\alpha+\beta-2}_{\frac{1}{2\alpha},1}$, we get
\begin{align}
\|\dot{T}_{\nabla v}\mathcal{P}^{\bot}u^{\epsilon}\|_{L^{1}_{T}(\dot{B}^{\beta}_{2,1})}&\leq C
\|\nabla v\|_{L^{\frac{1}{3\alpha}}_{T}(\dot{B}^{6\alpha+\beta-1}_{2,1})}
\|\mathcal{P}^{\bot}u^{\epsilon}\|_{L^{\frac{1}{1-3\alpha}}_{T}(\dot{B}^{2-10\alpha}_{\frac{2}{1-4\alpha},1})}\nonumber\\
&\leq C \epsilon^{\alpha} V_{\beta}(X_{\alpha}+\epsilon^{-\frac{1}{4}}Y_{\alpha}).\nonumber
\end{align}
Plugging all the above inequalities into \eqref{bonya}, we finally obtain
\begin{align}\label{Cc-25}
\|\mathcal{P}^{\bot}u^{\epsilon}\cdot\nabla v\|_{L^{1}_{T}(\dot{B}^{\beta}_{2,1})}&\leq C\epsilon^{\frac{\alpha}{2-6\alpha}}
\big(X_{0}V_{\beta}+V_{\beta}(X_{\alpha}+\epsilon^{-\frac{1}{4}}Y_{\alpha})\big).
\end{align}
Similar arguments lead to
\begin{align}
\|\mathcal{P}^{\bot}u^{\epsilon}\cdot\nabla B\|_{L^{1}_{T}(\dot{B}^{\beta}_{2,1})}&\leq C\epsilon^{\frac{\alpha}{2-6\alpha}}
\big(X_{0}V_{\beta}+V_{\beta}(X_{\alpha}+\epsilon^{-\frac{1}{4}}Y_{\alpha})\big),\label{Cc-26}\\
\|v\cdot \nabla\mathcal{P}^{\bot}u^{\epsilon}\|_{L^{1}_{T}(\dot{B}^{\beta}_{2,1})}& \leq C\epsilon^{\frac{\alpha}{1+6\alpha+\beta}}
\big(V_{0}X_{\beta}+V_{\beta}(X_{\alpha}+\epsilon^{-\frac{1}{4}}Y_{\alpha})\big),\label{Cc-27}\\
\|\dv\mathcal{P}^{\bot}u^{\epsilon}H^{\epsilon}\|_{L^{1}_{T}(\dot{B}^{\beta}_{2,1})}&\leq C\epsilon^{\frac{\alpha}{1+6\alpha+\beta}}
\big(P_{0}X_{\beta}+P_{\beta}(X_{\alpha}+\epsilon^{-\frac{1}{4}}Y_{\alpha})\big).\label{Cc-28}
\end{align}

Plugging the estimates  \eqref{Cc-8}-\eqref{Cc-11}, \eqref{Cc-13}-\eqref{Cc-15},
\eqref{Cc-19}-\eqref{Cc-23} or \eqref{Cc-25}-\eqref{Cc-28}
in \eqref{Cc-7}, we eventually get,
if $d=3$,
\begin{align}\label{Cc-29}
W_{\beta}\leq\,& Ce^{C(V_{0}+X_{0})}\Big(W_{0}W_{\beta}+\epsilon^{\alpha}X_{\alpha}\big(V_{\beta}+(1+W_{0}+V_{0})(W_{\beta}+V_{\beta})\big)\nonumber\\
& +\epsilon^{\frac{2\alpha}{5+2\beta}}\big((W_{0}+V_{0})X_{\beta}+
(W_{\beta}+V_{\beta})(X_{\alpha}+\epsilon^{-\frac{2\alpha}{1+2\alpha}}Y^{p_{\alpha}}_{\alpha})\big)\Big),
\end{align}
while if $d=2$,
\begin{align}\label{Cc-30}
W_{\beta}\leq \,& Ce^{C(V_{0}+X_{0})}\Big(W_{0}W_{\beta}+\epsilon^{\alpha}X_{\alpha}\big(V_{\beta}+(1+W_{0}+V_{0})(W_{\beta}+V_{\beta})\big)\nonumber\\
& +\epsilon^{\frac{\alpha}{2+\beta}}\big(X_{0}V_{\beta}+(W_{0}+V_{0})X_{\beta}+
(W_{\beta}+V_{\beta})(X_{\alpha}+\epsilon^{-\frac{1}{4}}Y_{\alpha})\big)\Big).
\end{align}


\subsection{Estimates for $(b^{\epsilon},\mathcal{P}^{\bot}u^{\epsilon})$ in  $E^{\frac{d}{2}+\beta}_{\epsilon,T}$.}

We first need the following Proposition.

\begin{Proposition}[\cite{DC}]\label{bprop3.2}
Let $\epsilon >0,s\in \mathbb{R},1\leq p,r<\infty$, and $(a,u)$ be a solution of the following system:
\begin{equation}
\left\{\begin{array}{l}
\partial_{t}a+\dot{T}_{v^{j}}\partial_{j}a+\frac{\Lambda u}{\epsilon}=F,\\
\partial_{t}u+\dot{T}_{v^{j}}\partial_{j} u-\nu\Delta u-\frac{\Lambda a}{\epsilon}=G,
\end{array}
\right.\nonumber
\end{equation}
with $\dot{T}_{v}a:=\sum_{j}\dot{S}_{j-1}v\dot{\Delta}_{j}a$.
Then there exists a constant C, depending only on $d,p,r,$ and $s,$ such that the following estimate holds:
\begin{align}
& \|a(t)\|_{\tilde{B}^{s,\infty}_{\epsilon}}+\|u(t)\|_{\dot{B}^{s-1}_{2,1}}+\int_{0}^{t}(\|a\|_{\tilde{B}^{s,1}_{\epsilon}}+
\|u\|_{\dot{B}^{s+1}_{2,1}}){\rm d}\tau\nonumber\\
& \ \ \leq C e^{C V^{p,r}_{\epsilon}(t)}  \Big( \|a_{0}\|_{\tilde{B}^{s,\infty}_{\epsilon}}+\|u_{0}\|_{\dot{B}^{s-1}_{2,1}}
+\int_{0}^{t}e^{-CV^{p,r}_{\epsilon}(\tau)}\big(\|F\|_{\tilde{B}^{s,\infty}_{\epsilon}}+\|G\|_{\dot{B}^{s-1}_{2,1}}\big)
{\rm d}\tau\Big).\nonumber
\end{align}
where
\begin{equation}
V^{p,r}_{\epsilon}(t):=
\left\{\begin{array}{l}
\displaystyle{\int_{0}^{t}}\big(\nu^{1-p}\|\nabla v\|^{p}_{\dot{B}^{\frac{2}{p}-2}_{\infty,\infty}}
+(\epsilon^{2}\nu)^{r-1}\|\nabla v\|^{r}_{L^{\infty}}\big){\rm d}\tau, \quad \mathrm{if}\,\,p>1,\\
\displaystyle{\int_{0}^{t}}\big(\|\nabla v\|_{L^{\infty}}+ (\epsilon^{2}\nu)^{r-1}\|\nabla v\|^{r}_{L^{\infty}} \big){\rm d}\tau,\qquad \qquad\mathrm{if}\,\,p=1.
\end{array}
\right.\nonumber
\end{equation}
\end{Proposition}

Set $$d^{\epsilon}:=\Lambda^{-1}\dv\mathcal{P}^{\bot}u^{\epsilon}.$$
 By the system \eqref{Cc-2}, we know that $(b^{\epsilon},d^{\epsilon})$  satisfies  the following
system:
\begin{equation}\label{bonyb}
\left\{\begin{array}{l}
\partial_{t}b^{\epsilon}+\dot{T}_{u_{j}^{\epsilon}}\partial_{j}b^{\epsilon}+\frac{\Lambda d^{\epsilon}}{\epsilon}=S^{\epsilon},\\
\partial_{t}d^{\epsilon}+\dot{T}_{u_{j}^{\epsilon}}\partial_{j}d^{\epsilon}-\bar{\nu}\Delta d^{\epsilon}-\frac{\Lambda b^{\epsilon}}{\epsilon}=T^{\epsilon},
\end{array}
\right.
\end{equation}
with
\begin{align}
S^{\epsilon}:=\,&-\dot{T}'_{\partial_{j}b^{\epsilon}}u_{j}^{\epsilon}-b^{\epsilon}\dv u^{\epsilon},\nonumber\\
T^{\epsilon}:=\,&\Lambda^{-1}\dv \mathcal{P}^{\bot}\Big(K(\epsilon b^{\epsilon})\frac{\nabla b^{\epsilon}}{\epsilon}
-I(\epsilon b^{\epsilon})\mathcal{A}u^{\epsilon}\Big)\nonumber\\
&+\Lambda^{-1}\dv \mathcal{P}^{\bot}\Big(\frac{1}{1+\epsilon b^{\epsilon}}
\big(H^{\epsilon}\cdot\nabla H^{\epsilon}-\frac{1}{2}|\nabla H^{\epsilon}|^{2}\big)\Big)\nonumber\\
&+u^{\epsilon}\cdot\nabla \Lambda^{-1}\dv \mathcal{P}^{\bot}u^{\epsilon}- \Lambda^{-1}\dv \mathcal{P}^{\bot}(u^{\epsilon}\cdot\nabla u^{\epsilon})
- \dot{T}'_{\partial_{j}\Lambda^{-1}\dv \mathcal{P}^{\bot}u^{\epsilon}}u_{j}^{\epsilon} . \nonumber
\end{align}
Applying Proposition \ref{bprop3.2} to \eqref{bonyb}, we get
\begin{align}\label{Cc-31}
X_{\beta}(T)\leq Ce^{V^{p,r}_{\epsilon}(T)}&\bigg(\|b_{0}\|_{\tilde{B}^{\frac{d}{2}+\beta,\infty}_{\epsilon}}+\|d_{0}\|_{\dot{B}^{\frac{d}{2}-1+\beta}_{2,1}}\nonumber\\
&\qquad+\|S^{\epsilon}\|_{L^{1}_{T}(\tilde{B}^{\frac{d}{2}+\beta,\infty}_{\epsilon})}+
\|T^{\epsilon}\|_{L^{1}_{T}(\dot{B}^{\frac{d}{2}-1+\beta}_{2,1})}\bigg)
\end{align}
with
\begin{align}\label{Cc-32}
V^{p,r}_{\epsilon}(t):=\int_{0}^{t}\big(\bar{\nu}^{1-p}\|\nabla u^{\epsilon}(\tau)\|^{p}_{\dot{B}^{\frac{2}{P}-2}_{\infty,1}}+(\epsilon^{2}\bar{\nu})^{r-1}
\|\nabla u^{\epsilon}(\tau)\|^{r}_{L^{\infty}}\big){\rm d}\tau.
\end{align}
for any $p,r>1$ (to be fixed hereafter). Below we give estimates on $S^{\epsilon}$ and $T^{\epsilon}$.

\subsubsection{Estimates for $S^{\epsilon}$}

Applying (ii) in Remark \ref{re2.2}, we have
\begin{align}\label{Cc-33}
\|S^{\epsilon}\|_{L^{1}_{T}(\tilde{B}^{\frac{d}{2}+\beta,\infty}_{\epsilon})}\leq C
\bigg(\|S^{\epsilon}\|_{L^{1}_{T}(\dot{B}^{\frac{d}{2}-1+\beta}_{2,1})}+\epsilon \|S^{\epsilon}\|_{L^{1}_{T}(\dot{B}^{\frac{d}{2}+\beta}_{2,1})}\bigg).
\end{align}
For both $d=3$ and $d=2$,
according to \eqref{Cc-12} we have the following inequality:
\begin{align}\label{Cc-34}
\|\dot{T}'_{\partial_{j}b^{\epsilon}}u_{j}^{\epsilon}\|_{L^{1}_{T}(\dot{B}^{\frac{d}{2}+\beta}_{2,1})}&\leq C
\|\nabla b^{\epsilon}\|_{L^{\infty}_{T}(\dot{B}^{\frac{d}{2}-1}_{2,1})}\|u^{\epsilon}\|_{L^{1}_{T}(\dot{B}^{\frac{d}{2}+1+\beta}_{2,1})}\nonumber\\
&\leq C\epsilon^{\alpha-1}\| b^{\epsilon}\|_{L^{\infty}_{T}(\tilde{B}^{\frac{d}{2}+\alpha,\infty}_{\epsilon})}
\|u^{\epsilon}\|_{L^{1}_{T}(\dot{B}^{\frac{d}{2}-1+\beta}_{2,1})}\nonumber\\
&\leq C \epsilon^{\alpha-1}X_{\alpha}(X_{\beta}+P_{\beta});
\end{align}
and using (ii) in Remark \ref{re2.2} and \eqref{Cc-12}, we conclude that
\begin{align}
\|b^{\epsilon}\dv\mathcal{P}^{\bot}u^{\epsilon}\|_{L^{1}_{T}(\dot{B}^{\frac{d}{2}}_{2,1})}&\leq C
\|\dv \mathcal{P}^{\bot}u^{\epsilon}\|_{L^{1}_{T}(\dot{B}^{\frac{d}{2}}_{2,1})}
\|b^{\epsilon}\|_{L^{\infty}_{T}(\dot{B}^{\frac{d}{2}}_{2,1})}\nonumber\\
&\leq C \epsilon^{\alpha-1}X_{\alpha}X_{0}, \label{Cc-35}\\
\|b^{\epsilon}\dv\mathcal{P}^{\bot}u^{\epsilon}\|_{L^{1}_{T}(\dot{B}^{\frac{d}{2}+\alpha}_{2,1})}&\leq C
\|b^{\epsilon}\|_{L^{\infty}_{T}(\dot{B}^{\frac{d}{2}}_{2,1})}\|\dv \mathcal{P}^{\bot}u^{\epsilon}\|_{L^{1}_{T}(\dot{B}^{\frac{d}{2}+\alpha}_{2,1})}\nonumber\\
&\quad +C\|\dv \mathcal{P}^{\bot}u^{\epsilon}\|_{L^{2/(2-\alpha)}_{T}(\dot{B}^{\frac{d}{2}}_{2,1})}
\|b^{\epsilon}\|_{L^{2/\alpha}_{T}(\dot{B}^{\frac{d}{2}+\alpha}_{2,1})}\nonumber\\
&\leq C \epsilon^{\alpha-1}\Big(\| b^{\epsilon}\|_{L^{\infty}_{T}(\tilde{B}^{\frac{d}{2}+\alpha,\infty}_{\epsilon})}
\|\mathcal{P}^{\bot}u^{\epsilon}\|_{L^{1}_{T}(\dot{B}^{\frac{d}{2}+1+\alpha}_{2,1})}\nonumber\\
&\quad +\| b^{\epsilon}\|_{L^{\frac{2}{\alpha}}_{T}(\tilde{B}^{\frac{d}{2}+\alpha,\frac{2}{\alpha}}_{\epsilon})}
\|\mathcal{P}^{\bot}u^{\epsilon}\|_{L^{\frac{2}{2-\alpha}}_{T}(\dot{B}^{\frac{d}{2}+1}_{2,1})}\Big)\nonumber\\
&\leq C  \epsilon^{\alpha-1} X_{\alpha}^{2}.\label{Cc-36}
\end{align}

Now  we deal with the  term $\|S^{\epsilon}\|_{L^{1}_{T}(\dot{B}^{\frac{d}{2}-1+\beta}_{2,1})}$ in \eqref{Cc-33} separately for  $d=3$ or $d=2$.

For $d=3$, thanks to Remark \ref{re2.3} we have
\begin{align}\label{Cc-37}
\|\dot{T}'_{\partial_{j}b^{\epsilon}}u_{j}^{\epsilon}\|_{L^{1}_{T}(\dot{B}^{\frac{1}{2}+\beta}_{2,1})}&\leq C
\|\partial_{j}b^{\epsilon}\|_{L_{T}^{\frac{1}{\alpha}}(\dot{B}^{2\alpha-2}_{\infty,1})}
\|u_{j}^{\epsilon}\|_{L_{T}^{\frac{1}{1-\alpha}}(\dot{B}^{5/2-2\alpha+\beta}_{2,1})}\nonumber\\
&\leq C\epsilon^{\alpha}\big(X_{\alpha}+P_{\alpha}\big)\big(\epsilon^{-\alpha}Y_{\alpha}^{\frac{1}{\alpha}}\big).
\end{align}
Applying   \eqref{Cc-16} and Proposition \ref{prop2.2}, we get
\begin{align}\label{Cc-38}
\|b^{\epsilon}\dv \mathcal{P}^{\bot}u^{\epsilon}\|_{L^{1}_{T}(\dot{B}^{\frac{1}{2}+\beta}_{2,1})}&\leq C
\|b^{\epsilon}\|_{L_{T}^{\frac{1}{\alpha}}(\dot{B}^{2\alpha-1}_{\infty,1})}
\|\dv \mathcal{P}^{\bot}u_{j}^{\epsilon}\|_{L_{T}^{\frac{1}{1-\alpha}}(\dot{B}^{3/2-2\alpha+\beta}_{2,1})}\nonumber\\
&\quad +C\|\dv \mathcal{P}^{\bot}u_{j}^{\epsilon}\|_{L^{2}_{T}(\dot{B}^{-1}_{\infty,1})}\|b^{\epsilon}\|_{L^{2}_{T}(\dot{B}^{\frac{3}{2}+\beta}_{2,1})}\nonumber\\
&\leq C\epsilon^{\alpha} X_{\beta}\big(X_{\alpha}+\epsilon^{-\alpha}Y_{\alpha}^{\frac{1}{\alpha}}+\epsilon^{-2\alpha/(1+2\alpha)}Y_{\alpha}^{p_{\alpha}}\big).
\end{align}
Plugging  \eqref{Cc-34}-\eqref{Cc-38} into \eqref{Cc-33}, we have
\begin{align}\label{Cc-39}
\|S^{\epsilon}\|_{L^{1}_{T}(\tilde{B}^{\frac{3}{2}+\beta,\infty}_{\epsilon})}\leq C \epsilon^{\alpha}\big(X_{\beta}+P_{\beta}\big)
\big(X_{\alpha}+\epsilon^{-\alpha}Y_{\alpha}^{\frac{1}{\alpha}}+\epsilon^{-\frac{2\alpha}{1+2\alpha}}Y_{\alpha}^{p_{\alpha}}\big).
\end{align}

For $d=2$,
we first remark that $S^{\epsilon}$ can be rewritten as
$$S^{\epsilon}=-\dot{T}_{\dv\mathcal{P}^{\bot} u^{\epsilon}}b^{\epsilon}-\partial_{j}\dot{T}'_{b^{\epsilon}} u_{j}^{\epsilon}.$$
By Remark \ref{re2.3},   (4) in Proposition \ref{prop2.1}, and  H\"{o}lder's inequality, we get
\begin{align}
\|\dot{T}_{\dv\mathcal{P}^{\bot} u^{\epsilon}}b^{\epsilon}\|_{L_{T}^{1}(\dot{B}^{\beta}_{2,1})}
\leq\, &C
\|\dv\mathcal{P}^{\bot} u^{\epsilon}\|_{L_{T}^{\frac{7}{4-3\alpha}}(\dot{B}^{-1}_{\infty,1})}
\|b^{\epsilon}\|_{L_{T}^{\frac{7}{3+3\alpha}}(\dot{B}^{1+\beta}_{2,1})}\nonumber\\
\leq\,& C\epsilon^{\frac{6\alpha-1}{7}}\|\dv\mathcal{P}^{\bot} u^{\epsilon}\|_{L_{T}^{1}(\dot{B}^{1+\alpha}_{2,1})}^{\frac{3-4\alpha}{7}}\nonumber\\
&\times\|\dv\mathcal{P}^{\bot} u^{\epsilon}\|_{L_{T}^{4}(\dot{B}^{\alpha-\frac{7}{4}}_{\infty,1})}^{\frac{4+4\alpha}{7}}
\|b^{\epsilon}\|_{L_{T}^{\frac{7}{3+3\alpha}}(\tilde{B}^{1+\beta,\frac{7}{3+3\alpha}}_{\epsilon})}\nonumber\\
\leq\, &C\epsilon^{\alpha}X_{\beta}(X_{\alpha}+\epsilon^{-\frac{1}{4}}Y_{\alpha}),\label{Cc-40}\\
\|\partial_{j}\dot{T}'_{b^{\epsilon}} u_{j}^{\epsilon}\|_{L_{T}^{1}(\dot{B}^{\beta+\alpha-\frac{1}{4}}_{2,1})}
\leq\, &C
\|b^{\epsilon}\|_{L_{T}^{4}(\dot{B}^{\alpha-\frac{3}{4}}_{\infty,1})}\|u^{\epsilon}\|_{L_{T}^{\frac{4}{3}}(\dot{B}^{\beta+\frac{3}{2}}_{2,1})}\nonumber\\
\leq\,& C (X_{\beta}+P_{\beta})Y_{\alpha}.\label{Cc-41}
\end{align}
According to the definition of hybrid  Besov norms, we get the following equivalent forms:
$$\|S^{\epsilon}\|_{L^{1}_{T}(\tilde{B}^{1+\beta,\infty}_{\epsilon})}\approx \|S^{\epsilon}_{BF}\|_{L^{1}_{T}(\dot{B}^{\beta}_{2,1})}
+\epsilon\|S^{\epsilon}_{HF}\|_{L^{1}_{T}(\dot{B}^{1+\beta}_{2,1})}.$$
Thus,
\begin{align}
&\!\!\!\!\!\! \!\!\!\!\!\!  \|S^{\epsilon}\|_{L^{1}_{T}(\tilde{B}^{1+\beta,\infty}_{\epsilon})}\nonumber\\
\,&\leq C
\big(\|(\dot{T}_{\dv\mathcal{P}^{\bot} u^{\epsilon} }b^{\epsilon})_{BF}\|_{L^{1}_{T}(\dot{B}^{\beta}_{2,1})}
+\|(\partial_{j}\dot{T}'_{b^{\epsilon}} u_{j}^{\epsilon})_{BF}\|_{L^{1}_{T}(\dot{B}^{\beta}_{2,1})}
+\epsilon\|S^{\epsilon}_{HF}\|_{L^{1}_{T}(\dot{B}^{1+\beta}_{2,1})}\big)\nonumber\\
\,&\leq C\big(\|\dot{T}_{\dv\mathcal{P}^{\bot} u^{\epsilon} }b^{\epsilon}\|_{L^{1}_{T}(\dot{B}^{\beta}_{2,1})}
+\epsilon^{\alpha-\frac{1}{4}}\|\partial_{j}\dot{T}'_{b^{\epsilon}} u_{j}^{\epsilon}\|_{L^{1}_{T}(\dot{B}^{\beta+\alpha-\frac{1}{4}}_{2,1})}
+\epsilon\|S^{\epsilon}\|_{L^{1}_{T}(\dot{B}^{1+\beta}_{2,1})}\big).\nonumber
\end{align}
Combining   \eqref{Cc-34}-\eqref{Cc-36}  and \eqref{Cc-40}-\eqref{Cc-41} together, we have
\begin{align}\label{Cc-42}
\|S^{\epsilon}\|_{L^{1}_{T}(\tilde{B}^{1+\beta,\infty}_{\epsilon})}\leq C \epsilon^{\alpha}(X_{\beta}+P_{\beta})(X_{\alpha}+\epsilon^{-\frac{1}{4}}Y_{\alpha}).
\end{align}

\subsubsection{Estimates for $T^{\epsilon}$}

First we consider the case $d=3$.
Thanks to Proposition \ref{prop2.2} and Lemma \ref{le2.2}, we get
\begin{align}
\|K(\epsilon b^{\epsilon})\nabla b^{\epsilon}\|_{\dot{B}^{\frac{1}{2}+\beta}_{2,1}}&\leq C
\|K(\epsilon b^{\epsilon})\|_{L^{\infty}}\|\nabla b^{\epsilon}\|_{\dot{B}^{\frac{1}{2}+\beta}_{2,1}}
+C\|K(\epsilon b^{\epsilon})\|_{\dot{B}^{\frac{3}{2}+\beta}_{2,1}}\|\nabla b^{\epsilon}\|_{\dot{B}^{-1}_{\infty,1}}\nonumber\\
&\leq C \|\epsilon b^{\epsilon}\|_{L^{\infty}}\| b^{\epsilon}\|_{\dot{B}^{\frac{3}{2}+\beta}_{2,1}}
+C\|\epsilon b^{\epsilon}\|_{\dot{B}^{\frac{3}{2}+\beta}_{2,1}}\| b^{\epsilon}\|_{\dot{B}^{0}_{\infty,1}}\nonumber\\
&\leq C\epsilon\| b^{\epsilon}\|_{\dot{B}^{\frac{3}{2}+\beta}_{2,1}}\left(\|b^{\epsilon}_{BF}\|_{\dot{B}^{0}_{\infty,1}}
+\|b^{\epsilon}_{HF}\|_{\dot{B}^{0}_{\infty,1}}\right).\nonumber
\end{align}
Thus,
\begin{align}\label{Cc-43}
\Big\|\frac{K(\epsilon b^{\epsilon})\nabla b^{\epsilon}}{\epsilon}\Big\|_{L^{1}_{T}(\dot{B}^{\frac{1}{2}+\beta}_{2,1})}
\leq\,& C \|b^{\epsilon}_{BF}\|_{L^{2}_{T}(\dot{B}^{0}_{\infty,1})}\|b^{\epsilon}\|_{L^{2}_{T}(\dot{B}^{3/2+\beta}_{2,1})}\nonumber\\
&+C \|b^{\epsilon}\|_{L^{1/\alpha}_{T}(\dot{B}^{3/2+\beta}_{2,1})}\|b^{\epsilon}_{HF}\|_{L^{1/(1-\alpha)}_{T}(\dot{B}^{0}_{\infty,1})}.
\end{align}
We notice that we can replace $\mathcal{P}^{\bot}u^{\epsilon}$ by $b^{\epsilon}_{BF}$ in the proof of \eqref{Cc-16}, thus we have
\begin{align}\label{Cc-44}
\|b^{\epsilon}_{BF}\|_{L^{2}_{T}(\dot{B}^{0}_{\infty,1})}\leq C \epsilon^{\alpha}\big(X_{\alpha}+\epsilon^{-2\alpha/(1+2\alpha)}Y_{\alpha}^{p_{\alpha}}\big).
\end{align}
Moreover, applying  (ii) in Remark \ref{re2.2}, we   obtain  that
\begin{align}\label{Cc-45}
\|b^{\epsilon}\|_{L^{1/\alpha}_{T}(\dot{B}^{3/2+\beta}_{2,1})}\leq C\epsilon^{2\alpha-1}\|b^{\epsilon}\|_{L^{1/\alpha}_{T}(\tilde{B}^{3/2+\beta,1/\alpha}_{\epsilon})}.
\end{align}
Thanks  to (3) and (4) in Proposition \ref{prop2.1}, and (ii) in Remark \ref{re2.2}, we obtain that
\begin{align}\label{Cc-46}
\|b^{\epsilon}_{HF}\|_{L^{1/(1-\alpha)}_{T}(\dot{B}^{0}_{\infty,1})}&\leq C
\|b^{\epsilon}_{HF}\|^{(1-2\alpha)/(1-\alpha)}_{L^{1}_{T}(\dot{B}^{3/2+\alpha}_{2,1})}
\|b^{\epsilon}_{HF}\|^{\alpha/(1-\alpha)}_{L^{1/\alpha}_{T}(\dot{B}^{-1+2\alpha}_{\infty,1})}\nonumber\\
&\leq C\big(\epsilon\|b^{\epsilon}\|_{L^{1}_{T}(\tilde{B}^{3/2+\alpha,1}_{\epsilon})}\big)^{\frac{1-2\alpha}{1-\alpha})}
\epsilon^{\frac{\alpha^{2}}{1-\alpha}}\Big(\epsilon^{-\alpha}Y_{\alpha}^{1/\alpha}\Big)^{\frac{\alpha}{1-\alpha}}\nonumber\\
&\leq C \epsilon^{1-\alpha}\big(X_{\alpha}+\epsilon^{-\alpha}Y_{\alpha}^{1/\alpha}\big).
\end{align}
Plugging  \eqref{Cc-44}-\eqref{Cc-46} into \eqref{Cc-43} gives
\begin{align}\label{Cc-47}
\Big\|\frac{K(\epsilon b^{\epsilon})\nabla b^{\epsilon}}{\epsilon}\Big\|_{L^{1}_{T}(\dot{B}^{\frac{1}{2}+\beta}_{2,1})}
\leq C \epsilon^{\alpha} X_{\beta}\Big(X_{\alpha}+\epsilon^{-2\alpha/(1+2\alpha)}Y_{\alpha}^{p_{\alpha}}+\epsilon^{-\alpha}Y_{\alpha}^{1/\alpha}\Big).
\end{align}
Applying Remark \ref{re2.3}  and  \eqref{Cc-16}, we have
\begin{align}\label{Cc-48}
\|\dot{T}'_{\partial_{j}\Lambda^{-1}\dv \mathcal{P}^{\bot}u^{\epsilon}}u_{j}^{\epsilon}\|_{L^{1}_{T}(\dot{B}^{\frac{d}{2}-1+\beta}_{2,1})}
&\leq C \|\partial_{j}\Lambda^{-1}\dv \mathcal{P}^{\bot}u^{\epsilon}\|_{L^{2}_{T}(\dot{B}^{-1}_{\infty,1})}
\|u^{\epsilon}\|_{L^{2}_{T}(\dot{B}^{\frac{d}{2}+\beta}_{2,1})}\nonumber\\
&\leq C \epsilon^{\alpha}\big(X_{\alpha}+\epsilon^{-\frac{2\alpha}{1+2\alpha}}Y_{\alpha}^{p_{\alpha}}\big)\big(X_{\beta}+P_{\beta}\big).
\end{align}
By means of Proposition \ref{prop2.2}  and  \eqref{Cc-12}, we have
\begin{align}\label{Cc-49}
\|\Lambda^{-1}\dv \mathcal{P}^{\bot}\big(I(\epsilon b^{\epsilon})\mathcal{A}u^{\epsilon}\big)\|_{L^{1}_{T}(\dot{B}^{\frac{d}{2}-1+\beta}_{2,1})}
&\leq C\epsilon \| b^{\epsilon}\|_{L^{\infty}_{T}(\dot{B}^{\frac{d}{2}}_{2,1})}
\|u^{\epsilon}\|_{L^{1}_{T}(\dot{B}^{\frac{d}{2}+1+\beta}_{2,1})}\nonumber\\
&\leq C\epsilon^{\alpha}X_{\alpha}(X_{\beta}+P_{\beta}).
\end{align}

Now  we introduce  the following decomposition:
\begin{align}
&\!\!\!\!\!\!\!\!\!\!\!\!\!\! \!\!\!\!
u^{\epsilon}\cdot\nabla\Lambda^{-1}\dv\mathcal{P}^{\bot}u^{\epsilon}-\Lambda^{-1}\dv(u^{\epsilon}\cdot\nabla u^{\epsilon})\nonumber\\
=&\mathcal{P}^{\bot}u^{\epsilon}\cdot\nabla\Lambda^{-1}\dv\mathcal{P}^{\bot}u^{\epsilon}-
\Lambda^{-1}\dv(\mathcal{P}^{\bot}u^{\epsilon}\cdot\nabla \mathcal{P}^{\bot}u^{\epsilon})\nonumber\\
&-\Lambda^{-1}\dv(\mathcal{P}u^{\epsilon}\cdot\nabla \mathcal{P}u^{\epsilon})
-\Lambda^{-1}\dv(\mathcal{P}^{\bot}u^{\epsilon}\cdot\nabla \mathcal{P}u^{\epsilon})\nonumber\\
&-\Lambda^{-1}\dv(\mathcal{P}u^{\epsilon}\cdot\nabla \mathcal{P}^{\bot}u^{\epsilon})
+\mathcal{P}u^{\epsilon}\cdot\nabla\Lambda^{-1}\dv\mathcal{P}^{\bot}u^{\epsilon}.\label{bonyc}
\end{align}
Thanks to Proposition \ref{prop2.2}, we can get the bounds for the first two terms of the right-hand side of \eqref{bonyc} as the following:
\begin{align}\label{Cc-50}
&\!\!\!\!\!\!\!\!\!\!\!\!\!\! \!\!\!\!
\|\mathcal{P}^{\bot}u^{\epsilon}\cdot\nabla\Lambda^{-1}\dv\mathcal{P}^{\bot}u^{\epsilon}-
\Lambda^{-1}\dv(\mathcal{P}^{\bot}u^{\epsilon}\cdot\nabla \mathcal{P}^{\bot}u^{\epsilon})\|_{L^{1}_{T}(\dot{B}^{\frac{d}{2}-1+\beta}_{2,1})}\nonumber\\
\leq \, & C \|\mathcal{P}^{\bot}u^{\epsilon}\|_{L^{2}_{T}(\dot{B}^{0}_{\infty,1})}
\|\mathcal{P}^{\bot}u^{\epsilon}\|_{L^{2}_{T}(\dot{B}^{\frac{d}{2}+\beta}_{2,1})}\nonumber\\
\leq\,& C\epsilon^{\alpha}X_{\beta}(X_{\alpha}+\epsilon^{-\frac{2\alpha}{1+2\alpha}}Y_{\alpha}^{p_{\alpha}}).
\end{align}
For the  third term on the right-hand side of \eqref{bonyc}, we have
\begin{align}\label{Cc-51}
\|\Lambda^{-1}\dv(\mathcal{P}u^{\epsilon}\cdot\nabla \mathcal{P}u^{\epsilon})\|_{L^{1}_{T}(\dot{B}^{\frac{d}{2}-1+\beta}_{2,1})}
&\leq C\|\mathcal{P}u^{\epsilon}\|_{L^{2}_{T}(\dot{B}^{\frac{d}{2}}_{2,1})}
\|\nabla\mathcal{P}u^{\epsilon}\|_{L^{2}_{T}(\dot{B}^{\frac{d}{2}-1+\beta}_{2,1})}\nonumber\\
&\leq CP_{0}P_{\beta}.
\end{align}
Applying similar computations to those in the second step, we have,
\begin{align}
&\|\Lambda^{-1}\dv(\mathcal{P}^{\bot}u^{\epsilon}\cdot\nabla \mathcal{P}u^{\epsilon})\|_{L^{1}_{T}(\dot{B}^{\frac{d}{2}-1+\beta}_{2,1})}\nonumber\\
& \qquad \qquad \leq  C \epsilon^{\frac{4\alpha}{2+d+2\beta}}\big(P_{0}X_{\beta}
+P_{\beta}(X_{\alpha}+\epsilon^{-\frac{2\alpha}{1+2\alpha}}Y_{\alpha}^{p_{\alpha}})\big),\label{Cc-52}\\
&\|\Lambda^{-1}\dv(\mathcal{P}u^{\epsilon}\cdot\nabla \mathcal{P}^{\bot}u^{\epsilon})\|_{L^{1}_{T}(\dot{B}^{\frac{d}{2}-1+\beta}_{2,1})}\nonumber\\
& \qquad\qquad \leq  C \epsilon^{\frac{2\alpha}{2+d+2\beta}}\big(P_{0}X_{\beta}
+P_{\beta}(X_{\alpha}+\epsilon^{-\frac{2\alpha}{1+2\alpha}}Y_{\alpha}^{p_{\alpha}})\big),\label{Cc-53}\\
&\|\mathcal{P}u^{\epsilon}\cdot\nabla \Lambda^{-1}\dv\mathcal{P}^{\bot}u^{\epsilon}\|_{L^{1}_{T}(\dot{B}^{\frac{d}{2}-1+\beta}_{2,1})}\nonumber\\
& \qquad\qquad \leq  C \epsilon^{\frac{2\alpha}{2+d+2\beta}}\big(P_{0}X_{\beta}
+P_{\beta}(X_{\alpha}+\epsilon^{-\frac{2\alpha}{1+2\alpha}}Y_{\alpha}^{p_{\alpha}})\big).\label{Cc-54}
\end{align}

Now, we estimates the last two terms in $T^\epsilon$. First, we have
\begin{align}
\|H^{\epsilon}\cdot\nabla H^{\epsilon} -\frac{1}{2}\nabla |H^{\epsilon}|^{2} \|_{L^{1}_{T}(\dot{B}^{\frac{d}{2}-1+\beta}_{2,1})}
&\leq C\|H^{\epsilon}\|_{L^{2}_{T}(\dot{B}^{\frac{d}{2}}_{2,1})}\|H^{\epsilon}\|_{L^{2}_{T}(\dot{B}^{\frac{d}{2}+\beta}_{2,1})}\nonumber\\
&\leq C P_{0}P_{\beta},\label{Cc-55}\\
\|I(\epsilon b^{\epsilon})H^{\epsilon}\cdot\nabla H^{\epsilon}\|_{L^{1}_{T}(\dot{B}^{\frac{d}{2}-1+\beta}_{2,1})}
&\leq C\epsilon \| b^{\epsilon}\|_{L^{\infty}_{T}(\dot{B}^{\frac{d}{2}}_{2,1})}
\|H^{\epsilon}\cdot\nabla H^{\epsilon}\|_{L^{1}_{T}(\dot{B}^{\frac{d}{2}-1+\beta}_{2,1})}\nonumber\\
&\leq C\epsilon^{\alpha}X_{\alpha}P_{0}P_{\beta}.\label{Cc-56}
\end{align}
Similarly,
\begin{align}\label{Cc-57}
\|I(\epsilon b^{\epsilon})|\nabla H^{\epsilon}|^{2}\|_{L^{1}_{T}(\dot{B}^{\frac{d}{2}-1+\beta}_{2,1})}
\leq C\epsilon^{\alpha}X_{\alpha}P_{0}P_{\beta}.
\end{align}
By \eqref{Cc-55}-\eqref{Cc-57}, we obtain that
\begin{align}\label{Cc-58}
&\Big\|\frac{1}{1+\epsilon b^{\epsilon}}\big(H^{\epsilon}\cdot\nabla H^{\epsilon}
-\frac{1}{2}\nabla |H^{\epsilon}|^{2}\big)\Big\|_{L^{1}_{T}(\dot{B}^{\frac{d}{2}-1+\beta}_{2,1})}\leq  C (1+\epsilon^{\alpha}X_{\alpha})P_{0}P_{\beta}.
\end{align}
Thanks to \eqref{Cc-47}-\eqref{Cc-54} and \eqref{Cc-58}, we end up with
\begin{align}\label{Cc-59}
\|T^{\epsilon}\|_{L^{1}_{T}(\dot{B}^{\frac{1}{2}+\beta}_{2,1})}
\leq \,&C\Big(P_{0}P_{\beta}+\epsilon^{\alpha}
\big(X_{\beta}+P_{\beta}\big)\big(X_{\alpha}+\epsilon^{-2\alpha/(1+2\alpha)}Y_{\alpha}^{p_{\alpha}}+\epsilon^{-\alpha}Y_{\alpha}^{1/\alpha}\big)\nonumber\\
&+\epsilon^{\alpha}X_{\alpha}P_{0}P_{\beta}+\epsilon^{\frac{2\alpha}{5+2\beta}}\big(P_{0}X_{\beta}
+P_{\beta}(X_{\alpha}+\epsilon^{-\frac{2\alpha}{1+2\alpha}}Y_{\alpha}^{p_{\alpha}})\big)\Big).
\end{align}

We now consider the case $d=2$.  We just have to deal with $K(\epsilon b^{\epsilon})b^{\epsilon}$,
the other terms can be treated by following the proof in the case $d=3.$ In fact, one just has to use \eqref{Cc-24} instead of \eqref{Cc-16},  that is, one  only needs to replace
 $\epsilon^{\frac{-2\alpha}{1+2\alpha}}Y_{\alpha}^{p_{\alpha}}$ by $\epsilon^{-\frac{1}{4}}Y_{\alpha}$.
Applying Bony's decomposition for $K(\epsilon b^{\epsilon})\nabla b^{\epsilon}$:
$$K(\epsilon b^{\epsilon}) \nabla b^{\epsilon}=\dot{T}_{\nabla b^{\epsilon}}K(\epsilon b^{\epsilon})+\dot{R}(\nabla b^{\epsilon},K(\epsilon b^{\epsilon}))
+\dot{T}_{K(\epsilon b^{\epsilon})}\nabla b^{\epsilon}.$$
Thanks to Remarks \ref{re2.3} and Lemma \ref{le2.2}, we get
\begin{align}
\|\dot{T}_{\nabla b^{\epsilon}}K(\epsilon b^{\epsilon})\|_{\dot{B}^{\beta}_{2,1}}&\leq C \|\nabla b^{\epsilon}\|_{\dot{B}^{-1}_{\infty,1}}
\|K(\epsilon b^{\epsilon})\|_{\dot{B}^{\beta+1}_{2,1}}\leq C\epsilon \| b^{\epsilon}\|_{\dot{B}^{0}_{\infty,1}}\| b^{\epsilon}\|_{\dot{B}^{\beta+1}_{2,1}},\nonumber\\
\|\dot{R}(\nabla b^{\epsilon},K(\epsilon b^{\epsilon}))\|_{\dot{B}^{\beta}_{2,1}}&\leq C \|\nabla b^{\epsilon}\|_{\dot{B}^{\beta}_{2,1}}
\|K(\epsilon b^{\epsilon})\|_{\dot{B}^{\frac{3-4\alpha}{7}}_{\frac{14}{3-4\alpha},1}}\leq C\epsilon\| b^{\epsilon}\|_{\dot{B}^{\beta+1}_{2,1}}
\| b^{\epsilon}\|_{\dot{B}^{\frac{3-4\alpha}{7}}_{\frac{14}{3-4\alpha},1}},\nonumber\\
\|\dot{T}_{K(\epsilon b^{\epsilon})}\nabla b^{\epsilon}\|_{\dot{B}^{\beta}_{2,1}}&\leq C \|\nabla b^{\epsilon}\|_{\dot{B}^{\beta}_{2,1}}
\|K(\epsilon b^{\epsilon})\|_{L^{\infty}}\leq C\epsilon\| b^{\epsilon}\|_{\dot{B}^{\beta+1}_{2,1}}\| b^{\epsilon}\|_{L^{\infty}}.\nonumber
\end{align}
By means of embeddings  $\dot{B}^{(3-4\alpha)/3}_{6/(3-4\alpha),1}\hookrightarrow \dot{B}^{(3-4\alpha)/7}_{14/(3-4\alpha),1}\hookrightarrow
\dot{B}^{0}_{\infty,1} \hookrightarrow L^{\infty}$, we have
\begin{align}
\|K(\epsilon b^{\epsilon})\nabla b^{\epsilon}\|_{\dot{B}^{\beta}_{2,1}}&\leq C \epsilon\| b^{\epsilon}\|_{\dot{B}^{\beta+1}_{2,1}}
\| b^{\epsilon}\|_{\dot{B}^{\frac{3-4\alpha}{7}}_{\frac{14}{3-4\alpha},1}}\nonumber\\
&\leq C \epsilon\| b^{\epsilon}\|_{\dot{B}^{\beta+1}_{2,1}}\Big(\| b^{\epsilon}_{BF}\|_{\dot{B}^{\frac{3-4\alpha}{7}}_{\frac{14}{3-4\alpha},1}}
+\| b^{\epsilon}_{HF}\|_{\dot{B}^{\frac{3-4\alpha}{3}}_{\frac{6}{3-4\alpha},1}}\Big).\nonumber
\end{align}
Using interpolation, H\"{o}lder inequality, and (ii) in Remark \ref{re2.2}, we deduce that
\begin{align}
&\!\!\!\!\!\!\!\!\!\!\!\!\!\! \!\!\!\! \!\!\! \!\!\!\!
 \Big\|\frac{K(\epsilon b^{\epsilon})\nabla b^{\epsilon}}{\epsilon}\Big\|_{L^{1}_{T}(\dot{B}^{\beta}_{2,1})}\nonumber\\
\leq\, & C
\|b^{\epsilon}\|_{L^{7/(3+3\alpha)}_{T}(\dot{B}^{\beta+1}_{2,1})}
\|b^{\epsilon}_{BF}\|_{L^{7/(4-3\alpha)}_{T}(\dot{B}^{(3-4\alpha)/7}_{14/(3-4\alpha),1})}\nonumber\\
&+C\|b^{\epsilon}\|_{L^{1/\alpha}_{T}(\dot{B}^{\beta+1}_{2,1})}
\|b^{\epsilon}_{HF}\|_{L^{1/(1-\alpha)}_{T}(\dot{B}^{(3-4\alpha)/3}_{6/(3-4\alpha),1})}\nonumber\\
\leq\, & C \epsilon^{\frac{6\alpha-1}{7}}\|b^{\epsilon}\|_{L^{7/(3+3\alpha)}_{T}(\tilde{B}^{\beta+1,7/(3+3\alpha)}_{\epsilon})}
\|b^{\epsilon}_{BF}\|_{L^{1}_{T}(\dot{B}^{2+\alpha}_{2,1})}^{(3-4\alpha)/7}
\|b^{\epsilon}_{BF}\|_{L^{4}_{T}(\dot{B}^{\alpha-3/4}_{\infty,1})}^{(4+4\alpha)/7}\nonumber\\
&+C\epsilon^{2\alpha-1}\|b^{\epsilon}\|_{L^{1/\alpha}_{T}(\tilde{B}^{\beta+1,1/\alpha}_{\epsilon})}
\|b^{\epsilon}_{HF}\|_{L^{1}_{T}(\dot{B}^{1+\alpha}_{2,1})}^{1-\frac{4\alpha}{3}}
\|b^{\epsilon}_{HF}\|_{L^{4}_{T}(\dot{B}^{\alpha-3/4}_{\infty,1})}^{\frac{4\alpha}{3}}\nonumber\\
\leq\, & C \epsilon^{\alpha}X_{\beta}\|b^{\epsilon}\|_{L^{1}_{T}(\tilde{B}^{1+\alpha,1}_{\epsilon})}^{\frac{3-4\alpha}{7}}
\big(\epsilon^{-\frac{1}{4}}\|b^{\epsilon}\|_{L^{4}_{T}(\dot{B}^{\alpha-3/4}_{\infty,1})}\big)^{\frac{4+4\alpha}{7}}\nonumber\\
&+C \epsilon^{\alpha}X_{\beta}\|b^{\epsilon}\|_{L^{1}_{T}(\tilde{B}^{1+\alpha,1}_{\epsilon})}^{1-\frac{4\alpha}{3}}
\big(\epsilon^{-\frac{1}{4}}\|b^{\epsilon}_{HF}\|_{L^{4}_{T}(\dot{B}^{\alpha-3/4}_{\infty,1})}\big)^{\frac{4\alpha}{3}}.\nonumber
\end{align}
Finally, we conclude  that
\begin{align}\label{Cc-60}
\|T^{\epsilon}\|_{L^{1}_{T}(\dot{B}^{\beta}_{2,1})}\leq\, & C\Big(P_{0}P_{\beta}+\epsilon^{\alpha}(X_{\beta}+P_{\beta})
(X_{\alpha}+\epsilon^{-\frac{1}{4}}Y_{\alpha})+\epsilon^{\alpha}X_{\alpha}P_{0}P_{\beta}\nonumber\\
&+\epsilon^{\frac{2}{2+\beta}}\big(P_{0}X_{\beta}+X_{0}P_{\beta}+P_{\beta}(X_{\alpha}+\epsilon^{-\frac{1}{4}}Y_{\alpha})\big)\Big).
\end{align}

 Now we set  $p=\frac{1}{\alpha}$ and $r=\frac{2}{2-\alpha}$.
Making use of interpolation, the following estimates hold,
if $d=3$,
\begin{align}
\|\nabla u^{\epsilon}\|_{L^{\frac{1}{\alpha}}_{T}(\dot{B}^{2\alpha-2}_{\infty,1})}&\leq C
\|\nabla\mathcal{P}^{\bot} u^{\epsilon}\|_{L^{\frac{1}{\alpha}}_{T}(\dot{B}^{2\alpha-2}_{\infty,1})}+
C \|\nabla\mathcal{P} u^{\epsilon}\|_{L^{\frac{1}{\alpha}}_{T}(\dot{B}^{2\alpha-1/2}_{2,1})}\nonumber\\
&\leq C \epsilon^{\alpha} \big(\epsilon^{-\alpha}Y_{\alpha}^{\frac{1}{\alpha}}\big)+ C P_{0};\nonumber
\end{align}
while if $d=2$,
\begin{align}
\|\nabla u^{\epsilon}\|_{L^{\frac{1}{\alpha}}_{T}(\dot{B}^{2\alpha-2}_{\infty,1})}&\leq C
\|\nabla \mathcal{P}^{\bot}u^{\epsilon}\|^{4\alpha}_{L^{4}_{T}(\dot{B}^{\alpha-\frac{7}{4}}_{\infty,1})}
\|\nabla \mathcal{P}^{\bot}u^{\epsilon}\|^{1-4\alpha}_{L^{\infty}_{T}(\dot{B}^{\alpha-2}_{\infty,1})}
+ C\|\nabla\mathcal{P} u^{\epsilon}\|_{L^{\frac{1}{\alpha}}_{T}(\dot{B}^{2\alpha-1}_{2,1})}\nonumber\\
&\leq C \epsilon^{\alpha} \big(X_{\alpha}+\epsilon^{-\frac{1}{4}}Y_{\alpha}\big)+ C P_{0}.\nonumber
\end{align}
Meanwhile, for any $d\geq 2$, we have
\begin{align}
\|\nabla u^{\epsilon}\|_{L^{2/(2-\alpha)}_{T}(L^{\infty})} \leq C \| u^{\epsilon}\|_{L^{2/(2-\alpha)}_{T}(\dot{B}^{d/2+1}_{2,1})}
 \leq C ( P_{\alpha}+X_{\alpha}).\nonumber
\end{align}
According to \eqref{Cc-32}, we thus have
\begin{align}
\left\{\begin{array}{l}
V^{1/\alpha,2/(2-\alpha)}_{\epsilon}\leq C \big(P_{0}^{\frac{1}{\alpha}}+\epsilon\big(\epsilon^{-\alpha}Y^{\frac{1}{\alpha}}_{\alpha}\big)^{\frac{1}{\alpha}}
+\big(\epsilon^{\alpha}(P_{\alpha}+X_{\alpha})\big)^{\frac{2}{2-\alpha}}\big) \quad\qquad \, \, \mathrm{if}\quad\, d=3,\\
V^{1/\alpha,2/(2-\alpha)}_{\epsilon}\leq C \big(P_{0}^{\frac{1}{\alpha}}+\epsilon\big(X_{\alpha}+\epsilon^{-\frac{1}{4}}Y_{\alpha}\big)^{\frac{1}{\alpha}}
+\big(\epsilon^{\alpha}(P_{\alpha}+X_{\alpha})\big)^{\frac{2}{2-\alpha}}\big) \quad \mathrm{if}\quad\, d=2.\nonumber
\end{array}
\right.
\end{align}
Plugging this latter inequality, \eqref{Cc-39}, and \eqref{Cc-59} into \eqref{Cc-31}, we eventually find that,
for $d=3$,
\begin{align}\label{Cc-61}
X_{\beta}\leq  & C \exp\Big\{C \big(P_{0}^{\frac{1}{\alpha}}+\epsilon\big(\epsilon^{-\alpha}Y^{\frac{1}{\alpha}}_{\alpha}\big)^{\frac{1}{\alpha}}
+\big(\epsilon^{\alpha}(P_{\alpha}+X_{\alpha})\big)^{\frac{2}{2-\alpha}}\big)\Big\}
\Big(X_{\beta}^{0}+P_{0}P_{\beta}\nonumber\\
&+\epsilon^{\alpha_{3}}\big(P_{0}X_{\beta}+(X_{\beta}+P_{\beta})
\big(X_{\alpha}+\epsilon^{-\alpha}Y^{1/\alpha}_{\alpha}+\epsilon^{-\frac{2\alpha}{1+2\alpha}}Y_{\alpha}^{p_{\alpha}}\big)+X_{\alpha}P_{0}P_{\beta}\big)\Big);
\end{align}
while for $d=2$, we can apply \eqref{Cc-42} and \eqref{Cc-60} to obtain that
\begin{align}\label{Cc-62}
X_{\beta}\leq\, & C \exp\Big\{C \big(P_{0}^{\frac{1}{\alpha}}+\epsilon\big(X_{\alpha}+\epsilon^{-\frac{1}{4}}Y_{\alpha}\big)^{\frac{1}{\alpha}}
+\big(\epsilon^{\alpha}(P_{\alpha}+X_{\alpha})\big)^{\frac{2}{2-\alpha}}\big)\Big\}
 \Big(X_{\beta}^{0}+P_{0}P_{\beta}\nonumber\\
&+\epsilon^{\alpha_{2}}\big(P_{0}X_{\beta}+X_{0}P_{\beta}
+(X_{\beta}+P_{\beta})
\big(X_{\alpha}+\epsilon^{-\frac{1}{4}}Y_{\alpha}\big)+X_{\alpha}P_{0}P_{\beta}\big)\Big),
\end{align}
where  $\alpha_{d}=\frac{2\alpha}{2+d+2\alpha}$ with $d=2$ or $3$.

\subsection{Bootstrap}

 The remaining part of  the   proof  works for both dimensions $d= 3$ and   $d=2$.
 Set
 $$
 X  :=X_{0}+X_{\alpha},\,\,\, V :=V_{0}+V_{\alpha},\,\,\, W :=W_{0}+W_{\alpha},\, \,\,
 X^{0}:=X^{0}_{0}+X^{0}_{\alpha}.
 $$
  With these new notations, by combining together   \eqref{Cc-5} or  \eqref{Cc-6},
 \eqref{Cc-29} or \eqref{Cc-30}, and \eqref{Cc-61} or  \eqref{Cc-62}, we conclude that
\begin{align}
W\leq\, &C e^{C(V+X)}\Big(W^{2}\big(1+W+W^{2}+W^{4}+\epsilon^{\alpha_{d}}(V+V^{2}+V^{4})\big)\nonumber\\
&+\epsilon^{\alpha_{d}}\big(X^{2}+(X^{0})^{2}+X^{4}+V(X+X^{2}+X^{0}+XV^{2}+V^{2}+V^{3})\big)\Big),\label{Cc-66}\\
X\leq\,& C \exp\Big\{C\Big(\epsilon(X+X^{2})^{\frac{1}{\alpha}}+\epsilon(X(V+W)^{2})^{\frac{1}{\alpha}}+(\epsilon^{\alpha}X)^{\frac{2}{2-\alpha}}\Big)\Big\}\nonumber\\
&\quad\times\exp\Big\{C\Big((V+W)^{\frac{1}{\alpha}}+\epsilon(X^{0}+(V+W)^{2})^{\frac{1}{\alpha}}
+\epsilon^{\frac{2\alpha}{2-\alpha}}(V+W)^{\frac{2}{2-\alpha}}\Big)\Big\}\nonumber\\
&\quad\times\Big(X^{0}+(V+W)\big(V+W+\epsilon^{\alpha_{d}}(X^{0}+V^{2}+W^{2})\big)+\epsilon^{\alpha_{d}}X\big(X^{0}\nonumber\\
&+X+X^{2}+(1+X)(W+V+W^{2}+V^{2})+V^{3}+W^{3}\big)\Big).\label{Cc-67}
\end{align}

 In order to get a bound for $(b^{\epsilon},u^{\epsilon},H^{\epsilon})$, 
 we need a bootstrap argument.
More precisely, we have the following lemma.

\begin{Lemma}\label{lemma31}
Suppose that $(v,B)\in F^{\frac{d}{2}}_{T_{0}}\cap F^{\frac{d}{2}+\alpha}_{T_{0}}$  for some finite or infinite $T_{0}.$
Then,  there exists an $\epsilon_{0}>0$,  depending only on $\alpha,d,V(T_{0}),$ and the norm of $(b_{0},\mathcal{P}^{\bot}u_{0},H_{0})$
in
$$\dot{B}^{\frac{d}{2}-1}_{2,1}\cap\dot{B}^{\frac{d}{2}+\alpha}_{2,1}\times(\dot{B}^{\frac{d}{2}-1}_{2,1}\cap\dot{B}^{\frac{d}{2}-1+\alpha}_{2,1})^{d}
\times(\dot{B}^{\frac{d}{2}-1}_{2,1}\cap\dot{B}^{\frac{d}{2}-1+\alpha}_{2,1})^{d}$$
such that if $\epsilon\leq\epsilon_{0}, (b^{\epsilon},u^{\epsilon},H^{\epsilon})\in E^{\frac{d}{2}}_{\epsilon,T}\cap E^{\frac{d}{2}+\alpha}_{\epsilon,T},$
and $\epsilon |b^{\epsilon}|\leq 3/4$  for some $T\leq T_{0}$,   the following estimates hold with the constant $C=C(d,\mu,\lambda,\nu,P,\alpha)$
appearing in \eqref{Cc-66} and \eqref{Cc-67}:
\begin{align}
X_{T}\leq \,&X_{M}:=16 Ce^{CV^{\frac{1}{\alpha}(T_{0})}}\big(X^{0}+V^{2}(T_{0})\big),\nonumber\\
\epsilon^{-\alpha_{d}}W_{T}\leq\, & W_{M}:= 4Ce^{C(V(T_{0})+X_{M})}
\Big(X^{2}_{M}+X^{2}_{0}+X^{4}_{M}\nonumber\\
&\qquad \quad +V(T_{0})\big(X_{M}+X^{2}_{M}+X_{0}+X_{M}V^{2}(T_{0})+V^{2}(T_{0})+V^{3}(T_{0})\big)\Big).\nonumber
\end{align}
\end{Lemma}
\begin{proof}
Let $I:=\{t\leq T|X(t)\leq X_{M}\ \textrm{and}\ W(t) \leq \epsilon ^{\alpha_{d}} W_{M}\}$. Obviously, $X$ and
$W$ are continuous nondecreasing functions so that if, say, $C\geq 1$, then $I$ is a closed interval of $\mathbb{R}^{+}$
with lower bound 0.

Let $T^{*}:=\sup I.$ Choose $\epsilon$ sufficiently small so that the following conditions are satisfied:
\begin{align}
&Ce^{C(V(T_{0})+X_{M})}\epsilon^{\alpha_{d}}W_{M}\Big(1+W_{M}+W^{2}_{M}+W^{4}_{M}\nonumber\\
&\qquad \qquad\qquad\qquad \qquad+\epsilon^{\alpha_{d}}
\big(V(T_{0})+V^{2}(T_{0})+V^{4}(T_{0})\big)\Big)\leq \frac{1}{2},\nonumber\\
& \exp\big\{C(\epsilon(X_{M}+X_{M}^{2})^{1/\alpha}+\epsilon(X_{M}(V(T_{0})+\epsilon^{\alpha_{d}}W_{M})^{2})^{1/\alpha}
+(\epsilon^{\alpha}X_{M})^{2/(2-\alpha)})\big\}\leq 2,\nonumber\\
&\exp\Big\{C((V(T_{0})+\epsilon^{\alpha_{d}}W_{M})^{1/\alpha}+\epsilon(X^{0}+(V(T_{0})+\epsilon^{\alpha_{d}}W_{M})^{2})^{1/\alpha}\nonumber\\
&\qquad\qquad\qquad\qquad\qquad +\epsilon^{2\alpha/(2-\alpha)}(V(T_{0})+\epsilon^{\alpha_{d}}W_{M})^{2/(2-\alpha)})\Big\}
\leq 2e^{C V^{\frac{1}{\alpha}}(T_{0})},\nonumber\\
&X_{0}+\big(V(T_{0})+\epsilon^{\alpha_{d}}W_{M}\big)\big(V(T_{0})+\epsilon^{\alpha_{d}}W_{M}
+\epsilon^{\alpha_{d}}\big(X^{0}+V^{2}(T_{0})+\epsilon^{2\alpha_{d}}W^{2}_{M}\big)\big)\nonumber\\
&\quad \qquad\qquad\leq 2\big(X^{0}+V^{2}(T_{0})\big),\nonumber\\
&Ce^{C V^{1/\alpha}(T_{0})}\epsilon^{\alpha_{d}}\Big\{X^{0}+X_{M}+X^{2}_{M} +(1+X_{M})
\nonumber\\
&\quad\qquad\times\big(V(T_{0})+V^{2}(T_{0})+\epsilon^{\alpha_{d}}W_{M}+\epsilon^{2\alpha_{d}}W^{2}_{M}\big)
+V^{3}(T_{0})+\epsilon^{3\alpha_{d}}W^{3}_{M}\Big\}\leq \frac{1}{12}.\nonumber
\end{align}
Then, by the   \eqref{Cc-66} and \eqref{Cc-67}, we obtain that
\begin{align}
X(T^{*})\leq \,& 12 Ce^{CV^{1/\alpha}(T_{0})}\big(X^{0}+V^{2}(T)\big),\nonumber\\
W(T^{*})\leq\,& 2Ce^{C\big(V(T_{0})+X_{M}\big)}\epsilon^{\alpha_{d}}\big(X^{2}_{M}+X^{2}_{0}+X^{4}_{M}\nonumber\\
&+V(T_{0})\big(X_{M}+X^{2}_{M}+X_{0}+X_{M}V^{2}(T_{0})+V^{2}(T_{0})+V^{3}(T_{0})\big)\big).\nonumber
\end{align}
In other words, at time $T^{*}$ the desired inequalities are strict. Hence, we must have $T^{*}=T.$
\end{proof}

\subsection{Continuation argument}

First, we have to establish the existence of a local solution in $E^{\frac{d}{2}}_{\epsilon,T}\cap E^{\frac{d}{2}+\alpha}_{\epsilon,T}$.
Making the change of function $a^{\epsilon}=\epsilon b^{\epsilon}$,  Theorem \ref{ThA} will enable us to  get a
local solution $(b^{\epsilon},u^{\epsilon},H^{\epsilon})$ on $[0,T]\times \mathbb{R}^{d}$ which belongs to $E^{\alpha}_{T}$ and satisfies
 $$1+\epsilon\underset{(t,x)\in[0,T]\times \mathbb{R}^{d}}{\inf} b^{\epsilon}(t,x)>0.$$

Moreover, due to  the facts that $b_{0}\in\dot{B}^{\frac{d}{2}-1}_{2,1}$ and $\partial_{t}b^{\epsilon}+u^{\epsilon}\cdot\nabla b^{\epsilon}\in
 L^{1}_T(\dot{B}^{\frac{d}{2}-1}_{2,1}),$  we  readily get  that $b^{\epsilon} \in \mathcal{C}([0,T];\dot{B}^{\frac{d}{2}-1}_{2,1})$.
 Therefore, $(b^{\epsilon},u^{\epsilon},H^{\epsilon})\in E^{\frac{d}{2}}_{\epsilon,T}\cap E^{\frac{d}{2}+\alpha}_{\epsilon,T}.$
 Now, assuming that we have $(v,B)\in F^{\frac{d}{2}}_{T_{0}}\cap F^{\frac{d}{2}+\alpha}_{T_{0}}$ for some $T_{0}\in (0,+\infty]$,
 we shall prove that the  lifespan $T_{\epsilon}$ satisfies $T_{\epsilon}\geq T_{0}$ if $\epsilon$ is sufficiently small,
 where $T_{\epsilon}$ is the supremum of the set
 $$
 \Big\{T\in\mathbb{R}^{+}\big|\,  (b^{\epsilon},u^{\epsilon},H^{\epsilon})\in E^{\frac{d}{2}}_{\epsilon,T}\cap E^{\frac{d}{2}+\alpha}_{\epsilon,T}
 \ \ \text{and} \ \ \forall\,(t,x)\in[0,T]\times\mathbb{R}^{d},|\epsilon b^{\epsilon}|\leq \frac34\Big \}.
 $$
  We suppose
  that $T_{\epsilon}$ is finite and satisfies $T_{\epsilon}\leq T_{0}.$ Thanks to
  Lemma \ref{lemma31}, we have, for any $T<T_{\epsilon}$ and $\epsilon\leq\epsilon_{0}$, that
  $$X(T)\leq X_{M}\quad \text{and} \quad W(T)\leq \epsilon^{\alpha_{d}}W_{M.}$$
 From the first inequality and \eqref{Cc-12}, we conclude that
 $$\epsilon\|b^{\epsilon}\|_{L^{\infty}_{T}(\dot{B}^{\frac{d}{2}}_{2,1})}\leq \epsilon^{\alpha} X_{M}.$$
 Obviously,  we require that,  for $\epsilon_{0}$ sufficiently small, 
 $$
 1+\epsilon\underset{(t,x)\in [0,T_{\epsilon}]\times \mathbb{R}^{d}}{\inf}|b^{\epsilon}(t,x)|>0.
 $$
 Since $\epsilon b^{\epsilon}\in L^{\infty}_{T_{\epsilon}}(\dot{B}^{\frac{d}{2}}_{2,1}\cap\dot{B}^{\frac{d}{2}+\alpha}_{2,1}),$
$\nabla u^{\epsilon}\in L^{1}_{T_{\epsilon}}(\dot{B}^{\frac{d}{2}}_{2,1}) $, and
$\nabla u^{\epsilon}\in L^{1}_{T_{\epsilon}}( \dot{B}^{\frac{d}{2}}_{2,1}) $, the continuation criterion stated in Proposition \ref{PropB}
ensures that $(b^{\epsilon},u^{\epsilon},H^{\epsilon})$ may be continued beyond $T_{\epsilon}$,
which contradicts  definition of $T_{\epsilon}$.   Therefore, $T_\epsilon \geq T_{0}$ for $\epsilon\leq\epsilon_{0}.$

The proof of Theorem \ref{ThB} is  now completed.
\end{proof}

\bigskip

\appendix

\section{Basic Facts on Besov Spaces}\label{AppA}

In this section we recall the definition and some basic properties of homogeneous Besov space.
Most of the materials stated below can be found in the books \cite{BCD,Ch,T}. We collect them below for the reader's convenience.

\begin{Def}[\cite{T}] \label{def2.1}
Let $\left\{\phi_j\right\}_{j\in\mathbb{Z}}$ be the
Littlewood-Paley dyadic decomposition of unity that satisfies $\hat{\phi}\in
C_0^\infty ({B_2}\setminus{B_{{1}/{2}}})$, $\hat{\phi}_j(\xi)=\hat{\phi}%
(2^{-j}\xi)$ and $\sum_{j\in\mathbb{Z}}\hat{\phi}_j(\xi)=1$ for any $%
\xi\neq0$, where $\hat{\phi}$ is the Fourier transform of $\phi$ and $B_r$
is the ball with radius $r$ centered at the origin. The homogeneous
Besov space is defined as
$$
\dot{B}_{p, q}^s:=\left\{f\in
\mathcal{S}^{\prime }/\mathbf{P}: \|f\|_{\dot{B}_{p, q}^s}<\infty\right\}
$$
with the  norm
\begin{align*}
\|f\|_{\dot{B}_{p, q}^s}:=\left\{
\begin{array}{l}
 \Big\{\sum_{j\in\mathbb{Z}%
}\|2^{js}\phi_j\ast f \|^q_{L^p}\Big\}^{\frac{1}{q}}, \quad 1\leq q< \infty,  \\
  \sup_{j\in\mathbb{Z}}\|2^{js}\phi_j\ast f \|_{L^\infty} ,\quad \, \ \quad q= \infty,
\end{array}\right.
\end{align*}
for all $s\in\mathbb{R}$ and  $1\leq p\leq\infty$, where $\mathcal{S}^{\prime }$
is the space of tempered distributious and $\mathbf{P}$ is the space of polynomials.
\end{Def}

\begin{Def}[\cite{T}]\label{def2.2}
 For $T >0,s\in \mathbb{R}$, and $1\leq r,\rho\leq\infty,$ we set
 $$\|u\|_{\widetilde{L}^{\rho}_{T}(\dot{B}^{s}_{p,r})}:=
 \|2^{js}\|\dot{\Delta}_{j}u\|_{L^{\rho}_{T}(L^{p})}\|_{l^{r}(\mathbb{Z})}.$$
 We can then define the space $\widetilde{L}^{\rho}_{T}(\dot{B}^{s}_{p,r})$
 as the set of tempered distributions $u$ over $(0,T)\times \mathbb{R}^{d}$ such that
 $\underset{j\rightarrow-\infty}{\lim}\dot{S}_{j}u=0 $ in $L^{\rho}_T(L^{\infty})$
 and $\|u\|_{\widetilde{L}^{\rho}_{T}(\dot{B}^{s}_{p,r})}<+\infty$.
\end{Def}

\begin{Remark}[\cite{T}]\label{rem2.1}
The spaces $\widetilde{L}^{\rho}_{T}(\dot{B}^{s}_{p,r})$ may be linked with the more classical spaces
 $L^{\rho}_{T}(\dot{B}^{s}_{p,r})$ via the Minkowski inequality:
 $$
 \|u\|_{\widetilde{L}^{\rho}_{T}(\dot{B}^{s}_{p,r})}\leq \|u\|_{L^{\rho}_{T}(\dot{B}^{s}_{p,r})}\ \ \text{if} \ \ r\geq \rho,\quad
 \|u\|_{\widetilde{L}^{\rho}_{T}(\dot{B}^{s}_{p,r})}\geq\|u\|_{L^{\rho}_{T}(\dot{B}^{s}_{p,r})} \ \ \text{if} \ \ r\leq \rho.
 $$
 The general principles is that all the properties of continuity for the product, composition, remainder, and paraproduct remain true in these spaces.
 The exponent $\rho$ just has to behave according to H\"{o}lder's inequality for the
 time variable.
\end{Remark}

\begin{Proposition}[\cite{DC}]\label{prop2.1}
The following properties hold:
\begin{itemize}
\item[(1)] Derivation: there exists a universal constant C such that
$$C^{-1}\|u\|_{\dot{B}^{s}_{p,1}}\leq \|\nabla u\|_{\dot{B}^{s-1}_{p,1}}\leq C \|u\|_{\dot{B}^{s}_{p,1}};$$

\item[(2)] Fractional derivation: let $\Lambda:=\sqrt{-\Delta}$ and $\sigma\in \mathbb{R}$,
then the operator $\Lambda^{\sigma}$ is an isomorphism from $\dot{B}^{s}_{p,1}$ to $\dot{B}^{s-\sigma}_{p,1}$;

\item[(3)] Sobolev embeddings: if $p_{1}<p_{2}$ then $\dot{B}^{s}_{p,1}\hookrightarrow\dot{B}^{s-d(1/p_{1}-1/p_{2})}_{p,2}$;

\item[(4)] Interpolation : $[\dot{B}^{s_{1}}_{p,1},\dot{B}^{s_{1}}_{p,1}]_{\theta}=\dot{B}^{\theta s_{1}+(1-\theta)s_{2}}_{p,1}$;

\item[(5)] Algebraic properties: for $s>0,\dot{B}^{s}_{p,1}\cap L^{\infty}$ is an algebra;

\item[(6)] Scaling properties:
\begin{itemize}
\item[(a)] for all $\lambda>0$ and $u\in \dot{B}^{s}_{p,1}$, we have
$$\|u(\lambda\cdot)\|_{\dot{B}^{s}_{p,1}}\approx (\lambda)^{s-d/p}\|u\|_{\dot{B}^{s}_{p,1}},$$
\item[(b)] for $u=u(t,x)$ in $L^{r}_T(\dot{B}^{s}_{p,1})$, we have
$$\|u(\lambda^{a}\cdot,\lambda^{b}\cdot)\|_{L^{r}_{T}(\dot{B}^{s}_{p,1})}
\approx \lambda^{b(s-d/p)-a/r}\|u\|_{L^{r}_{\lambda^{a}T}(\dot{B}^{s}_{p,1})}.
$$
\end{itemize}
\end{itemize}
\end{Proposition}

Let us state some continuity results for the product.
\begin{Proposition}[\cite{DC}]\label{prop2.2}
If $u\in \dot{B}^{s_{1}}_{p_{1},1} $ and $v\in \dot{B}^{s_{2}}_{p_{2},1}$ with
$1\leq p_{1}\leq p_{2}\leq+\infty,s_{1}\leq d/p_{1},s_{2}\leq d/p_{2}$ and $s_{1}+s_{1}>0$.
 then $uv\in \dot{B}^{s_{1}+s_{2}-d/p_{1}}_{p_{2},1}$ and
 $$
 \|uv\|_{\dot{B}^{s_{1}+s_{2}-d/p_{1}}_{p_{2},1}}\leq C\|u\|_{\dot{B}^{s_{1}}_{p_{1},1}} \|v\|_{\dot{B}^{s_{2}}_{p_{2},1}}.
 $$
 If $u\in \dot{B}^{s_{1}}_{p_{1},1}\cap\dot{B}^{s_{2}}_{p_{2},1} $ and
  $v\in \dot{B}^{t_{1}}_{p_{2},1}\cap\dot{B}^{t_{2}}_{p_{2},1}$ with
  $1\leq p_{1}, p_{2}\leq+\infty,s_{1},t_{1}\leq d/p_{1}$ and
  $s_{1}+t_{2}=s_{2}+t_{1}>d \max\big\{0,\frac{1}{p_{1}}+\frac{1}{p_{2}}-1\big\}$, then
  $uv\in \dot{B}^{s_{1}+t_{2}-d/p_{1}}_{p_{2},1}$ and
 $$
 \|uv\|_{\dot{B}^{s_{1}+t_{2}-d/p_{1}}_{p_{2},1}}\leq C\|u\|_{\dot{B}^{s_{1}}_{p_{1},1}} \|v\|_{\dot{B}^{t_{2}}_{p_{2},1}}
 +\|u\|_{\dot{B}^{s_{2}}_{p_{2},1}} \|v\|_{\dot{B}^{t_{1}}_{p_{1},1}}.
 $$
 Moreover, if $s_{1}=0$ and $ p_{1}=+\infty$, then $\|u\|_{\dot{B}^{0}_{\infty,1}}$ may be replaced with $\|u\|_{L^{\infty}}$.
\end{Proposition}

\begin{Def}[\cite{T}]\label{def2.3}
Let $s\in \mathbb{R},\alpha>0$ and $1\leq r\leq +\infty$ and
$$\|u\|_{\tilde{B}^{s,r}_{\alpha}} := \sum_{q\in \mathbb{Z}}2^{qs}\max\{\alpha,2^{-q}\}^{1-2/r}\|\dot{\Delta}_{q}u\|_{L^{2}}.$$

Let $m=-[d/2+2-2/r-s],$ we then define
\begin{align*}
 &\tilde{B}^{s,r}_{\alpha}(\mathbb{R}^{d}):=\big\{u\in \mathcal{S}'(\mathbb{R}^{d})\big|\|u\|_{\tilde{B}^{s,r}_{\alpha}}<+\infty\big\}\quad\quad \quad\qquad   \text{if} \quad m<0,\\
& \tilde{B}^{s,r}_{\alpha}(\mathbb{R}^{d})
:=\big\{u\in \mathcal{S}'(\mathbb{R}^{d})/\mathbf{P}_{m}(\mathbb{R}^{d})\big|\|u\|_{\tilde{B}^{s,r}_{\alpha}}<+\infty\big\}\quad \text{if} \quad m\geq0.
  \end{align*}
  \end{Def}

We will use the   following high-low frequencies decomposition:
$$u_{BF} := \sum_{q\leq[-\log_{2}\alpha]}\dot{\Delta}_{q}u, \quad u_{HF} := \sum_{q>[-\log_{2}\alpha]}\dot{\Delta}_{q}u. $$

\begin{Remark}[\cite{DC}]\label{re2.2}
\begin{itemize}
\item[(i)] $\tilde{B}^{s,2}_{\alpha}=\dot{B}^{s}_{2,1}$;

\item[(ii)] If $r\geq 2$ then $\tilde{B}^{s,r}_{\alpha}=\dot{B}^{s+2/r -1}_{2,1}\cap\dot{B}^{s}_{2,1}$ and
$$\|u\|_{\tilde{B}^{s,r}_{\alpha}}\approx \|u\|_{\dot{B}^{s+2/r -1}_{2,1}}+\alpha^{1-2/r}\|u\|_{\dot{B}^{s}_{2,1}},$$
If $r\leq 2$ then $\tilde{B}^{s,r}_{\alpha}=\dot{B}^{s+2/r -1}_{2,1}+\dot{B}^{s}_{2,1}$ and
$$\|u\|_{\tilde{B}^{s,r}_{\alpha}}\approx \|u_{BF}\|_{\dot{B}^{s+2/r -1}_{2,1}}+\alpha^{1-2/r}\|u_{HF}\|_{\dot{B}^{s}_{2,1}};$$

\item[(iii)] For all $\lambda >0$ and $u\in \tilde{B}^{s,r}_{\alpha}$, we have
$$\|u(\lambda\cdot)\|_{\tilde{B}^{s,r}_{\alpha}}\approx \lambda^{s-d/2+2/r-1}\|u\|_{\tilde{B}^{s,r}_{\lambda\alpha}}.$$
\end{itemize}
\end{Remark}
\smallskip
The paraproduct between $u$ and $v$ is given by
$$\dot{T}_{u}v := \sum_{q\in \mathbb{Z}}\dot{S}_{q-1}u\dot{\Delta}_{q}v, \quad
 \dot{R}(u,v) := \sum_{q\in \mathbb{Z}}\dot{\Delta}_{q}u\tilde{\dot{\Delta}}_{q} v,$$
with $\tilde{\dot{\Delta}}_{q}=\dot{\Delta}_{q-1}+\dot{\Delta}_{q}+\dot{\Delta}_{q+1}$.
We have the following Bony decomposition (modulo a polynomial):
$$uv=\dot{T}_{u}v+\dot{T}_{u}v+ \dot{R}(u,v).$$
The notation $\dot{T}'_{u}v:=\dot{T}_{u}v+\dot{R}(u,v)$ will be employed likewise.

\begin{Remark}[\cite{DC}]\label{re2.3}
Let $1\leq p_{1},p_{2}\leq+\infty$. For all $ s_{2}\in \mathbb{R}$ and $s_{1}\leq d/p_{1}$, we have
$$\|\dot{T}_{u}v\|_{\dot{B}^{s_{1}+s_{2}}_{p_{2},1}}\leq C \|u\|_{\dot{B}^{s_{1}}_{\infty,\infty}}\|v\|_{\dot{B}^{s_{2}}_{p_{2},1}}.$$
If $(s_{1},s_{2})\in \mathbb{R}^{2}$ satisfies $s_{1}+s_{2}>d \max\big\{0,\frac{1}{p_{1}}+\frac{1}{p_{2}}-1\big\}$,  then
$$\|\dot{R}(u,v)\|_{\dot{B}^{s_{1}+s_{2}-d/p_{1}}_{p_{2},1}}\leq C \|u\|_{\dot{B}^{s_{1}}_{p_{1},\infty}}\|v\|_{\dot{B}^{s_{2}}_{p_{2},1}}.$$
\end{Remark}

\begin{Proposition}[\cite{BCD}]\label{prop2.3}
Let K be  a compact subset of $\mathbb{R}^{d}$. Denote by $B^{s}_{p,r}(K)$ \emph{[}resp.,
$\dot{B}^{s}_{p,r}(K)$\emph{]} the set of distributions $u$ in $B^{s}_{p,r}(resp.,\dot{B}^{s}_{p,r})$, the support of which is included in $K$.
If $s>0$,  then the spaces $B^{s}_{p,r}(K)\,and\,\dot{B}^{s}_{p,r}(K)$ coincide.
Moreover, a constant C exists such that for any $u$ in $\dot{B}^{s}_{p,r}(K)$,
$$\|u\|_{B^{s}_{p,r}}\leq C(1+|K|)^{\frac{s}{d}}\|u\|_{\dot{B}^{s}_{p,r}}.$$
Here ${B}_{p, r}^s$ denotes the inhomogeneous Besov space.
\end{Proposition}

\begin{Proposition}[\cite{BCD}]\label{prop2.4}
If $s'< s$, then for all $\varphi \in \mathcal{C}_{c}^{\infty}(\mathbb{R}^{d}),$ multiplication by $\varphi$ is a compact operator from $\dot{B}^{s}_{p,\infty}$
to $\dot{B}^{s'}_{p,1}$.
\end{Proposition}

\begin{Proposition}[Fatou's property \cite{BCD}]\label{prop2.5}
Let $(s_{1},s_{2})\in \mathbb{R}^{2}$ and $1\leq p_{1},p_{2},r_{1},r_{2}\leq\infty$. Assume that
$(s_{1},p_{1},r_{1})$ satisfies $s_{1}<\frac{d}{p_{1}},\,\,or\,\,s_{1}=\frac{d}{p_{1}}\,and\,r_{1}=1.$ the space $\dot{B}^{s_{1}}_{p_{1},r_{1}}
\cap\dot{B}^{s_{2}}_{p_{2},r_{2}}$ endowed with the norm $\|\cdot\|_{\dot{B}^{s_{1}}_{p_{1},r_{1}}}+\|\cdot\|_{\dot{B}^{s_{2}}_{p_{2},r_{2}}}$
is then complete and satisfis the Fatou's property: If $(u_{n})_{n\geq 1}$
is bounded sequence of $\dot{B}^{s_{1}}_{p_{1},r_{1}}\cap\dot{B}^{s_{2}}_{p_{2},r_{2}}$,
then an element $u$ of $\dot{B}^{s_{1}}_{p_{1},r_{1}}\cap\dot{B}^{s_{2}}_{p_{2},r_{2}}$
and a subsequence $u_{\psi(n)}$ exist such that
$\lim_{n\rightarrow\infty} u_{\psi(n)}=u$ in $\mathcal{S}'$ and $\|u\|_{\dot{B}^{s_{k}}_{p_{k},r_{k}}}\leq C  \liminf_{n\rightarrow\infty}
\|u_{\psi(n)}\|_{\dot{B}^{s_{k}}_{p_{k},r_{k}}}$ for $k=1,2$.
\end{Proposition}

\begin{Lemma}[Bernstein inequality \cite{BCD}]\label{le2.1}
Let $\mathcal{C}$ be an annulus and $B$ a ball, A constant $C$ exists such that for
any nonnegative integer $k$, any couple $(p,q)$ in $[1,\infty]^{2}$ with $q\geq p\geq 1$, and any function $u$ of $L^{p}$, we have
\begin{align*}
 &  {\rm Supp}\, \widehat{u}\subset \lambda B\Rightarrow \|D^{k}u\|_{L^{q}}
:=\sup_{|\alpha|=k}\|\partial^{\alpha}u\|_{L^{q}}
\leq C^{k+1}\lambda ^{k+d(\frac{1}{p}-\frac{1}{q})}\|u\|_{L^{p}},\\
&{\rm Supp}\,  \widehat{u}\subset \lambda \mathcal{C}\Rightarrow C^{-k-1}\lambda^{k}\|u\|_{L^{p}}
\leq\|D^{k}u\|_{L^{p}}\leq C^{k+1}\lambda ^{k}\|u\|_{L^{p}}.
\end{align*}
\end{Lemma}

We also need the following  composition propositions  in $\dot{B}^{s}_{p,1}$.

\begin{Lemma}[\cite{DC}]\label{le2.2}
Let $s>0$, $p\in [1,+\infty],   u\in\dot{B}^{s}_{p,1}\cap L^{\infty}$, and $F \in   W^{[s]+2,\infty}_{loc} $
 such that $F(0)=0$. Then $F(u)\in \dot{B}^{s}_{p,1} $ and there exists
a constant $C=C(s,p,d,F, \|u\|_{L^{\infty}})$ such that
$$\|F(u)\|_{\dot{B}^{s}_{p,1}}\leq C \|u\|_{\dot{B}^{s}_{p,1}}.$$
\end{Lemma}

\begin{Lemma}[\cite{DC}]\label{le2.3}
Let $f$ be a smooth function such that $f'(0)=0$, $s$ be a positive real number and $(p,r)$ in $[1,\infty]^{2}$
  such that  $s<\frac{d}{p}$, or $s=\frac{d}{p}$ and $r=1$.  For any couple $(u,v)$ of functions
in $\dot{B}^{s}_{p,r} \cap L^{\infty}$, the function $f\circ v-f\circ u$ then belongs to $\dot{B}^{s}_{p,r} \cap L^{\infty}$ and
\begin{align}
\|f(u)-f(v)\|_{\dot{B}^{s}_{p,r}}\leq\, &C\big(\|v-u\|_{\dot{B}^{s}_{p,r}}\sup_{\tau\in[0,1]}\|u+\tau(v-u)\|_{L^{\infty}}\nonumber\\
&+\|v-u\|_{L^{\infty}}\sup_{\tau\in[0,1]}\|u+\tau(v-u)\|_{\dot{B}^{s}_{p,r}}\big),\nonumber
\end{align}
where $C$ depends on $f'', \|u\|_{L^{\infty}}$ and $\|v\|_{L^{\infty}}.$
\end{Lemma}

\begin{Lemma}[\cite{DC}]\label{le2.4}
Let $\mathcal{C}'$ be an annulus and $(u_{j})_{j\in \mathbb{Z}}$ be a sequence of functions such that
$${\rm Supp}\,\hat{u}_{j}\subset 2^{j}\mathcal{C}'\quad and\quad \big\| (2^{js}\|u_{j}\|_{L^{p}})_{j\in \mathbb{Z}}\big\|_{l^{r}}<\infty.$$
If the series $\sum_{j\in \mathbb{Z}} u_{j}$ converges in $\mathcal{S}'$  to some $u$ in $\mathcal{S}'_{h}$, then $u$ is in $\dot{B}^{s}_{p,r}$
and
$$\|u\|_{\dot{B}^{s}_{p,r}}\leq C \big\| (2^{js}\|u_{j}\|_{L^{p}})_{j\in \mathbb{Z}}\big\|_{l^{r}}.$$
\end{Lemma}

\bigskip


\section{Local Existence Results for Incompressible MHD Equations}\label{lo-imhd}

\begin{Proposition}\label{imhd}
Let $\alpha \geq 0$, $d=2$ or $3$. and  $(u_{0},b_{0})\in \dot{B}^{\frac{d}{2}-1}_{2,1}\cap \dot{B}^{\frac{d}{2}+\alpha-1}_{2,1}$
be two divergence-free vector fields. Then there exists a time $T$ and a unique local solution $(u,b)$ to the following initial value problem
\begin{align}\label{eqb1}
\left\{\begin{array}{l}
\partial_{t} u-\Delta u+ u\cdot\nabla u-b\cdot\nabla b-\nabla P=0,\\
\partial_{t} b-\Delta b +u\cdot\nabla b-b\cdot\nabla u=0,\\
\dv u=0,\quad \dv b=0,\\
(u,b)|_{t=0}=(u_{0},b_{0}), \quad x\in \mathbb{R}^d,
\end{array}\right.
\end{align}
such that
$(u,b)\in F^{\frac{d}{2}}_{T}\cap F^{\frac{d}{2}+\alpha}_{T}$
and there exist two constants c and C depending only on $d$ such that the time $T$ is bounded from below by
$$\sup\bigg\{T'>0\bigg|\sum_{j\in \mathbb{Z}}2^{j(\frac{d}{2}-1)}\big(1-e^{-c2^{2j}T'}\big)^{\frac{1}{2}}\big(\|\dot{\Delta}_{j}u_{0}\|_{L^{2}}
+\|\dot{\Delta}_{j}b_{0}\|_{L^{2}}\big)\bigg\}\leq C.$$
\end{Proposition}

\begin{proof}
We shall adopt the fixed point method. Denote by $e^{t\Delta}$ the semi-group of the heat equation.
Let $(u_{L},b_{L})\in F^{\frac{d}{2}}\cap F^{\frac{d}{2}+\alpha}$ be the solution of
\begin{align}
\left\{\begin{array}{l}
\partial_{t}u_{L}-\Delta u_{L}=0,\\
\partial_{t}b_{L}-\Delta b_{L}=0,\\
(u_{L},b_{L})_{|t=0}=(u_{0},b_{0})(x), \quad x\in \mathbb{R}^d.
\end{array}\right.\nonumber
\end{align}
Assume that the time $T\in(0,+\infty]$ has been chosen in such a way that
\begin{align}\label{eqb2}
\|(u_{L},b_{L})\|_{L^{2}_{T}(\dot{B}^{\frac{d}{2}}_{2,1})}\leq \frac{1}{4C}
\end{align}
for a constant C to be defined below.

Let $0<R<\frac{1}{4C}$ and $R_{\alpha}:=\|(u_{L},b_{L})\|_{L^{2}_{T}(\dot{B}^{\frac{d}{2}+\alpha}_{2,1})}.$
Let $\mathcal{G}$ be the set of divergence-free vector fields with coefficients in $ F_{T}^{\frac{d}{2}}\cap F_{T}^{\frac{d}{2}+\alpha}$, and such that
$\|(u,b)\|_{F_{T}^{\frac{d}{2}}}\leq R$ and $\|(u,b)\|_{F_{T}^{\frac{d}{2}+\alpha}}\leq R_{\alpha}$.
Define
\begin{align}
\mathcal{F}(\bar{u},\bar{b}) := \bigg(\int_{0}^{t}e^{(t-\tau)\Delta}\mathcal{P}(b\cdot\nabla b-u\cdot\nabla u)\mathrm{d}\tau,\
\int_{0}^{t}e^{(t-\tau)\Delta}\mathcal{P}(b\cdot\nabla u-u\cdot\nabla b)\mathrm{d}\tau\bigg)\nonumber
\end{align}
with $u=\bar{u}+u_{L}$ and $b=\bar{b}+b_{L}$.
According to Propositions \ref{prop2.2} and \ref{prop2.3}, $\mathcal{F}$ maps $F^{\frac{d}{2}}_{T}\cap F^{\frac{d}{2}+\alpha}_{T}$
into itself, and, for $\beta=0$  or $\alpha$,
we have
\begin{align}
\|\mathcal{F}(\bar{u},\bar{b})\|_{ F^{\frac{d}{2}+\beta}_{T}}
\leq\, &  C\bigg(\|(\bar{u},\bar{b})\|_{F^{\frac{d}{2}}_{T}}+\|(u_{L},b_{L})\|_{L^{2}_{T}(\dot{B}^{\frac{d}{2}}_{2,1})}\bigg)\nonumber\\
& \quad \times \bigg(\|(\bar{u},\bar{b})\|_{F^{\frac{d}{2}+\beta}_{T}}+\|(u_{L},b_{L})\|_{L^{2}_{T}(\dot{B}^{\frac{d}{2}+\beta}_{2,1})}\bigg).\nonumber
\end{align}
Hence it is easy to check that  $\mathcal{F}$ maps $\mathcal{G}$ to $\mathcal{G}$. Similar computations imply that
\begin{align}
\|\mathcal{F}(\bar{u}_{1}-\bar{u}_{2},\bar{b}_{1}-\bar{b}_{2})\|_{F^{\frac{d}{2}}_{T}}\leq & C \Big(\|(\bar{u}_{1},\bar{b}_{1})\|_{F^{\frac{d}{2}}_{T}}
+\|(\bar{u}_{2},\bar{b}_{2})\|_{F^{\frac{d}{2}}_{T}}+\|(\bar{u}_{L},\bar{b}_{L})\|_{L^{2}_{T}(\dot{B}^{\frac{d}{2}}_{T})}\Big)\nonumber\\
&\quad \times\|(\bar{u}_{1}-\bar{u}_{2},\bar{b}_{1}-\bar{b}_{2})\|_{F^{\frac{d}{2}}_{T}},\nonumber\\
\|\mathcal{F}(\bar{u}_{1}-\bar{u}_{2},\bar{b}_{1}-\bar{b}_{2})\|_{F^{\frac{d}{2}+\alpha}_{T}}
\leq & C \Big(\|(\bar{u}_{2},\bar{b}_{2})\|_{F^{\frac{d}{2}+\alpha}_{T}}+\|(\bar{u}_{L},\bar{b}_{L})\|_{L^{2}_{T}(\dot{B}^{\frac{d}{2}+\alpha}_{T})}\Big)\nonumber\\
&\quad \times\|(\bar{u}_{1}-\bar{u}_{2},\bar{b}_{1}-\bar{b}_{2})\|_{F^{\frac{d}{2}}_{T}}\nonumber\\
&+C \Big(\|(\bar{u}_{1},\bar{b}_{1})\|_{F^{\frac{d}{2}}_{T}}+\|(\bar{u}_{L},\bar{b}_{L})\|_{L^{2}_{T}(\dot{B}^{\frac{d}{2}}_{T})}\Big)\nonumber\\
&\quad \times\|(\bar{u}_{1}-\bar{u}_{2},\bar{b}_{1}-\bar{b}_{2})\|_{F^{\frac{d}{2}}_{T}+\alpha}.\nonumber
\end{align}
Set $k=\frac{1}{2}+2 R C$ and $K=4R_{\alpha}C$. According to the above inequalities, we have, for all $\eta >0,$
\begin{align}
&\!\!\!\!\!\!\!\!\!\!\!\!\!\! \!\!\!\!
\|\mathcal{F}(\bar{u}_{1}-\bar{u}_{2},\bar{b}_{1}-\bar{b}_{2})\|_{F^{\frac{d}{2}}_{T}}+
\eta\|\mathcal{F}(\bar{u}_{1}-\bar{u}_{2},\bar{b}_{1}-\bar{b}_{2})\|_{F^{\frac{d}{2}+\alpha}_{T}}\nonumber\\
\leq & C \|(\bar{u}_{1}-\bar{u}_{2},\bar{b}_{1}-\bar{b}_{2})\|_{F^{\frac{d}{2}}_{T}}
 \bigg(\sum_{i=1}^{2}\|(\bar{u}_{i},\bar{b}_{i})\|_{F^{\frac{d}{2}}_{T}}+\|(\bar{u}_{L},\bar{b}_{L})\|_{L^{2}_{T}(\dot{B}^{\frac{d}{2}}_{T})}\nonumber\\
& \qquad +\eta\|(\bar{u}_{2},\bar{b}_{2})\|_{F^{\frac{d}{2}+\alpha}_{T}}+\eta\|(\bar{u}_{L},\bar{b}_{L})\|_{L^{2}_{T}(\dot{B}^{\frac{d}{2}+\alpha}_{T})}\bigg)\nonumber\\
&+C\eta\|(\bar{u}_{1}-\bar{u}_{2},\bar{b}_{1}-\bar{b}_{2})\|_{F^{\frac{d}{2}+\alpha}_{T}}
(\|(\bar{u}_{1},\bar{b}_{1})\|_{F^{\frac{d}{2}}_{T}}+\|(\bar{u}_{L},\bar{b}_{L})\|_{L^{2}_{T}(\dot{B}^{\frac{d}{2}}_{T})})\nonumber\\
\leq\, & C\big(\|(\bar{u}_{1}-\bar{u}_{2},\bar{b}_{1}-\bar{b}_{2})\|_{F^{\frac{d}{2}}_{T}}+
\eta\|(\bar{u}_{1}-\bar{u}_{2},\bar{b}_{1}-\bar{b}_{2})\|_{F^{\frac{d}{2}+\alpha}_{T}}\big)\nonumber\\
&\times \bigg(\sum_{i=1}^{2}\|(\bar{u}_{i},\bar{b}_{i})\|_{F^{\frac{d}{2}}_{T}}+\|(\bar{u}_{L},\bar{b}_{L})\|_{L^{2}_{T}(\dot{B}^{\frac{d}{2}}_{T})} \nonumber\\
&\qquad  +
\eta\|(\bar{u}_{2},\bar{b}_{2})\|_{F^{\frac{d}{2}+\alpha}_{T}}+\eta\|(\bar{u}_{L},\bar{b}_{L})\|_{L^{2}_{T}(\dot{B}^{\frac{d}{2}+\alpha}_{T})}\bigg)\nonumber\\
\leq &(k+\eta K)\Big(\|(\bar{u}_{1}-\bar{u}_{2},\bar{b}_{1}-\bar{b}_{2})\|_{F^{\frac{d}{2}}_{T}}+
\eta\|(\bar{u}_{1}-\bar{u}_{2},\bar{b}_{1}-\bar{b}_{2})\|_{F^{\frac{d}{2}+\alpha}_{T}}\Big).\nonumber
\end{align}
Choosing $\eta$   and $R$  sufficiently small such that $k+\eta K<1$,
we conclude that $\mathcal{F}$  is a contraction map on $\mathcal{G}$ endowed with the norm
$\|\cdot\|_{F^{\frac{d}{2}}_{T}}+\eta\|\cdot\|_{F^{\frac{d}{2}+\alpha}_{T}}.$
Denoting
$$u=\bar{u}+u_{L},\quad b=\bar{b}+b_{L},$$
where $(\bar{u},\bar{b})$
is the unique point of $\mathcal{F}$ in $\mathcal{G}$, we easily find that $(u,b)$ solves the problem \eqref{eqb1}.

Now, according to Proposition 2.3 in \cite{DA}, we have, for the two constants $c$ and $C$ depending only on $d$,
$$\|u_{L}\|_{L^{2}_{T}(\dot{B}^{\frac{d}{2}}_{2,1})}
\leq C \bigg(\sum_{j\in \mathbb{{Z}}}2^{j(\frac{d}{2}-1)}\big(1-e^{-c2^{2j}T}\big)^{\frac{1}{2}}\bigg)\|\dot{\Delta}_{j}u_{0}\|_{L^{2}},$$
$$\|b_{L}\|_{L^{2}_{T}(\dot{B}^{\frac{d}{2}}_{2,1})}
\leq C \bigg(\sum_{j\in \mathbb{Z}}2^{j(\frac{d}{2}-1)}\big(1-e^{-c2^{2j}T}\big)^{\frac{1}{2}}\bigg)\|\dot{\Delta}_{j}b_{0}\|_{L^{2}}.$$
Thanks to the Lebesgue's dominated convergence theorem, the right-hand sides on the above equalities tend to zero as $T$ tends to zero.
Combining  this with \eqref{eqb2}  gives us a bound from
below for the life span of $(u,b)$. The uniqueness of solution can be proved in a standard way.
\end{proof}

Below we give a priori estimates for the following initial value problem:
\begin{equation}\label{eqb3}
\left\{
\begin{array}{l}
\partial_{t} w-\mu\Delta w+A\cdot\nabla w+w\cdot\nabla A-B\cdot\nabla E-E\cdot\nabla B=f,\\
\partial_{t} B-\nu\Delta B+A\cdot\nabla B-B\cdot\nabla A+w\cdot\nabla E-E\cdot\nabla w=g,\\
(w,B)|_{t=0}=(w_{0},B_{0})(x), \quad x\in \mathbb{R}^d.
\end{array}
\right.
\end{equation}

\begin{Proposition}\label{AProp2}
Let $s\in(-\frac{d}{2},\frac{d}{2}].$
Assume that $w_0,B_0\in \dot{B}^{s}_{2,1}$,
$f,g\in L^1_T(  \dot{B}^{s}_{2,1})$, and
$A,E\in L^1_T( \dot{B}^{\frac{d}{2}+1}_{2,1})$ are time-dependent vector fields.
Then there exists a universal constant $\kappa$, and a constant C depending only on $d$
and $s$, such that, for all $t\in [0,T]$,
\begin{align}
\|(w,B)\|_{\tilde{L}^{\infty}_{t}(\dot{B}^{s}_{2,1})}+\kappa\underline{\nu}\|(w,B)\|_{L^{1}_{t}(\dot{B}^{s+2}_{2,1})}
\leq\, &\left(\|(w_{0},B_{0})\|_{\dot{B}^{s}_{2,1}}+\|(f,g)\|_{L^{1}_{t}(\dot{B}^{s}_{2,1})}\right)\nonumber\\
&  \times\exp\bigg(C\int_{0}^{t}(\|\nabla A\|_{\dot{B}^{\frac{d}{2}}_{2,1}}+\|\nabla E\|_{\dot{B}^{\frac{d}{2}}_{2,1}})\mathrm{d}\tau\bigg)\nonumber
\end{align}
with $\underline{\nu}:=\min\{\mu,\nu\}$.
\end{Proposition}

\begin{proof}
The desired estimate will be obtained after localizing the equations \eqref{eqb3} by means  of the homogeneous Littlewood-Paley decomposition.
More precisely, applying $\dot{\Delta}_{j}$ to \eqref{eqb3} yields

\begin{equation}\nonumber
\left\{
\begin{array}{l}
 \partial_{t} w_{j}-\mu\Delta w_{j}+A\cdot\nabla w_{j}-E\cdot\nabla B_{j}\nonumber\\
  \qquad\qquad =f_{j}-\dot{\Delta}_{j}(w\cdot\nabla A)+\dot{\Delta}_{j}(B\cdot\nabla E)
+R_{j}^{1}-R_{j}^{2},\\
\partial_{t} B_{j}-\nu\Delta B_{j}+A\cdot\nabla B_{j}-E\cdot\nabla w_{j}\nonumber\\
  \qquad\qquad =g_{j}+\dot{\Delta}_{j}(B\cdot\nabla A)-\dot{\Delta}_{j}(w\cdot\nabla E)
+R_{j}^{3}-R_{j}^{4},
\end{array}
\right.
\end{equation}
with
\begin{align*}
&  w_{j}:=\dot{\Delta}_{j}w,\quad \qquad \qquad  \ \ \quad B_{j}:=\dot{\Delta}_{j}B,\\
& R_{j}^{1}:=\sum_{k}[A^{k},\dot{\Delta}_{j}]\partial_{k}w,\qquad R_{j}^{2}:=\sum_{k}[E^{k},\dot{\Delta}_{j}]\partial_{k}B,\\
& R_{j}^{3}:=\sum_{k}[A^{k},\dot{\Delta}_{j}]\partial_{k}B,\qquad R_{j}^{4}:=\sum_{k}[E^{k},\dot{\Delta}_{j}]\partial_{k}w.
\end{align*}
Taking the $L^{2}$ inner product of the above equations with $w_{j}$ and $ B_{j},$ respectively, we easily get
\begin{align}
&\!\!\!\!\!\! \!\!\! \!\!\!   \frac{1}{2}\frac{\mathrm{d}}{\mathrm{dt}}\big(\|w_{j}\|^{2}_{L^{2}}+\|B_{j}\|^{2}_{L^{2}}\big)
+\mu\int|\nabla w_{j}|^{2}\mathrm{d}x+\nu\int|\nabla B_{j}|^{2}\mathrm{d}x\nonumber\\
=\,&\frac{1}{2}\int(\dv A)(|w_{j}|^{2}+|B_{j}|^{2})\mathrm{d}x+\int f_{j}w_{j}\mathrm{d}x+\int g_{j}B_{j}\mathrm{d}x-\int\dv E(B\cdot w)\mathrm{d}x\nonumber\\
&-\int\dot{\Delta}_{j}(w\cdot\nabla A)w_{j}\mathrm{d}x+\int\dot{\Delta}_{j}(B\cdot\nabla E)w_{j}\mathrm{d}x-\int\dot{\Delta}_{j}(w\cdot\nabla E)B_{j}\mathrm{d}x
\nonumber\\
&+\int\dot{\Delta}_{j}(B\cdot\nabla A)B_{j}\mathrm{d}x
+\int(R_{j}^{1}-R_{j}^{2})w_{j}\mathrm{d}x+\int(R_{j}^{3}-R_{j}^{4})w_{j}\mathrm{d}x.\nonumber
\end{align}
Hence, thanks to the Bernstein's inequality, we get, for some universal constant $\kappa,$
\begin{align}\label{eqb4}
&\!\!\!\!\!\! \!\!\! \!\!\! \!\!\frac{1}{2}\frac{\mathrm{d}}{\mathrm{dt}}\big(\|w_{j}\|^{2}_{L^{2}}+\|B_{j}\|^{2}_{L^{2}}\big)+\kappa\underline{\nu}2^{2j}
\big(\|w_{j}\|^{2}_{L^{2}}+\|B_{j}\|^{2}_{L^{2}}\big)\nonumber\\
\leq &
\big(\|f_{j}\|_{L^{2}}+\|\dv A\|_{L^{\infty}}\|w_{j}\|_{L^{2}}\nonumber\\
&+\|\dot{\Delta}_{j}(w\cdot\nabla A)\|_{L^{2}}+\|\dot{\Delta}_{j}(B\cdot\nabla E)\|_{L^{2}}
+\|R_{j}^{1}\|_{L^{2}}+\|R_{j}^{2}\|_{L^{2}}\big)\|w_{j}\|^{2}_{L^{2}}
\nonumber\\
&+\Big(\|g_{j}\|_{L^{2}}+\|\dv A\|_{L^{\infty}}\|B_{j}\|_{L^{2}}
+\|\dot{\Delta}_{j}(w\cdot\nabla E)\|_{L^{2}}\nonumber\\
& \quad +\|\dot{\Delta}_{j}(B\cdot\nabla A)\|_{L^{2}}
+\|R_{j}^{3}\|_{L^{2}}+\|R_{j}^{4}\|_{L^{2}}\Big)\|B_{j}\|^{2}_{L^{2}}\nonumber\\
&
+\|\dv E\|_{L^{\infty}}\|w_{j}\|_{L^{2}}\|B_{j}\|_{L^{2}}.
\end{align}
By   Propositions \ref{prop2.2} and \ref{prop2.3}  and   the commutator estimates in \cite{BCD}, we have the following estimates:
\begin{align*}
 \|\dot{\Delta}_{j}(w\cdot\nabla A)\|_{L^{2}}&\leq C c_{j}2^{-js}\|\nabla A\|_{\dot{B}^{\frac{d}{2}}_{2,1}}\|w\|_{\dot{B}^{s}_{2,1}},\\
 \|\dot{\Delta}_{j}(B\cdot\nabla E)\|_{L^{2}}&\leq C c_{j}2^{-js}\|\nabla E\|_{\dot{B}^{\frac{d}{2}}_{2,1}}\|B\|_{\dot{B}^{s}_{2,1}},\\
 \|\dot{\Delta}_{j}(B\cdot\nabla A)\|_{L^{2}}&\leq C c_{j}2^{-js}\|\nabla A\|_{\dot{B}^{\frac{d}{2}}_{2,1}}\|B\|_{\dot{B}^{s}_{2,1}},\\
 \|\dot{\Delta}_{j}(w\cdot\nabla E)\|_{L^{2}}&\leq C c_{j}2^{-js}\|\nabla E\|_{\dot{B}^{\frac{d}{2}}_{2,1}}\|w\|_{\dot{B}^{s}_{2,1}},\\
 \|R_{j}^{1}\|_{L^{2}}&\leq C c_{j}2^{-js}\|\nabla A\|_{\dot{B}^{\frac{d}{2}}_{2,1}}\|w\|_{\dot{B}^{s}_{2,1}},\nonumber\\
 \|R_{j}^{2}\|_{L^{2}}&\leq C c_{j}2^{-js}\|\nabla E\|_{\dot{B}^{\frac{d}{2}}_{2,1}}\|B\|_{\dot{B}^{s}_{2,1}},\nonumber\\
 \|R_{j}^{3}\|_{L^{2}}&\leq C c_{j}2^{-js}\|\nabla A\|_{\dot{B}^{\frac{d}{2}}_{2,1}}\|B\|_{\dot{B}^{s}_{2,1}},\nonumber\\
 \|R_{j}^{4}\|_{L^{2}}&\leq C c_{j}2^{-js}\|\nabla E\|_{\dot{B}^{\frac{d}{2}}_{2,1}}\|w\|_{\dot{B}^{s}_{2,1}},\nonumber
\end{align*}
where $(c_{j})_{j\in \mathbb{Z}}$ denotes a positive sequence such that $\sum_{j\in \mathbb{Z}}c_{j}=1$.

Formally dividing both sides of the inequality \eqref{eqb4} by $\|w_{j}\|_{L^{2}}+\|B_{j}\|_{L^{2}}$ and integrating over $[0,t]$  yields
\begin{align}\label{eqb5}
&\!\!\!\!\!\!\!\!\!\!\!\!\!\!\!\!\!\!\!
\!\!
\|w_{j}(t)\|_{L^{2}}+\|B_{j}(t)\|_{L^{2}}+\underline{\nu}2^{2j}\int_{0}^{t}(\|w_{j}(\tau)\|_{L^{2}}+\|B_{j}(\tau)\|_{L^{2}})\mathrm{d}\tau\nonumber\\
\leq\,&\|w_{j}(0)\|_{L^{2}}+\|B_{j}(0)\|_{L^{2}}+\int_{0}^{t}(\|f_{j}(\tau)\|_{L^{2}}+\|g_{j}(\tau)\|_{L^{2}})\mathrm{d}\tau\nonumber\\
&+C\int_{0}^{t}\big(\|\nabla A\|_{\dot{B}^{\frac{d}{2}}_{2,1}}+\|\nabla E\|_{\dot{B}^{\frac{d}{2}}_{2,1}}\big)
\big(\|w_{j}(\tau)\|_{L^{2}}+\|B_{j}(\tau)\|_{L^{2}}\big)\mathrm{d}\tau\nonumber\\
&+2^{-js}C\int_{0}^{t}\big(\|\nabla A\|_{\dot{B}^{\frac{d}{2}}_{2,1}}+\|\nabla E\|_{\dot{B}^{\frac{d}{2}}_{2,1}}\big)
\big(\|w\|_{\dot{B}^{s}_{2,1}}+\|B\|_{\dot{B}^{s}_{2,1}}\big)\mathrm{d}\tau.
\end{align}
Now, multiplying the both sides of \eqref{eqb5} by $2^{js}$ and summing over $j$, we end up with
\begin{align}
\|w\|_{\tilde{L}^{\infty}_{t}(\dot{B}^{s}_{2,1})}&+\|B\|_{\tilde{L}^{\infty}_{t}(\dot{B}^{s}_{2,1})}
+\kappa\underline{\nu}\|w\|_{L^{1}_{t}(\dot{B}^{s+2}_{2,1})}+\kappa\underline{\nu}\|B\|_{L^{1}_{t}(\dot{B}^{s+2}_{2,1})}\nonumber\\
\leq \, &\|w_{0}\|_{\dot{B}^{s}_{2,1}}+\|B_{0}\|_{\dot{B}^{s}_{2,1}}+\|f\|_{L^{1}_{t}(\dot{B}^{s}_{2,1})}+\|g\|_{L^{1}_{t}(\dot{B}^{s}_{2,1})}\nonumber\\
&+C\int_{0}^{t}\big(\|\nabla A\|_{\dot{B}^{\frac{d}{2}}_{2,1}}+\|\nabla E\|_{\dot{B}^{\frac{d}{2}}_{2,1}}\big)
\big(\|w\|_{\dot{B}^{s}_{2,1}}+\|B\|_{\dot{B}^{s}_{2,1}}\big)\mathrm{d}\tau\nonumber
\end{align}
for some constant $C$ depending only on $d$ and $s$. Applying Gronwall's lemma then completes the proof.
\end{proof}


If $(w,B)$ solve the following systems
\begin{equation}\label{eqb6}
\left\{
\begin{array}{l}
\partial_{t} w-\mu\Delta w+\mathcal{P}(A\cdot\nabla w)+\mathcal{P}(w\cdot\nabla A)\\
\qquad \qquad \qquad \qquad -\mathcal{P}(B\cdot\nabla E)-\mathcal{P}(E\cdot\nabla B)=\mathcal{P}f,\\
\partial_{t} B-\nu\Delta B+\mathcal{P}(A\cdot\nabla B)-\mathcal{P}(B\cdot\nabla A)\\
\qquad \qquad \qquad \qquad +\mathcal{P}(w\cdot\nabla E)-\mathcal{P}(E\cdot\nabla w)=\mathcal{P}g, \\
\dv w=0,\quad \dv B=0,\\
(w,B)_{|t=0}=(w_{0},B_{0})(x), \quad x\in\mathbb{R}^d,
\end{array}
\right.
\end{equation}
we have

\begin{Proposition}\label{AProp3}
Let $s\in (-\frac{d}{2},\frac{d}{2}].$
Let $s\in(-\frac{d}{2},\frac{d}{2}].$
Assume that $w_0,B_0\in \dot{B}^{s}_{2,1}$ with $\dv w_0=\dv B   =0$, $f,g\in L^1_T( \dot{B}^{s}_{2,1})$, and
$A,E\in L^1_T( \dot{B}^{\frac{d}{2}+1}_{2,1})$ are time-dependent vector fields.
Then there exists a universal constant $\kappa$, and a constant C depending only on $d$
and $s$, such that, for all $t\in [0,T]$,
\begin{align}
&\!\!\!\!\!\!\!\!\!\!\!\!\|(w,B)\|_{\tilde{L}^{\infty}_{t}(\dot{B}^{s}_{2,1})}+\kappa\underline{\nu}\|(w,B)\|_{L^{1}_{t}(\dot{B}^{s+2}_{2,1})}\nonumber\\
\leq\, &\big(\|(w_{0},B_{0})\|_{\dot{B}^{s}_{2,1}}+\|(f,g)\|_{L^{1}_{t}(\dot{B}^{s}_{2,1})}\big)
 \exp\bigg\{C\int_{0}^{t}\big(\|\nabla A\|_{\dot{B}^{\frac{d}{2}}_{2,1}}+\|\nabla E\|_{\dot{B}^{\frac{d}{2}}_{2,1}}\big)\mathrm{d}\tau\bigg\}\nonumber
\end{align}
with $\underline{\nu}:=\min\{\mu,\nu\}$.
\end{Proposition}

\begin{proof}
The proof  is similar to that of Proposition \ref{AProp2}. The evolution equations for $(w_{j},B_{j}):=(\dot{\Delta}_{j}w,\dot{\Delta}_{j}B)$
now read
\begin{equation}\label{eqb7}\left\{
\begin{array}{l}
\partial_{t} w_{j}-\mu\Delta w_{j}+\mathcal{P}(A\cdot\nabla w_{j})-\mathcal{P}(E\cdot\nabla B_{j})\\
\qquad\qquad  =\mathcal{P}f_{j}-\dot{\Delta}_{j}\mathcal{P}(w\cdot\nabla A)+\dot{\Delta}_{j}\mathcal{P}(B\cdot\nabla E)+\mathcal{P}R_{j}^{1}-\mathcal{P}R_{j}^{2}£¬\\
\partial_{t} B_{j}-\nu\Delta B_{j}+\mathcal{P}A\cdot\nabla B_{j}-\mathcal{P}E\cdot\nabla w_{j}\\
\qquad \qquad =\mathcal{P}g_{j}+\dot{\Delta}_{j}\mathcal{P}(B\cdot\nabla A)-\dot{\Delta}_{j}\mathcal{P}(w\cdot\nabla E)+\mathcal{P}R_{j}^{3}-\mathcal{P}R_{j}^{4}¡£
\end{array}
\right.\nonumber
\end{equation}
Since $\dv w_{j}=0$ and $\dv H_{j}=0$, we can deduce that
$$\int h\cdot w_{j}\mathrm{d}x=\int \mathcal{P}h\cdot w_{j}\mathrm{d}x$$ for any $h\in L^{2}(\mathbb{R}^{d})$.
Taking the $L^{2}$ inner product for the equations in \eqref{eqb7} with $w_{j}$ and $B_{j}$ respectively,
the operator $\mathcal{P}$ may be ``omitted" in the computations so that
by proceeding along the lines of the proof of Proposition \ref{AProp2}, we get the desired inequality.
\end{proof}

\begin{Remark}
In the case of $d=2$, the  existence  time $T$ in  Propositions \ref{imhd}, \ref{AProp2}, and  \ref{AProp3} may take $+\infty$.
Since we  mainly study the local solution, we shall not discuss the details here.
\end{Remark}

\bigskip
\section*{Acknowledgements}

F.-C. Li is supported in part by NSFC  Grant Nos. 11271184, 11671193  and
   PAPD.
Y.-M. Mu is supported by the Doctoral Starting up Foundation of Nanjing University of Finance \& Economics  Grant No. MYMXW16001.
D. Wang's research is supported in part by the NSF Grant DMS-1312800 and NSFC Grant No. 11328102.

 \bigskip

\medskip

Received  xxxx  20xx; revised  xxxx  20xx.

\medskip

\end{document}